\documentclass[a4paper]{article}

\usepackage{blindtext}
\usepackage[utf8]{inputenc}
\usepackage{amsmath}
\usepackage{amsthm}
\usepackage{amssymb}
\usepackage{csquotes}
\usepackage[toc,page]{appendix}

\usepackage[
backend=biber,
style=alphabetic,
sorting=nyt,
giveninits=true
]{biblatex}
\usepackage[english]{babel}
\usepackage{amsfonts}
\usepackage{tikz-cd}
\usepackage{tikz}
\usepackage{tikz}
\usetikzlibrary{matrix}
\usepackage[T1]{fontenc}
\usepackage{float}
\usepackage{graphicx}
\usepackage{caption}
\usepackage{subcaption}
\usepackage{diagbox}
\usepackage{url}
\usepackage{mathrsfs}
\usepackage{bm}
\usepackage[left=2.5cm,right=2.5cm, top=2.5cm,bottom=2.5cm]{geometry}
\usepackage[inline]{enumitem}

\usepackage{hyperref}
\hypersetup{
unicode=true
}

\newcommand{\nocontentsline}[3]{}
\let\origcontentsline\addcontentsline
\newcommand\stoptoc{\let\addcontentsline\nocontentsline}
\newcommand\resumetoc{\let\addcontentsline\origcontentsline}

\newcommand*{\QEDB}{\hfill\ensuremath{\square}}

\numberwithin{equation}{subsection}
\newtheorem*{thrm*}{Theorem}
\newtheorem{thrm}[equation]{Theorem}
\newtheorem{prop}[equation]{Proposition}
\newtheorem{lemma}[equation]{Lemma}

\newtheorem{cor}[equation]{Corollary}

\theoremstyle{definition}
\newtheorem*{dfn*}{Definition}
\newtheorem{dfn}[equation]{Definition}
\newtheorem{dfn_prop}[equation]{Definition-Proposition}
\newtheorem{notation}[equation]{Notation}

\newtheorem{rmrk}[equation]{Remark}
\newtheorem{exmpl}[equation]{Example}

\makeatletter
\newcommand\level[1]{%
  \ifcase#1\relax\expandafter\chapter\or
    \expandafter\section\or
    \expandafter\subsection\or
    \expandafter\subsubsection\else
    \def\next{\@level{#1}}\expandafter\next
  \fi}
\newcommand{\@level}[1]{%
  \@startsection{level#1}
    {#1}
    {\z@}%
    {-3.25ex\@plus -1ex \@minus -.2ex}%
    {1.5ex \@plus .2ex}%
    {\normalfont\normalsize\bfseries}}

\newcounter{level4}[subsubsection]
\@namedef{thelevel4}{\thesubsubsection.\arabic{level4}}
\@namedef{level4mark}#1{}
\count@=4
\loop\ifnum\count@<100
  \begingroup\edef\x{\endgroup
    \noexpand\newcounter{level\number\numexpr\count@+1\relax}[level\number\count@]
    \noexpand\@namedef{thelevel\number\numexpr\count@+1\relax}{%
      \noexpand\@nameuse{thelevel\number\count@}.\noexpand\arabic{level\number\numexpr\count@+1\relax}}
    \noexpand\@namedef{level\number\numexpr\count@+1\relax mark}####1{}}
  \x
  \advance\count@\@ne
\repeat
\makeatother
\newcommand\DMO[2]
        {\DeclareMathOperator{#1}{#2}}
\DMO{\Ker}{Ker}
\DMO{\Coker}{Coker}
\DMO{\lar}{\longrightarrow}
\DMO{\lal}{\longleftarrow}
\DMO{\hocolim}{hocolim}
\DMO{\holim}{holim}
\DMO{\coeq}{coeq}
\DMO{\Eq}{eq}
\DMO{\sBar}{Bar}
\DMO{\coker}{Coker}
\DMO{\Ab}{Ab}

\DeclareMathOperator{\ari}{\hookrightarrow}

\DeclareMathOperator{\sm}{\setminus}

\DeclareMathOperator{\eps}{\varepsilon}
\DeclareMathOperator{\PSh}{\mathbf{PSh}}
\DeclareMathOperator{\Top}{\mathbf{Top}}
\DeclareMathOperator{\id}{id}
\DeclareMathOperator{\op}{op}
\DeclareMathOperator{\Mu}{\mathcal{M}}
\DeclareMathOperator{\rk}{rk}
\DeclareMathOperator{\codim}{codim}
\DeclareMathOperator{\im}{im}
\DeclareMathOperator{\Mod}{Mod}
\DeclareMathOperator{\Atoms}{Atoms}
\DeclareMathOperator{\MVHC}{MVHC}
\DeclareMathOperator{\MVHD}{MVHD}
\DeclareMathOperator{\MV}{MV^\circ}
\DeclareMathOperator{\Gr}{Gr}
\DeclareMathOperator{\Dec}{Dec}
\DeclareMathOperator{\Tot}{Tot}
\DeclareMathOperator{\OS}{OS}
\DeclareMathOperator{\bbk}{\Bbbk}
\DeclareMathOperator{\Ch}{\mathbf{Ch}}
\DeclareMathOperator{\Cat}{Cat}

\DeclareMathOperator{\zgqis}{\overset{\sim}{\leftrightsquigarrow}}

\DeclareMathOperator{\Hom}{Hom}
\DeclareMathOperator{\Ho}{Ho}
\DeclareMathOperator{\Vect}{\mathbf{Vect}}
\DeclareMathOperator{\Mon}{\mathbf{Mon}}
\DeclareMathOperator{\CMon}{\mathbf{CMon}}
\DeclareMathOperator{\CDGA}{\mathbf{CDGA}}
\DeclareMathOperator{\DGA}{\mathbf{DGA}}
\DeclareMathOperator{\NCD}{\mathbf{NCD}}
\DeclareMathOperator{\cs}{\mathbf{cs}}

\DeclareMathOperator{\Sh}{\mathbf{Sh}}

\DeclareMathOperator{\Sm}{\mathbf{Sm}}
\DeclareMathOperator{\PL}{PL}

\DeclareMathOperator{\dR}{dR}
\DeclareMathOperator{\Cech}{\text{\v{C}ech}}
\DeclareMathOperator{\MHC}{\mathbf{MHC}}
\DMO{\MHS}{\mathbf{MHS}}
\DMO{\Th}{th}
\DeclareMathOperator{\MHD}{\mathbf{MHD}}
\DeclareMathOperator{\ChF}{\boldsymbol{\mathcal{C}}\boldsymbol{\mathit{h}}}
\DeclareMathOperator{\Fun}{Fun}

\DeclareMathOperator{\chr}{char}
\DeclareMathOperator{\Var}{{\mathbf{Var}_{sm}}}

\DeclareMathOperator{\AP}{AP}
\DeclareMathOperator{\MVC}{MV^{\check{C}}}
\DeclareMathOperator{\MVCC}{MV^{\check{C}}_{com}}

\makeatletter
\DMO{\colim}{colim}
\DMO{\prlim}{lim}
\newcommand\mb[1]
        {\mathbb{#1}}
\newcommand\os[2]
        {\overset{#1}{#2}}

\newcommand\ul[1]
        {\underline{#1}}
\newcommand\ol[1]
        {\overline{#1}}
\newcommand\mc[1]
        {\mathcal{#1}}

\newcommand{\scup}{\mbox{ }\cup\mbox{ }}
\newcommand{\ssqcup}{\mbox{ }\sqcup\mbox{ }}

\newcommand{\ar}[1][]
{\xrightarrow{#1}}
\newcommand{\al}[1][]
{\xleftarrow{#1}}

\newcommand{\lto}{\mathrel{\leadsto}}
\newcommand\summaryname{Abstract}
    {\small\begin{center}%
    \bfseries{\summaryname} \end{center}}

\newcommand{\RNum}[1]{\uppercase\expandafter{\romannumeral #1\relax}}

\title{Rational Homotopy Type Of Complements Of Submanifold Arrangements}
\author{Alexander Zakharov}
\date{}
\begin{document}
\linespread{1.0}\selectfont 

\thispagestyle{plain}%
\setcounter{page}{1}
\maketitle
\begin{abstract}
In this work we provide an explicit cdga that controls
the rational homotopy type of the complement $X-\cup_i Z_i$, 
	where $X$ is a smooth compact algebraic variety and 
	$\{Z_i\}$ is a collection of subvarieties such that all 
	set-theoretical intersections are smooth.
The model is given in terms of the cohomology 
	of all intersections of $Z_i$'s, and the natural 
	maps induced by the inclusions.	
Our construction is inspired by the work of J.~Morgan, who covered the fundamental case where $\{Z_i\}$ is a divisor with 
normal crossings, and it is built on developments of the theory 
	of mixed Hodge diagrams by Cirici-Horel. 
We avoid any explicit reduction to the normal crossings divisor case, e.g. via
the wonderful compactification of De Concini-Procesi.
As an application of our approach we recover and generalize 
a few separate results on the complements of arrangements
	in a uniform manner. 
These include the 
Kritz-Totaro model for graph configuration spaces, 
Yuzvinsky's model for 
affine subspace arrangements and 
Dupont's model for complements of hypersurfaces with 
	hyperplane-like intersection.
\end{abstract}


\section*{Introduction}
\addcontentsline{toc}{section}{Introduction}
\subsection{Main results}\label{main:intro}

\begin{dfn*}\label{dfn:arrangement} A \textit{(smooth compact algebraic) arrangement} $(X,Z_*)$
is a smooth compact algebraic variety $X$ over $\mb{C}$
and a collection $\{Z_i\subset X,i\in N\}$ of subvarieties parametrized by 
a finite set $N$ such that the set-theoretical 
intersections 
$Z_I=\bigcap\limits_{i\in I}Z_i,I\subset N$ and $Z_\emptyset=X$
are smooth $C^\infty$-submanifolds. (In other words, the corresponding 
scheme-theoretic intersections with the reduced scheme structure are smooth subschemes.)
The complement of the arrangement $(X,Z_*)$ is the 
variety $U=X-\cup_i Z_i$.
\end{dfn*}
Let $A_{\PL}(-)\colon \Top\to \CDGA_{\mb{Q}}$ denote the Sullivan functor (e.g.\ {\cite[\S 10]{FHT}}).
A (cdga) {\it model} of $U$ is a cdga $B^\bullet$ quasi-isomorphic to 
$A_{\PL}(U)$ in 
$\CDGA_{\mb{Q}}$, meaning there is a zig-zag of quasi-isomorphisms $B^\bullet\zgqis A_{\PL}(U)$.
In this paper, by the rational, respectively complex homotopy type of a space $X$ we mean the quasi-isomorphism type of $A_{\PL}(X)$, respectively of its complex analog.


\begin{dfn*}
	A collection $\{D_i\subset X\}$ is called an \textit{NCD}
	if all $D_i$ are smooth and $\cup_i D_i$ is a 
	divisor with normal crossings.
\end{dfn*}
Consider an arrangement $(X,D_*)$ where $D_*$ is an NCD.
Let $U=X-\cup_i D_i, i\leq n$ and $j\colon U\to X$ be the inclusion.
Denote by $\ul{\mb{Q}}$ the constant sheaf over $X$.
The following theorem 
is due to J. Morgan \cite{M}.
It describes the $\mb{Q}$-homotopy type of the complement of 
an NCD-arangement:
\begin{thrm}[Morgan]\label{thrm:Morgan}
	Let $E^{pq}_2=H^p(X;R^q j_*j^* \ul{\mb{Q}})$ be
	the second sheet of the Leray spectral sequence. 
	Then $E^{**}_2$ is a  model of $U$, i.e.
	there is a quasi-isomorphism 
	$E^{**}_2\zgqis A_{\PL}(U)\in \CDGA_{\mb{Q}}$.
\end{thrm}

%
In particular Theorem \ref{thrm:Morgan} implies that 
$E_3=E_\infty$. 
Note that if $(X,Z_*)$ is a smooth 
arrangement with some $Z_i$ 
of codimension $>1$, then
the Leray spectral sequence does not 
degenerate at $E_3$. An example is
$(\mb{P}^2,pt)$.

In this work we will generalize 
Theorem~\ref{thrm:Morgan} to the case of an arbitrary
smooth arrangement 
$(X,Z_*)$ in a compact algebraic $X$, 
providing a model for the $\mb{Q}$-homotopy type of 
the complement $U=X-\cup_i Z_i$.
The corresponding model turns out to be functorial in the arrangement 
$(X,Z_*)$ (see Definition \ref{dfn:arrangement_posets:morphism}).

To state the results let us introduce some notation.
In order to simplify sign issues
and make the product structure in what follows clear, 
it will be convenient to order the set $N$ and to 
use the monomials
$d\nu_I=d\nu_{i_1}\wedge\ldots\wedge d\nu_{i_k}$ for
$I=\{i_1<\ldots<i_k\}$, viewed as degree $-|I|$ elements of
the Grassmann algebra $\Lambda\langle d\nu_i\mid i\in N\rangle$.
Let $\iota_i$ be the derivation of 
$\Lambda \langle d\nu_i\mid i\in N \rangle$
of degree $1$ 
determined by 
$\iota_i d\nu_j=\delta_{ij}$.
Note that the one dimensional $\mb{Z}$-module $\mb{Z}\cdot d\nu_I$ 
and the derivation 
$\iota_i$ of $\Lambda\langle d\nu_i\mid i\in N\rangle$
are independent of the ordering of $N$.

The default coefficient ring for cohomology groups is 
$\mb{Q}$ and will usually be omitted from notation. 
Let $(X,Z_*)$ be a smooth compact algebraic 
arrangement.
For $i\in I \subset N$, set $H^*_{Z_I}(X)=H^*(X,X-Z_{I})$, and let
$$g_{I-\{i\},I}:H^*_{Z_I}(X)\to H^*_{Z_{I-\{i\}}}(X)$$
denote the natural map induced by the embedding 
$Z_I\ari Z_{I-\{i\}}$.
%
The main result is the following
\begin{thrm}[On the homotopy type of the complement,
	Corollary \ref{thrm:main_cubical}]
\label{thrm:main_cubical_intro}
	There is a second quadrant 
	multiplicative spectral sequence $(E^{pq}_r,d_r)$ with these properties:
	\begin{enumerate}
		\item We have $E^{pq}_1=
		\bigoplus\limits_{|I|=-p}H^{q}_{Z_I}(X)$, and 
		$$E^{**}_1=\bigoplus\limits_{I}H^*_{Z_I}(X)\otimes 
			\mb{Z}\cdot d\nu_I,$$
		is a cdga with the differential 
		$$d_1=\sum\limits^{i,I}_{i\in I} g_{I-\{i\},I}\otimes 
		\iota_i,$$
		and the straightforward multiplication obeying the usual
		sign rule.
		\item The terms $(E_r,d_r)$ are functorial
			in the arrangement $(X,Z_*)$ for $r>0$.

		\item We have $E_2=E_\infty$.

		\item There is a functorial quasi-isomorphism 
			$$E^{**}_1\zgqis A_{\PL}(U)\in \CDGA_{\mb{Q}},$$
			where both sides are considered as 
				functors of $(X,Z_*)$.
	\end{enumerate}
\end{thrm}
See Example \ref{exmpl:thrm_for_quadric} of Appendix~\ref{first_example} 
for an explicit computation using this result.

To clarify the functoriality of the result, 
assume $f\colon X'\ar X$ is a morphism
and $g\colon N\ar N'$ is an injective map such that 
$f^{-1}Z_i\subset Z'_{g(i)}$for all $i\in N$.
This data provides a morphism of arrangements $(X',Z'_*)\ar (X,Z_*)$ (see Definition \ref{dfn:arrangement_posets:morphism}), inducing a natural morphism $E^{**}_1(X,Z_{i\in N})\ar E^{**}_1(X',Z'_{i\in N'})$.
On the other hand $f$ restricts to a map of complements $U'\ar U$.
The zig-zag provided by the Theorem 
commutes with these morphisms.
In particular, if $(X,Z_*)$ admits a group action 
which permutes the arrangement $Z_*$, 
we obtain an expression for 
the action on the rational homotopy type.

In the case of an NCD-complement this spectral sequnce 
appears already in \cite{M} and
is related to the Leray spectral sequence in Theorem \ref{thrm:Morgan} 
via d\'ecalage (see Proposition \ref{prop:decalage_ss}). 

\smallskip

It is natural to look for a similar model for the rational homotopy type 
which would better reflect 
the combinatorics of the arrangement. One way to do this is 
to consider for example the corresponding intersection poset. 
The notion of a locally geometric lattice 
was applied in \cite{JCW} to describe the cohomology ring of some
arrangement complements in terms of the Orlik-Solomon algebra and
the cohomology of smooth intersections. 
We will prove here similar results about the rational homotopy type and without requiring 
the intersections to be clean, i.e.\ to locally look like a 
vector subspace arrangement.
Let $[X,L,\Mu]$ be a smooth compact algebraic 
\textit{multiplicative arrangement}
on $X$ given by a presheaf $(L,\leq)^{\op}\ar 2^X$, 
and suppose $\mc{M}$ is a corresponding \textit{OS-algebra}
see Definitions \ref{dfn:arrangement_poset} and \ref{dfn:os_algebra}.
For simplicity one can assume $(L,\le)$ is a geometric lattice, 
and take $\Mu=\OS(L)$ to be the usual Orlik-Solomon algebra of $(L,\le)$ 
viewed as an $L$-graded object, so that $\Mu_x,x\in L$ is concentrated in 
cohomological degree $-r(x)$.
Let us mention that arrangements form a category 
(Definition \ref{dfn:arrangement_posets:morphism}).

An example of $(L,\leq)$ is the cubical arrangement $\mc{C}$
of all intersections $\mc{C}_I=\cap_{i\in I}Z_i$
as in Theorem \ref{thrm:main_cubical_intro}. 
For $x\in L$ we denote by $L_x$ the value of the presheaf on $x$. Note that $L_0=X$ and $x\leq y\in L$ implies $L_x\supset L_y\subset X$. 
We require that 
$(L,\leq)$ should be graded by a rank function $r$ 
such that $r(0)=0$ and 
$r(x)=r(y)+1$ if $x$ lies over $y$. The last condition is 
denoted by 
$x:>y$.
For each $y<x\in L$ there is a morphism
$g_{yx}:H^*_{L_x}(X)\to H^*_{L_y}(X)$ corresponding to 
the embedding $L_x\hookrightarrow L_y\subset X$, and for each
$y<:x$ we have the \textit{structure map} 
$\partial_{yx}:\mc{M}_x\to \mc{M}_y[1]$.

Set $U=X-\bigcup\limits^{x\in L}_{x>0}L_x$ to be the complement. The following result is similar to Theorem 1.3 in \cite{JCW}, 
and it also appears in the case of a clean hypersurface arrangement
in \cite{Dupont}:
\begin{thrm}[On the cohomology of the complement, theorem
	\ref{thrm:mv_spectral_sequence}]
	\label{thrm:main_coh_intro}
	There is a multiplicative spectral sequence, 
	functorial in the multiplicative arrangement $[X,L,\Mu]$,
	of pure Hodge structures with $E^{pq}_1$ of weight $q$
	with these properties:
	\begin{enumerate}
		\item  We have $$E^{pq}_1=
		\bigoplus\limits^{x\in L}_{r(x)=-p}
			H^q_{L_X}(X)\otimes \mc{M}_x[-r(x)],$$
		and $(E^{**}_1,d_1)$ is a cdga
		with the differential given by
		$$d_1=\sum\limits^{y,x}_{y<:x}g_{yx}\otimes 
				\partial_{yx}$$
		\item We have $E_2=E_\infty$.
		\item There is a natural multiplicative isomorphism 
			of Hodge structures: 
			$E^{pq}_2\simeq 
			\Gr^{\Dec\ W}_q H^{p+q}(U)$.
		\item 
  There is a multiplicative splitting 
  $H^{*}(U)\simeq \Gr^{\Dec W}H^*(U)$ of algebras. 
  In particular, there is an isomorphism of algebras 
			$E^{**}_2\simeq H^*(U)$ which 
			is functorial in
			$[X,L,\Mu]$
	\end{enumerate}
\end{thrm}
\begin{rmrk}
	Here $\Gr^{\Dec\ W}_q H^n(U)$ denotes the usual 
	weight $q$ piece of the mixed Hodge structure on 
	$H^n(U)$, see Remark \ref{rmrk:decalage} for the 
	filtration d\'ecal\'ee $\Dec\ W$.
\end{rmrk}
This spectral sequence is called the \textit{Mayer-Vietoris 
spectral sequence} of the arrangement $[X,L,\Mu]$, and it also exists
in general topological setting (see Theorem \ref{thrm:ss_for_lattices}).
We will refer to the spectral sequence in Theorem \ref{thrm:main_cubical_intro}
as the \textit{cubical} Mayer-Vietoris spectral sequence.

\begin{rmrk}
The above theorem is analogous to the fact that 
the Leray spectral sequence of an open embedding 
has a mixed Hodge structure (see \cite{A}).
Both results agree in the above-mentioned case of an open embedding of the 
complement of an NCD arrangement.
In the general case the weight condition implies that the cdg-algebra $E^{**}_r,r>0$ 
can be considered as a pure Hodge 
complex (a Hodge diagram in fact).
Considering $H^*(U)$ as a mixed Hodge structure one can rephrase the statement of the theorem by saying that 
there is a natural isomorphism of Hodge structures 
$E^{**}_2\simeq \Gr^{\Dec W} H^*(U)$.
\
We emphasize however that $H^*(U)$ itself is 
rarely pure, and one cannot in general recover 
the mixed Hodge structure on $H^*(U)$
from our model.
\end{rmrk}

Assume the arrangement $[X,L,\OS(L)]$ is locally geometric 
(Definition \ref{dfn:arrangement_posets:geometric}). 
Our next result is an analog of 
Theorem \ref{thrm:main_cubical_intro} that gives a more economical model which captures 
the combinatorial structure of the given arrangement $(L,\leq)$.
We state it as an extension of Theorem \ref{thrm:main_coh_intro}.
\begin{thrm}[The lattice model for the homotopy type of the complement, 
	Theorem \ref{thrm:main_lattices}]
	\label{thrm:main_lattices_intro}
	We keep the notation of Theorem~\ref{thrm:main_coh_intro}. For a 
 locally geometric arrangement $[X,L]$, there is a quasi-isomorphism:
	$$E^{**}_1\zgqis A_{\PL}(U)\in \CDGA_{\mb{Q}}$$
	which is functorial in $[X,L]$.
\end{thrm}
We illustrate this result in Example \ref{exmpl:thrm:for_lines} of Appendix~\ref{first_example}.

Since the cubical lattice is a geometric arrangement, 
Theorem \ref{thrm:main_cubical_intro} is a formal corollary of 
Theorem \ref{thrm:main_lattices_intro}.
In the course of the proof we will do the opposite and deduce 
Theorem \ref{thrm:main_lattices_intro} from 
Theorem \ref{thrm:main_cubical_intro}.

\smallskip

The proof is based on the notion of a mixed Hodge diagram 
introduced by Morgan in \cite{M} and then developed in 
a functorial form by Navarro Aznar in \cite{Aznar}. 
It is the cdga version of the mixed Hodge complex invented by 
Deligne \cite{D}.

A non-trivial property of a mixed Hodge diagram $(K,F,W)$ 
is that it 
produces a mixed Hodge 
\textit{structure} $(M,F,\Dec W)$ on the minimal model $M$ of 
the cdga $K$ \cite{M},\cite{JC}. 
In particular its weight
filtration admits 
a natural Deligne splitting (Proposition \ref{prop:deligne_splitting})
over $\mb{C}$, thus providing a natural
quasi-isomorphism  $E_1(K_{\mb{C}},W)\zgqis E_0(K_{\mb{C}},\Dec W)$ with 
$E_0(M,\Dec W)\otimes \mb{C}\simeq M\otimes\mb{C}$.
By a general result due to Morgan-Sullivan (e.g.\ \cite[Theorem 10.1]{M}) 
we have a non-canonical quasi-isomorphism
$E_1(K_{\mb{Q}},W)\zgqis K_{\mb{Q}}$.

It is possible however to avoid using minimal models and 
upgrade the statement to a lax symmetric 
monoidal quasi-isomorphism
of functors $(K,F,W)\ar E_1(K_{\mb{Q}},W)$ and $(K,F,W)\ar K_{\mb{Q}}$
from $(\MHC,\otimes)$ to $(\Ch(\Vect_{\mb{Q}}),\otimes)$.
We will adopt this approach due to Cirici and Horel \cite{CH20}, see
Theorem \ref{thrm:mhd:E_1_qis}.

Let $j\colon U\to X$ be the inclusion of the complement, let 
$\ul{\mb{Q}}_X$ denote the constant sheaf over $X$ and let $[X,L,\Mu]$
be a multiplicative arrangement (e.g.\ a cubical one).
The sheaf $Rj_*j^*\ul{\mb{Q}}_X$ has a natural 
resolution $\MV(L)\in \Mon(\Ch^{\le 0}(\Sh_X))$ in terms of the arrangement, 
called the Mayer-Vietoris 
complex (Definition \ref{dfn:mv_complex}). The Mayer-Vietoris complex 
is equipped with 
the \emph{rank filtration} $W'$, such that $E^{pq}_1(\MV(L),W')$
is described by Theorem \ref{thrm:main_lattices_intro}.
Though $E^{pq}_1(\MV(L),W')$
is a cdga, $\MV(L)$ is non-commutative, and it does not admit a 
structure of a Hodge complex on the nose.

The initial aim of this work 
is to construct a Hodge diagram $(\MVHD,F,W')$ based on this resolution
and such that $E^{pq}_1(\MVHD,W')=E^{pq}_1(\MV(L),W')$. 
Once this is done, Theorem \ref{thrm:main_lattices_intro} follows from Theorem \ref{thrm:mhd:E_1_qis}.

\smallskip

Let us mention some other approaches and related difficulties.
First of all, there is no available analog for 
the log-de Rham complex for 
an inclusion $X-\cup Z_i\to X$ when $\codim Z_i/X>1$.
On the other hand, Dupont \cite{Dupont} defined a version of the
log-de Rham complex 
$\Omega^\bullet\langle D_i\rangle\subset \Omega^\bullet_X$
in the case $\codim D_i/X=1$ and a hypersurface arrangement $D_i$
is clean but not neccessary an NCD, i.e. 
$\bigcup D_i\subset X$ locally looks like an 
arrangement of hyperplanes. 

In other concrete examples of interest, such as configuration spaces  
$U=F(M,n)\subset M^n$ (see \cite{FM}, \cite{K}), 
there is a direct approach which is to construct a compactification 
$\ol{U}$ of $U$ with an NCD complement 
$\ol{U}-U=\cup D_i$ explicitly. 
In this case the complex part of 
the Hodge diagram is given by the usual 
log-de Rham complex $\Omega^*(\ln \ol{U}-U)$.

Note that in general 
a realization of $U=X-\cup Z_i$ as an NCD-complement $\ol{U}-\cup D_j$ 
uses
Hironaka's resolution of singularities. 
In the case of clean intersections one can also use wonderful compactification \cite{CP} 
introduced by de Concini-Procesi. These are more explicit in principle but still difficult to handle in examples.

To summarize, our approach is homological in nature and 
does not use an 
explicit construction of the resolution.
Thus in the case of an arbitrary arrangement with smooth 
set-theoretical intersections
(i.e.\ regardless of whether or not the intersection poset 
is locally geometric) we have a model for 
the rational homotopy type of the complement 
given by Theorem~\ref{thrm:main_cubical}.

\subsection{Applications}
\label{sect:intro:app}
We will now describe a few applications of our results. 

Consider an arrangement of affine subspaces in 
$\mb{C}^n$.
Let $(L,\leq)$ be the corresponding intersection poset
and $U=\mb{C}^n-\bigcup\limits_{x>0}L_x$ be the complement. 
In Theorem \ref{thrm:model_subspaces} we describe a model for the rational homotopy type of $U$.
The result generalize Yuzvinsky's \emph{relative atomistic} model for linear subspaces complements 
\cite{YuzAtomic}
to the affine case.

Assume now that $(L,\leq)$ is a locally geometric lattice.
This is equivalent to 
requiring that for all $p\in \mb{C}^n$ the lattice of subspaces
$\{x\in L\mid L_x\ni p\}\subset L$ is geometric.
This condition is satisfied for example in case of \emph{$C$-equal codimension arrangement}\label{equal_codimension}, 
i.e.  $C=\codim(Z_i/\mb{C}^n)$ and $\codim(L_x/\mb{C}^n)$ is divisible by $C$ for all $i,x$.
For instance $C=1$ corresponds to an arrangement of hyperplanes.
Using the model for $U$ we will show that
\begin{thrm}[Theorem \ref{thrm:model_subspaces}]
	The complement $U$ is 
	formal. 
\end{thrm}

\begin{rmrk}
Note that an arrangement given by affine hyperplanes $\{f_i=0\}$ is always
locally geometric; In this case the quasi-isomorphism $H^*(U;\mb{C})$ and
$A^*_{\PL}(U;\mb{C})$ is induced by the inclusion of subalgebra generated by 1-forms 
$df_i/f_i$. 
More generally, the case of $C$-equal codimension arrangement was 
covered in \cite{CH20} by applying principle ``purity implies formality'' 
(see also \cite{DupontArrangements}), 
by noting that in this case the MHS
$H^n(U)\ne 0$ is pure of weight $\frac{2Cn}{2C-1}\in \mb{Z}$ (see also \ref{rmrk:purity}).
\end{rmrk}

The next theorem generalizes the Kriz-Totaro model \cite{K},\cite{Totaro} 
which describes a model of the rational homotopy type 
of the configuration space $F(M,n)$ of 
$n$-tuples of points on 
a smooth compact algebraic variety $M$. 
Namely, for any finite unoriented graph $G$ one can consider the 
\emph{chromatic 
configuration space} $F(M,G)$ of points 
$x_i\in M$ labeled by 
vertices 
$i\in V(G)$ with the condition $x_i\neq x_j$ for each edge $\{ij\}\in E(G)$. 
Let $\Delta_{ab}\subset M^n$ be the diagonal $x_a=x_b$
and $[\Delta_{ab}]\in H^{2\dim M}(M^n)$ the corresponding
cohomology class.
\begin{rmrk}
The word ``chromatic'' is justified by the fact that 
$\chi(F(M,G))=p_G(\chi(M))$, where $p_G$ is the chromatic polynomial of 
$G$ and $\chi$ is any motivic measure on the Grothendieck ring of varieties.
\end{rmrk}
\begin{thrm}[Theorem \ref{thrm:kriz_model}]
		The rational homotopy type of $F(M,G)$ 
        has a model equal to the cdga over $H^*(M^{V(G)})$ given by the quotient 
	$$H^*(M^{V(G)})\otimes 
	\Lambda^*\langle \tilde{\Delta}_{ab}\rangle/\mc{J}$$
        where $\langle \tilde{\Delta}_{ab}\rangle$ is the $\mathbb{Q}$-vector space spanned by generators
	$\tilde{\Delta}_{ab}$ of degree 
 $2\dim_{\mb{C}} M-1$ for all \emph{ordered} pairs $(a,b)\in E(G)$,
	and $\mc{J}$ is the ideal generated by the following relations
	\begin{enumerate}
    \item $\tilde{\Delta}_{ab}=\tilde{\Delta}_{ba}$;
    
		\item $\partial(
		\tilde{\Delta}_{i_1i_2}\wedge\ldots\wedge \tilde{\Delta}_{i_{k-1}i_k})=0$
		for each simple cycle $(i_1,i_2),\ldots,(i_{k-1},i_k),(i_k,i_1)\in E(G)$ in $G$;

		\item $p^*_a(\gamma)\cdot \tilde{\Delta}_{ab}=
			p^*_b(\gamma)\cdot \tilde{\Delta}_{ab}$ for all $\gamma\in H^*(M)$ and $(a,b)\in E(G)$.
	\end{enumerate}

	The differential $d$ is zero on $H^*(M^{V(G)})$, and we have
	$d\tilde{\Delta}_{ab}=[\Delta_{ab}]\in H^{2\dim M}(M^{V(G)})$ where $[\Delta_{ab}]$ 
	is the Poincar\'e dual class of the diagonal $[\Delta_{ab}]$.

    We define an action of $Aut(G)$ on our model by setting 
	$\sigma.\tilde{\Delta}_{ab}=\tilde{\Delta}_{\sigma(a)\sigma(b)}, \sigma\in Aut(G)$ and using the permutation action on $M^{V(G)}$. With this definition, the model is $Aut(G)$-equivariant in the sense that there is a zig-zag of $Aut(G)$-equivariant quasi-isomorphisms connecting the model and $A_{PL}(F(M,G))$.

\end{thrm}

\begin{rmrk}
Assume the action of $Aut(G)$ on $F(M,G)$ is free, i.e.\ $G$ is the complete graph on $n=|V(G)|$ elements, and so $F(M,G)=F(M,n)$, the configuration space of ordered $n$-tuples of points of $M$. Then our model allows one to recover the $\Sigma_n$-equivariant rational homotopy type of $F(M,n)$, see~\cite{triantafillou}.
\end{rmrk}

\subsection{Organization of the work}
The work is organized as follows. In Section~\ref{sect:prelim} we recall 
basic facts about mixed Hodge complexes and 
state a result about $E_1$-formality of these, see 
Theorem \ref{thrm:mhd:E_1_qis}.
\
Then, in Section~\ref{sect:lattices} we review and/or prove a few facts about lattices which we need in order to 
treat the general case of our theorem beyond the case of the cubical lattice. 
There we introduce $(L,\leq)$-chain algebras and define 
Orlik-Solomon algebras in this context.
\
In Section~\ref{sect:aznar_mhd} we describe a functorial construction
of mixed Hodge diagram $(\mc{K}_{U/X},F,W)\in \MHD(X)$ associated with 
a variety $U$ over a smooth proper $X$
following Navarro Aznar \cite{Aznar}.

In Section~\ref{sect:construction} we construct the Mayer-Vietoris resolution 
in the topological setting. We describe the corresponding spectral sequence
in Theorem \ref{thrm:ss_for_lattices}, then
relate it to the Leray spectral sequence and to the lattice spectral sequence studied by Petersen and Tosteson \cite{Petersen}, \cite{To}. In Section~\ref{sect:compl:mvhc} we apply the $\Cech$ construction from Section~\ref{sect:cech_model} to construct a mixed Hodge diagram $\MVHD(X,\mc{C})\in \MHD(X)$, which we use to prove the main Theorem \ref{thrm:main_cubical}.  Finally, in Section~\ref{sect:applications} we prove the applications which we described above. 

Finally, we decided to move a technical lemma on the multiplicative structure of a certain homological M\"obius inversion formula into Appendix \ref{sect:tech_lemma}. 
In \ref{cup_product} the reader will find a simple result which allows, in the case of clean intersections, 
to describe our multiplicative model in terms of the arrangement cohomology of the form $H^*(Z)$ instead of $H^*_Z(X)$ via the Thom isomorphism. 
In \ref{first_example} we give two computational examples.

\subsection{Acknowlegements}
I would like to thank Dmitry Kaledin, Joana Cirici, Geoffroy Horel, Cl\'ement Dupont 
and Alexandru Suciu for their helpful comments on my thesis, the core part of which 
is presented here.
I owe a lot to stimulating discussions with Pierre Godfard, Dmitry Kubrak, Grigory Kondyrev and Grigory Papayanov.
Finally, I am especially grateful to my scientific advisors Alexey Gorinov 
and Julien March\'e
for
posing the problem, their encouragment and support, 
not to mention their
comments, remarks and criticism, without which this work would not have been
possible. 


\subsection{Notation, conventions and terminology}\label{sec:notation_etc}
\begin{enumerate}
        \item All complexes (of vector spaces, sheaves etc.) will be 
		assumed bounded below unless stated otherwise.
	\item $\sum\limits^{x\in S}_{f(x)}$ 
		denotes the sum 
		$\sum\limits_{\{x\in S\mid f(x)\}}$. 
		We use the same convention for direct sums etc.
	\item $S_1-S_2$ denotes the complement of $S_2\subset S_1$ in $S_1$.
        \item A {\it multi-index} is an ordered subset $I$ of 
		an ordered set $N$; 
		the order in $I$ is induced from $N$.
	\item For multi-indices $I,J$ we write $IJ:=I\cup J$.
	\item If $I=\{i_1<\ldots<i_{|I|}\}$ is a multi-index in 
		an ordered set $N$, 
		then we write 
		$d\tau^I=d\tau^{i_1}\wedge\ldots \wedge d\tau^{i_{|I|}}$,
		$d\nu_I=d\nu_{i_1}\wedge\ldots \wedge d\nu_{i_{|I|}}$
		to denote the Grassmann monomials 
		of degress $|I|$ and $-|I|$ respectively.

	\item $\Bbbk\langle v_i\mid 1\leq i\in N\rangle$ is the
		free $\Bbbk$-module on the variables $v_i$.
		Similarly, for a multi-index $I\subset N$, 
		denote by $\Bbbk\cdot d\nu_I$ 
		the free $\Bbbk$-module (in degree $-|I|$) 
		generated by $d\nu_I$. 
		It is independent of the ordering of $N$.

	\item $U/X$ is the object corresponding to an arrow $U\to X$.
	\item $\Bbbk-\Sh_X$ and
  $\Bbbk-\PSh_X$ is the category of sheaves, respectively presheaves of 
		$\Bbbk$-modules on a space $X$.
        We always work with sheaves and presheaves of abelian groups 
	and we omit $\Bbbk$ from the
        notation in the case $\Bbbk=\mb{Z}$.
        
	\item The sheafification of a presheaf $F\in \PSh_X$ is 
		denoted by $\ul{F}\in \Sh_X$.
        \item We write ${(j_*j^*)}_V={j_V}_*j^*_V$
            and ${(i_*i^!)}_Z={i_Z}_* i^!_Z$ 
            for an open embedding $j_V\colon V\to X$ and a closed embedding $i_Z\colon Z\to X$.
	\item MHC (MHD) means a mixed Hodge complex (diagram).
    	\item $\MHC(X)$ (respectively $\MHD(X)$) is the category of 
		mixed Hodge complexes (respectively diagrams) of sheaves 
		over $X$.
	\item 	$\Mon(C)$ (respectively $\CMon(C)$) is 
		the category of monoids 
		(respectively commutative monoids) in a (symmetric) 
		monoidal category $C$.
		We will also work with non-unital 
		monoids denoted by $\Mon^{\circ}(C),\CMon^{\circ}(C)$.
  
        \item If $C$ is a category, then $\cs C$ denotes the 
		category of cosimplicial objects in $C$.
	
    	\item Unless stated otherwise, all filtrations are assumed to be decreasing. 
		If $W$ is a decreasing filtration on an object $K$, then its 
		$k$-th piece is denoted by 
		$W^k K=W_{-k}K$ where $W_*$ denotes the corresponding 
		increasing filtration. 
		Likewise, we agree to use this notation for the
		associated graded objects, e.g.
		$\Gr^k_W K={}_{W}\Gr^k K=\Gr^W_{-k} K=
		\Gr^k(K,W)=\Gr_{-k}(K,W)$ etc.
		
	\item All filtrations are assumed to be biregular, 
		exhaustive and separated.
	
	\item All cohomological spectral sequences 
		we consider will be induced by decreasing filtrations,
	and the terms are indexed accordingly.

	\item $\Ch(\mc{A})$ (resp. $\Ch_F(\mc{A})$) is the category of 
 cohomologically bounded
		cochain complexes 
		(respectively filtered cochain complexes)
		in an additive category $\mc{A}$.

	\item $\mc{C}$ is the cubical lattice with vertices corresponding 
		to the subsets of an ordered set $N$. 

	\item $\ChF_{F},\ChF_{W},\ChF_{F,W}$ ($\ChF$ for short) 
		denotes the category of chain 
		complexes 
		with a (bi)filtration together with a prescribed behavior of 
		the shift functor $[-]$, see Definition \ref{dfn:chf}.

	\item 	If not stated otherwise, 
		by a quasi-isomorphism in $\ChF$ we mean 
		a quasi-isomorphism \textit{compatible} with 
		filtrations. Howevewer we will also work with 
		a more natural notion of
		a \emph{filtered quasi-isomorphism}. 
		We will be explicit by separating the two notions, 
		see Remark \ref{rmrk:compatible_filtered}.
		An explicit zig-zag of quasi-isomorphisms is 
		denoted by $\zgqis$, see Definition
		\ref{dfn:zigzag}. 


	\item All topological spaces are assumed to 
		be paracompact, Hausdorff and locally contractible.
	\item $\Var\subset \Top$ denotes the subcategory 
		of topological spaces formed by 
		smooth algebraic varieties over $\mb{C}$ and regular maps.
	
\end{enumerate}

\setcounter{tocdepth}{2}
\tableofcontents
\newpage
\section{Preliminaries on Hodge Complexes}\label{sect:prelim}
In this section we recall basic facts on filtered complexes, 
and fix some notation. 
The main result is Theorem \ref{thrm:mhd:E_1_qis} is due to 
Cirici and Horel\cite{CH20} stating in particular that cdga $K_{\mb{Q}}$ 
arising from a mixed Hodge diagram $(K,F,W)\in \MHD$
can be modeled by ${}_{W}E_1(K_{\mb{Q}})$ in a functorial manner.

In the main Section \ref{sect:compl:mvhc} 
we will construct a mixed Hodge diagram of 
sheaves $(\MVHD(X,\mc{Q}),F,W')$ with a simple description of 
${}_{W'}E_1$-term.
The passage $\MHD(X)\ar \MHD$
is given by the
Navarro Aznar construction $\Gamma_{TW}$ 
involving the Thom-Whitney-Sullivan functor, 
which we mention in Theorem \ref{thrm:prelim:aznar}. 

In Section \ref{sect:prelim:mixed_cone} 
we introduce the category of (multi)filtered 
complexes $\ChF=\ChF(\Sh_X)$ with a prescribed behavior 
of the shift functor $[-]$. 
An object in $\ChF$ can be thought of as 
a straightforward abstraction of a 
mixed Hodge complex viewed as (bi)filtered complex.

\subsection{Filtrations}
Here we recall some general preliminaries which 
can be found in
\cite{PS},\cite{D}. 
The reader can skip this subsection and 
come back later.  

Let $K^*$ be a complex with a filtration $W$.
Define
\begin{align}
	Z^{pq}_r&:=\ker[d\colon W^p K^{p+q}\to K^{p+q+1}/W^{p+r}K^{p+q+1}]\\
	B^{pq}_r&:=dW^{p-r+1}K^{p+q-1}+W^{p+1}K^{p+q}.
\end{align}
We note that:
$$dZ^{pq}_r\subset Z^{p+r,q-r+1}_r$$
$$dB^{pq}_r\subset B^{p+r,q-r+1}_r$$

There is a spectral sequence $E^{pq}_1={}_{W}E^{pq}_1(K^*)$
associated with filtration $W$:
\begin{gather}\label{eq:dfn:ss:dir}
	{}_{W}E^{pq}_r=\frac{Z^{pq}_r}{Z^{pq}_r\cap B^{pq}_r}\\
	\label{eq:dfn:ss:dirstar}
	\simeq \frac{Z^{pq}_r+B^{pq}_r}{B^{pq}_r}.
\end{gather}
with differential $d_r\colon E^{pq}_r\to E^{p+r,q-r+1}_r$.

One has a natural isomorphism:
\begin{equation}\label{eq:er_er}
H^*(E_r,d_r)\simeq E_{r+1}.
\end{equation}
A non trivial observation is that the realization of 
$E^{pq}_{r+1}$ as a sub-quotient of $E^{pq}_r$ 
is not consistent with either of the two given descriptions of $E_r$ in 
\eqref{eq:dfn:ss:dir},\eqref{eq:dfn:ss:dirstar} as a subquotient.
Namely, $Z_{r+1}\subset Z_{r}$ and $B_{r+1}\supset B_{r}$,
but it is in general not true that 
$Z_{r+1}\cap B_{r+1}\supset Z_{r}\cap B_r$
or $Z_{r+1}+B_{r+1}\subset Z_r+B_r$.

Let $F$ be another filtration on $K^*$.
Set $F_{dir},F_{dir^*}$ to be the filtrations on $E^{pq}_r={}_{W}E^{pq}_r$ 
induced by $F$ which are
given by the identifications
\eqref{eq:dfn:ss:dir} and \eqref{eq:dfn:ss:dirstar} respectively.
On the other hand, since 
$E_{r+1}=\ker[E_r\os{d_r}{\longrightarrow} E_r ]/\im[E_r\os{d_r}{\longrightarrow}E_r]$
one can define a filtration $F_{ind}$ on each $E_r$ inductively using \ref{eq:er_er} 
starting from $F$ on $E_0$.
This defines three natural filtrations on $E^{pq}_r={}_{W}E^{pq}_r$ 
with respect to $d_r$.

\begin{lemma}[
{\cite[Proposition 1.3.13]{D}}, {\cite[Lemma 3.11]{PS}}]
	For the filtrations introduced above the following holds:
	\begin{enumerate}
		\item $F_{dir}\subset F_{ind}\subset F_{dir^*}$.
		\item On $E_0,E_1$ all three filtrations coincide.
	\end{enumerate}
\end{lemma}
\begin{thrm}[on three filtrations; {\cite[Theorem  1.3.16]{D}},
	{\cite[Theorem 3.12]{PS}}]
	\label{lemma:three_filtrations}
	Assume the differential $d_{r}$ 
	is $F_{ind}$-strict
 for all $r=0,\ldots, r_0$,
	then for all $r'\in 0,\ldots,r_0+1$ we have:
	\begin{enumerate}
		\item $F_{dir}=F_{ind}=F_{dir^*}$ on $E_{r'}$.
		\item The sequence 
			$0\to {}_{W}E_{r'}(F^i K^*)\to {}_{W}E_{r'}(K^*)\to 
				{}_{W}E_{r'}(K^*/F^i K^*) \to 0$
			is exact for all $i$.
		\item For all $i$ we have ${}_{F}\Gr^i {}_{W}E_{r'}(K^*)\simeq 
			{}_{W}E_{r'}\ {}_{F}\Gr^i (K^*)$.
	\end{enumerate}
\end{thrm}
\begin{cor}
	If $d_r$ is $F_{ind}$-strict for all $r\geq 0$, then
	${}_{F}E_1={}_{F}E_\infty$
\end{cor}
\begin{proof}
	By the previous ${}_{F}\Gr^i\ {}_{W}E_r(K^*)=
		{}_{W}E_r\ {}_{F}\Gr^i(K^*)$. 
	Putting $r=\infty$ one has:

	$${}_{F}\Gr\  {}_{W}\Gr\ H^*(K^*)={}_{F}\Gr\ {}_{W}E_{\infty}(K^*)=
	{}_{W}E_\infty\ {}_{F}\Gr(K^*)={}_{W}\Gr\ {}_{F}E_1(K^*),$$
	hence by dimension reasons 
	${}_{F}E_\infty(K^*)={}_{F}E_1(K^*)$.
\end{proof}

\subsection{Mixed Hodge Complexes and Hodge Diagrams}
\label{sect:prelim:mhd}
In this section we recall the notions of mixed Hodge ($\mb{Q}$-)complexes, 
mixed Hodge complexes of sheaves and Hodge diagrams.

\begin{dfn}
        A \textit{zig-zag} of length $s$ is a diagram $I_s$ of the form:
        $$1 \longrightarrow 2 \longleftarrow 3 \longrightarrow\ldots 
                \longleftarrow s$$
\end{dfn}
Fix a sufficiently large $s$ and the corresponding zig-zag 
$I=I_s$ once and for all.
Consider for a moment a category $U$
with a chosen class of morphisms called weak equivalences.
The main example of $U$ will be the category of filtered complexes in
an abelian category with the class of quasi-isomorphisms compatible 
with filtration.

\begin{dfn}\label{dfn:zigzag}
        A \textit{zig-zag quasi-isomorphism} $\phi\colon A\zgqis B$ between
        $A,B\in U$
        is
        a functor $\Phi\in \Fun(I,U)$ where $\Phi(1)=A,\Phi(s)=B$
        and for all $i,j\in I$ and $[i\to j]\in \Hom_I(i,j)$,
        $\Phi([i\to j])$ is a weak equivalence in $U$.
\end{dfn}

In the sequel we will often omit the word ``zig-zag'' and refer to what 
we have just defined simply as quasi-isomorphisms.

\begin{dfn}[
{\cite[Definition 2.32]{PS}}]
	A data 
	$$\left(C_{\mb{Q}}^*,(C_{\mb{C}}^*,F),
	\phi\colon C_{\mb{Q}}\bigotimes\limits_{\mb{Q}}\mb{C}
	\zgqis 
		C_{\mb{C}}\right)$$ 
	where 
	\begin{itemize}
		\item $C_{\bbk}^*$ is a complex of modules over
			$\bbk=\mb{Q},\mb{C}$,
		\item
			$\phi$ is a quasi-isomorphism. 
	\end{itemize}
	is called a \textit{pure Hodge 
	complex}  of weight $m$
	if
	\begin{enumerate}
		\item ${}_{F}E_1 C_{\mb{C}}^*=
			{}_{F}E_{\infty} C_{\mb{C}}^*$, and
		\item $(H^n(C^*),F)$ is a 
			pure Hodge \textit{structure} of weight $n+m$. (This
			condition is phrased by saying that
			$F$ is \emph{$(n+m)$-conjugate}.)
	\end{enumerate}
	The data will be denoted simply by $(C^*,F)$.
\end{dfn}

If $(C^*,F)$ is a Hodge complex of weight $m$, 
then by the last condition there is a canonical identification:
$${}_{F}E^{pq}_1=F^p H^{p+q}\cap \bar{F}^{q+m}H^{p+q}$$
where $\bar{F}$ is the complex conjugate filtration to $F$.

\begin{dfn}[{\cite[Definition 3.13]{PS}}]
	A data 
	$$\left((K_{\mb{Q}}^*,W_{\mb{Q}}),
		(K_{\mb{C}}^*,W_{\mb{C}},F),
	\phi\colon (K_{\mb{Q}}\bigotimes\limits_{\mb{Q}}\mb{C},
	W_{\mb{Q}}\bigotimes\limits_{\mb{Q}} \mb{C})
	\zgqis
	(K_{\mb{C}},W_{\mb{C}})\right)$$
	where
	\begin{itemize}
		\item $K_{\bbk}$ is a complex of modules over 
			$\bbk=\mb{Q},\mb{C}$,
		\item $\phi$ is a 
			filtred quasi-isomorphism,
	\end{itemize}
	is called \textit{a mixed Hodge complex} (\textit{MHC} for short) if
	$\Gr^W_p K^*=\Gr^{-p}_W K^*$ is a pure Hodge complex of weight $p$.
	The data will be denoted by $(K^*,W,F)$.
\end{dfn}
Note that in the original definition given by Deligne the data defining 
a mixed Hodge complex lands in various derived categories
from the beginning. Following \cite{PS}, we consider mixed Hodge complexes
concretely, for example the quasi-isomorphism $\phi$ is a concrete 
zig-zag of actual maps. 
This makes the notion of a \textit{morphism} between mixed Hodge complexes
clear, thus we define the category of mixed Hodge complexes $\MHC$.

\begin{dfn}\label{dfn:decalage}
The \textit{d\'ecalage} $\Dec\ W$ of a filtration $W$ on a complex $K^*$ is the filtration given by:
$$(\Dec\ W)^{k} K^n:=\{x\in W^{k+n}K^n\mid dx\in W^{k+n+1}K^{n+1}\}$$
\end{dfn}

We denote a mixed Hodge \textit{structure} on a $\mb{Q}$-vector space $H$ 
by $(H,F,W)$, where $W$ is the \textit{weight} filtration on $V$ and
$F$ is the \textit{Hodge} filtration on $H\otimes \mb{C}$.

Giving an MHC $(K,F,W)$, passing to the cohomology we obtain
a filtered isomorphism 
$$(H^*(K_{\mb{C}}),H^*(W_{\mb{C}}))\simeq 
(H^*(K_{\mb{Q}})\otimes \mb{C},H^*(W_\mb{Q})\otimes \mb{C}),$$
and a bifiltered 
vector space $(H^*(K_{\mb{C}}),H^*(F),H^*(W_{\mb{C}}))$.
This provides a main source of mixed Hodge structures:
\begin{thrm}[Deligne; {\cite[Theorem 3.18]{PS}}]
	The cohomology of an MHC 
	$(K,F,W)$ 
	gives one an MHS 
	$(H^*(K_{\mb{Q}}),H^*(F),\Dec H^*(W))$.
\end{thrm}
\begin{rmrk}\label{rmrk:decalage}
Thus for the induced \textit{mixed Hodge structure}, the
weight filtration on $H^n:=H^n(K)$ is given by the d\'ecalage of the weight filtration on the
mixed Hodge complex, i.e.\ 
$(\Dec\ W)_k H^{n}=W_{k-n} H^n$, and thus the 
weight $k$ piece of the mixed Hodge structure on $H^n$
is equal to:
$$\Gr^{\Dec\ W}_{k}H^n=\Gr^{-k}_{\Dec\ W}H^{n}=
\Gr^{n-k}_W H^{n}.$$
For example, if the Hodge complex $(K,F,W)$ is pure of weight $0$, 
then $H^n$ is a pure Hodge structure of weight $n$.
\end{rmrk}

Recall the following basic fact about MHS's:
\begin{prop}[Deligne splitting;{\cite[Lemma-Definition 3.4]{PS}}]
	\label{prop:deligne_splitting}
	An MHS $(V,F,W)$ admits a canonical 
	decomposition $V=\bigoplus V^{p,q}$ defined in terms of $F,W$, 
	such that 
	$\bigoplus_{p+q=k}V^{p,q}\to \Gr^{k}_W V$ is an isomorphism onto the
	Hodge components of the associated pure Hodge structure.
\end{prop}	
This decomposition provides a canonical isomorphism of vector spaces
	$\Gr^*_W V_{\mb{C}}\simeq V_{\mb{C}}$.
It follows that a morphism between MHS's is necessary \emph{strict} 
with respect to filtrations $F,W$.

In particular a morphism between two pure Hodge structures 
of different weights is necessary zero.
The basic corollary of this fact is the following
\begin{prop}[{\cite[Theorem 3.18]{PS}}]
The Hodge filtration $F$ on
$E^{pq}_1(K,W)=\Gr^p_W H^n(K)$ is $q$-conjugate, i.e.\ 
$E^{pq}_1(K,W)$ is a pure Hodge structure of weigth $q$.
Differentials $d_r$ are morphisms of Hodge structures, 
in particular $E^{pq}_2(K,W)=E^{pq}_\infty(K,W)$,
and we have an isomorphism of pure Hodge structures:
$$\Gr^{\Dec\ W}_q H^n(K)=E^{pq}_2(K,W)$$
\end{prop}

\begin{rmrk}\label{rmrk:compatible_filtered}
In particular a quasi-isomorphism of MHC's $(K_1,F,W)\zgqis (K_2,F,W)$ 
which is required to be only \emph{compatible} with filtrations,
induces an \emph{isomorphism}
on cohomology considered as MHS's. In particular the induced zig-zag 
${}_{W}E_1(K_1)\zgqis {}_{W}E_1(K_2)$ is automatically a 
quasi-isomorphism, though not necessary an isomorphism.

More generally, in the setting of MHC's, the notion 
of filtered quasi-isomorphisms is too restrictive.
	For example Beilinson's theorem \cite{Beilinson} suggest that
it is convinient to define a weak equivalence in $\MHC$ 
to be a compatible quasi-isomorphism
$(K_1,F,W)\ar (K_2,F,W)$, see also \cite{CiriciGuillen}.

Since $\ChF$ was introduced to model behavior of MHC's, we 
will explicitely use adjective ``filtered'' to refer
to filtered quasi-isomorphisms. 
We note however that most constructions
of quasi-isomorphisms in this work are naturally filtered.
\end{rmrk}

The following property holds for any filtered complex:
\begin{prop}[Deligne; {\cite[Proposition 1.3.4]{D}}]
	\label{prop:decalage_ss}
There is a natural quasi-isomorphism
$$E_0(K,\Dec\ W)\zgqis E_1(K,W)$$
and natural isomorphisms for $r>1$
$$E^{pq}_r(K,\Dec\ W)\simeq E^{p+n,-p}_{r+1}(K,W),n=p+q.$$
\end{prop}

\begin{dfn}
	A \textit{mixed Hodge diagram} (\textit{MHD} for short)
	is an object in $\CMon(\MHC)$.
\end{dfn}
We will often denote the corresponding category by $\MHD$.
Set
$U((K,F,W)):=K_{\mb{Q}}$ and $E_1((K,F,W)):={}_{W}E_1(K_{\mb{Q}})$
gives a pair of functors  
$$U,E_1\colon (\MHC,\otimes)\ar (\Ch(\Vect_{\mb{Q}}),\otimes).$$
The following result is the key for our work:
\begin{thrm}[$E_1$-formality principle,{\cite[Cirici-Horel]{CH20}}]
	\label{thrm:mhd:E_1_qis}
	There is a quasi-isomorphism of lax symmetric monoidal functors
	$U\zgqis E_1$.
\end{thrm}
\begin{proof}
	The statement follows by the proof of \cite[Theorem 7.8]{CH20} 
	and by \cite[Theorem 2.3]{CH20},
which in turn was proved by Hinich in \cite[Theorem 4.1.1]{hinich}.
\end{proof}
\begin{rmrk}
	In particular, given $(K,F,W)\in \MHD$ provides a quasi-isomorphism 
	$K_{\mb{Q}}\zgqis {}_{W}E_1(K_{\mb{Q}})$ of cdga's.
	Over $\mb{C}$ this was proved by Morgan in 
	\cite[Theorem 9.6]{M} by showing that $K$ admits 
	a minimal Sullivan model $M$ 
	with a compatible mixed Hodge \emph{structure} and 
	then applying Deligne's splitting to $M$.
	Over $\mb{Q}$ an existence of a non-canonical splitting was 
	proved in \cite[Theorem 10.1]{M}.
\end{rmrk}
\begin{cor}\label{cor:cirici_splitting}
	The object $(H^*(K_{\mb{Q}}),\Dec W)$ admits a 
	functorial splitting in $(K,F,W)\in \MHC$. 
\end{cor}
\begin{rmrk}
	It is not true that this splitting over $\mb{C}$ coincides 
	with the Deligne's splitting.
\end{rmrk}

Up to this moment we discussed mixed Hodge complexes and diagrams of 
$\mb{Q}$-modules. In fact it is easier to work with Hodge complexes of sheaves of 
$\mb{Q}$-modules over a space $X$ and then pass to global sections.
\begin{dfn}[{\cite[Definition 2.32]{PS}}]
	A data 
	$$\left(\mc{K}_{\mb{Q}}^*,(\mc{K}_{\mb{C}}^*,F),
	\phi:\mc{K}_{\mb{Q}}\bigotimes\limits_{\mb{Q}}\mb{C}
	\zgqis 
		\mc{K}_{\mb{C}}\right)$$ 
	where 
	\begin{itemize}
		\item $\mc{K}_{\bbk}^*$ is a complex of sheaves of
			$\bbk=\mb{Q}-,\mb{C}-$ modules over $X$,
		\item
			$\phi$ is a quasi-isomorphism of complexes of sheaves 
	\end{itemize}
	is called a \textit{pure Hodge 
	complex of sheaves} of weight $m$
	if
	\begin{enumerate}
		\item The associated spectral sequence 
	$${}_{F}E^{pq}_1=
	\mb{H}^{p+q}(X;\Gr^{p}_F\mc{K}^*_{\mb{C}})
		\Rightarrow 
		\Gr^{p}_F\mb{H}^{p+q}(X;\mc{K}_{\mb{C}})$$
			degenerates at $E_1$.
		\item The induced filtration $F$ on 
			$(H^n(\mc{K}^*_{\mb{C}}),F)$ is 
			$(n+m)$-conjugate, i.e.~$(H^n(\mc{K}^*_{\mb{C}}),F)$ 
			defines a pure Hodge structure
			of weight $n+m$.
	\end{enumerate}
	The data will be denoted simply by $(\mc{K}^*,F,W)$.
\end{dfn}

\begin{dfn}[{\cite[Definition 3.13]{PS}}]\label{dfn:mhc}
	A \textit{mixed Hodge complex of sheaves} over $X$ is a data
	$$\left((\mc{K}_{\mb{Q}},W_{\mb{Q}}),
	(\mc{K}_{\mb{C}},F,W_{\mb{C}}),
	\phi:(\mc{K}_{\mb{Q}}\otimes \mb{C},
		W_{\mb{Q}}\otimes \mb{C})\zgqis
	(\mc{K}_{\mb{C}},W_{\mb{C}})\right)$$
	where
	\begin{itemize}
		\item ${\mc{K}}_{\bbk}$ is a complex of sheaves of 
			$\bbk$-modules over $X$ where $\bbk=\mb{Q}$ or $\mb{C}$,
		\item $\phi$ is a 
			filtered quasi-isomorphism of complexes of 
			sheaves over $X$
	\end{itemize}
	such that for each $p$,
	$(\Gr^{-p}_{W_{\mb{Q}}}\mc{K}_{\mb{Q}},
		(\Gr^{-p}_{W_{\mb{C}}}\mc{K}_{\mb{C}},F))$
	is a pure Hodge complex of sheaves of weight $p$.
\end{dfn}
We denote the category of MHC's of sheaves over $X$ by 
$\MHC(X)$.
The following fact is also useful.
\begin{thrm}[{\cite[3.18]{PS}}]
	\label{thrm:ss_degeneration_mhc}
	Assume that $(\mc{K},F,W)\in \MHC(\Sh_X)$.
	Then the following is true.
	\begin{enumerate}
		\item The spectral sequence 
			${}_{F}E^{pq}_1 \mc{K}=
		\mb{H}^{p+q}(X,\Gr^p_F \mc{K})\implies 
		\Gr^p_F \mb{H}^{p+q}(X,\mc{K})$
		degenerates at ${}_{F}E_1$.
		\item The spectral sequence 
			${}_{W}E^{pq}_1 \mc{K} = 
			\mb{H}^{p+q}(X,\Gr^p_W \mc{K})
			\implies Gr^p_W \mb{H}^{p+q}(X,\mc{K})$
			degenerates at ${}_{W}E_2$.
	\end{enumerate}
\end{thrm}

Observe that for $U\subset X$, $(\mc{K}(U),F,W)$ is rarely 
an MHC.
However passing to derived global sections 
by means e.g.\ of the canonical Godement resolution $G\colon \Ch(\Sh)\ar \Ch(\Sh)$
is.
Namely we have a functor
$R_{Gdm}\Gamma\colon \MHC(X)\to \MHC$
(\cite[Lemma-Definition 3.25]{PS}) given by $R_{Gdm}\Gamma:=\Gamma\circ G$.
Note that $R_{Gdm}\Gamma$ is not lax symmetric monoidal and 
does not induce a functor $\MHD(X)\ar \MHD$.
Though $R_{Gdm}\Gamma$ is not a lax symmetric monoidal,
by
working over $\mb{Q}$ it is possible to replace this functor by an equivalent
lax symmetric monoidal by using the Thom-Whitney-Sullivan functor $s_{TW}$ 
described in \ref{sect:aznar_mhd}. For now we only 
briefly mention
\begin{thrm}[Navarro Aznar; Theorem \ref{thrm:mhd_sheaves_sections}]
	\label{thrm:prelim:aznar}
There is a natural lax symmetric monoidal functor
$R_{TW}\Gamma\colon (\MHC(X),\otimes)\to (\MHC,\otimes)$.
\end{thrm}
In particular we get a functor $R_{TW}\Gamma\colon \MHD(X)\ar \MHD$.
\subsubsection{Mixed Cone}
\label{sect:prelim:mixed_cone}
\begin{dfn}
	If $f:A\to B$ is a morphism between complexes, 
	set
	$[A\os{f}{\to} B]:=Cone(A\os{f}{\to} B)[-1]$.
\end{dfn}

\begin{prop}[Mixed cone construction; {\cite[Theorem 3.22]{PS}}]
	\label{prop:mixed_cone}
	Assume $f\colon A\to B$ is a morphism of mixed Hodge complexes. Then
	$K:=[A\os{f}{\to} B],K^n=A^{n}\oplus B^{n-1}$ 
	becomes a mixed Hodge complex if we put:
	\begin{itemize}
		\item $W^p K:=W^{p}A\oplus W^{p-1}B$
		\item $F^p K:=F^p A\oplus F^p B$
	\end{itemize}
	There is an exact sequence of MHC:
	$$0\to B[-1]\to [A\to B]\to A\to 0.$$
\end{prop}
\begin{proof}
	We will use the idea of the following proof in the sequel, so we reproduce it here. It is enough to show that 
 $\Gr^p_W K^*$ is a pure Hodge complex
	of weight $-p$.
	Note that $\Gr^p_W K^*$ as a complex ignores $f$ and 
	is equal to the direct sum 
	$$\Gr^{p}_W A\oplus \Gr^{p-1}_W B [-1].$$
	So $F$ induces a pure Hodge structure of weight $n-p$ on the
	$n$-th cohomology equal to 
	$$H^{n}(\Gr^{p}_W A)\oplus H^{n-1}(\Gr^{p-1}_W B).$$
	Hence $[A\os{f}{\to} B]$ is a mixed Hodge complex. 
\end{proof}

\
\begin{dfn}[Shift functor]
	If $(C^*,F)$ is a Hodge complex of weight $m$, 
	then $(C^*[a],F)$ is a Hodge complex of weight $m+a$.
	If $(K^*,F,W)$ is an MHC, then
	$(K^*[a],F,W[a])$ is an MHC where
	$W[a]^p=W^{a+p}$.
\end{dfn}

Let $\Bbbk-\mc{A}$ denote a symmetric monoidal additive category
enriched over $\Bbbk$, e.g.~the
category of sheaves of $\Bbbk$-modules over $X$, $\Bbbk-\Sh_X$.
To capture the mixed cone behavior on the Hodge and the weight filtrations
let us introduce categories $\ChF_{F}=\ChF_{F}(\Bbbk-\mc{A})$ and 
$\ChF_{W}=\ChF_{W}(\Bbbk-\mc{A})$.
Each of these is the category of filtered complexes in $\mc{A}$ with the filtration denoted $F$, 
respectively $W$, with an extra data of a shift functor $[-]$ defined as follows. 
For $(K^*,F)\in \ChF_{F}$ we define $(K^*,F)[a]=(K^*[a],F)$,
while for $(K^*,W)\in \ChF_{W}$ we set $(K^*,W)[a]=(K^*[a],W[a])$.

Similarly we denote by $\ChF_{F,W}(\Bbbk-\mc{A})$ the category of 
complexes
in $\mc{A}$ with filtrations
$F$ and $W$ and a shift functor defined by $(K^*,F,W)[a]=(K^*[a],F,W[a])$.

\begin{notation}\label{dfn:chf}
	We will refer to either $\ChF_{F}$, $\ChF_{W}$ or $\ChF_{F,W}$
as $\ChF$. 
\end{notation}
By analogy with~Proposition~\ref{prop:mixed_cone} one can define cones 
$[A\to B]\in \ChF$ for $A,B\in \ChF$. 
\begin{rmrk}\label{rmrk:ChF}
	The resulting cone functors model the rational part, 
respectively complex part of the mixed cone of 
a morphism between MHC's.
We will abuse notation and consider
$\MHC(X)$ as a subcategory of $\ChF_{F,W}(\Sh_X)$.
Objects of $\ChF_{F,W}(\Sh_X)$ are the same as in Definition \ref{dfn:mhc} except
the last cohomological condition. There natural projections 
$\ChF_{F,W}(\Sh_X)$ to $\ChF_W(\mb{Q}-\Sh_X)$ and to $\ChF_{F,W}(\mb{C}-\Sh_X)$
respectively.
\end{rmrk}
If $\mc{K}\in \ChF$ is a 
filtered complex, then we denote the corresponding filtration by  
$V_\mc{K}$ or $V\mc{K}$. 
Therefore, if $\mc{K}\in \ChF_{F}$, then 
$(V^p\mc{K}[i])^n=V^p\mc{K}^{n+i}$.
If $\mc{K}\in \ChF_{W}$, then 
$(V^p{\mc{K}[i]})^n=V^{p+i}\mc{K}^{n+i}$.

If $\mc{K}',\mc{K}''\in \ChF$, then 
the complex $\mc{K}'\otimes \mc{K}''$ admits 
a tensor product filtration 
$$(V_{\mc{K}'}\otimes V_{\mc{K}''})^n=
	\sum_i V_{\mc{K}'}^{i}\otimes V_{\mc{K}''}^{n-i}.$$

One has the following straightforward generalization of the mixed cone 
construction.
\begin{dfn}[Iterated cone]\label{dfn:iterated_cone}
	For a bicomplex $\ldots\to A_{-1}\to A_0\to A_1\to \ldots$ 
	with $A_i\in \ChF$,
	we set 
	$K=[\ldots\to A_{-1}\to A_0\to A_1\to \ldots]\in \ChF$
	equal to the totalization
	$$K=\oplus_i A_i[i].$$
\end{dfn}
For example $[A_0\to A_1\to A_2\to ]$ is equal to 
the \emph{iterated cone} $[A_0\to [A_1\to [A_2\to \ldots]]$.
Proposition \ref{prop:mixed_cone} shows that
if the maps $A_i\ar A_{i+1}$ are morphisms of mixed Hodge complexes
for all $i$, 
then $[\ldots\to A_{-1}\to A_0\to A_1\to\ldots]$ is a mixed Hodge complex.
 
\begin{dfn}
	The complex $\mc{K}'\otimes \mc{K}''\in \ChF$
	is called the \textit{tensor product} of 
		$\mc{K}'$ and $\mc{K}''$.
\end{dfn}
For all $a,b$, there is a natural isomorphism 
$$(\mc{K}'\otimes \mc{K}'')[a+b]\simeq \mc{K}'[a]\otimes \mc{K}''[b].$$

\begin{prop}\label{prop:vect_to_ch}
	There is a natural embedding $\Ch(\Vect_{\Bbbk})\to \ChF$
	commuting with $[-]$ and sending $\Bbbk$ to $1_{\mc{A}}$.
\end{prop}
\begin{proof}
	In fact the functor is unique up to a natural isomorphism.
	We start with a functor $\Vect_{\Bbbk}\to \Bbbk-\mc{A}$ 
	sending $\Bbbk$ to the unital object $1_{\mc{A}}\in \Bbbk-\mc{A}$.
	This functor is determined up to a natural isomorphism.
	It defines an embedding $\Ch(\Vect_{\Bbbk})\to \Ch(\Bbbk-\mc{A})$.
	If $\ChF=\ChF_{F}$, then $D\in \Ch(\Vect_{\Bbbk})$ 
	admits the trivial filtration $V_D^n=0,n>0$ and $V_D^n=D,n\leq 0$.
	If $\ChF=\ChF_{W}$, then we define $V_D$ by  
	the dumb filtration $V_D^n (D^i)=D^i,n-i\leq 0$ and zero otherwise, i.e.\ $V^n_D D^*=D^{\geq n}$. 
	Then $d_D V_D^n\subset V_D^{n+1}$.
	In both cases $\mc{D}:=(D,V_D)\in \ChF$.
	It is straighforward to check that the functor commutes with $[-]$
	and is unique up to a natural isomorphism.
\end{proof}
Assume $D\in \Ch(\Vect_{\Bbbk})$ and $\mc{K}\in \ChF$.
\begin{dfn}\label{dfn:chf_tensor_product}
	The tensor product $\mc{K}\otimes D$ of $\mc{K}$ by $D$ is 
	by definition the tensor product $\mc{K}\otimes \mc{D}$,
 where $\mc{D}$ is the image of $D$ under the natural embedding
 from Proposition \ref{prop:vect_to_ch}.
\end{dfn}
Equivalently we have a natural isomorphism
$$\mc{K}\otimes D\simeq \oplus_i \mc{K}[-i]\otimes D^i.$$
where the RHS is a filtered complex with the differential equal to $d_{\mc{K}}[i]+d_D$ 
obeying the usual sign rule and the filtration
induced by $V_{\mc{K}[i]}$.

\newpage
\section{Lattices}\label{sect:lattices}
Let $(L,\leq)$ be a poset. We start this section by recalling necessary
definitions from theory of lattices and then in Section \ref{sect:lattices:arrangement_poset}
give the definition of an arrangement $[X,L]$
over a topological space $X$ given by a presheaf 
$(L,\le)^{\op}\to 2^X$,
to which our 
algebraic model of Theorem \ref{thrm:main_lattices} can be applied. 
In Section \ref{sect:lattices:l_algebras} 
we introduce, under some conditions on $L$, the notion of an $(L,\leq)$-chain
algebra. Roughly speaking it is a dga graded by the vertices of $L$ with an additional 
differential that reflects the combinatorics of $(L,\le)$. 
This notion is different from what we call an $(L,\le)$-object, 
which is simply a presheaf on the category $(L,\le)$ (so we can speak about 
$(L,\le)$-complexes, etc.). 
We will construct basic operations such as
the tensor product/convolution of $(L,\le)$-algebra with an 
$(L,\le)$-chain algebra, totalization,
pullback, pushforward, and prove
the projection formula. 

In Section \ref{sect:lattices:os} we recall the definition of the 
Orlik-Solomon algebra $\OS(L)$ viewed as an $(L,\le)$-chain algebra.

In Section \ref{sect:lattices:mhc} we will see, in particular, that
the convolution $K*\OS(L)=\Tot(K\otimes \OS(L))$ of an 
$(L,\le)$-pure Hodge complex $K$ of weight $0$
with the $(L,\le)$-chain complex $\OS(L)$
defines two mixed Hodge complexes $(K*\Mu,F,W)$ and $(K*\OS(L), F, W')$, 
called the {\it natural} and the {\it rank} Hodge complex respectively.
In Lemma \ref{lemma:two_weight_filtrations} we will show that 
the corresponding mixed Hodge structures 
coincide.

The formalism below can be viewed as an attempt 
to create ``lattice homological algebra'' from scratch. 
These tools allow us to handle 
the homological combinatorics behind complexes of sheaves 
related to the arrangement, which in turn describes the homotopy type of the complement $U=X-\bigcup\limits^{x\in L}_{x>0}L_x$ in the complex algebraic case. 
Another perspective on the 
additive part of the story was developed in \cite{To}.
Our machinery is designed mostly to handle multiplicative questions, and it will be used in the 
main construction of the mixed Hodge diagram $\MVHD(L)$ 
in Section \ref{sect:compl:mvhc}.

\
\subsection{Posets}
\begin{dfn}
For a subset $S\subset L$ we write $S\leq t,t\in L$, and 
call $t$ an \textit{upper bound} for $S$,
if $s\leq t$ for all $s\in S$.
\end{dfn}
Dually we define an \textit{lower bound} of $S\subset L$.
\begin{dfn}
	An element $t\in S$ is \textit{maximal} in $S\subset L$ if
	for all $t'\in S$,
	$t'\geq t$ implies $t'=t$.
\end{dfn}
The set of maximal elements in $S\subset L$ 
is denoted by $\max(S)\subset L$. Dually we define $\min(S)\subset L$.

\begin{dfn}
	If $(L,\leq)$ is a poset and $S\subset L$,  then
	$t=\sup {S}\in L$ is called the \textit{supremum} of $S$, if 
	$t\geq S$ is an upper bound and 
	if for any upper bound $t'\geq S$
	we have $t'\leq t$.
	In other words, $\sup {S}$ is the lowest upper bound of $S$.
\end{dfn}
\begin{dfn}
	Denote $\os{\circ}{\sup{}} {S}\subset L$ to be the subset of 
	all upper bounds $t\geq S$ such that
	for all $t'\geq S$, $t\geq t'$ implies
	$t'=t$.
	In other words $\os{\circ}{\sup{}} S=\min\{t\in L|t\geq S\}$
	is the set of minimal upper bounds.
\end{dfn}
Dually we define $\inf$ and $\os{\circ}{\inf}$.
\begin{dfn}
	We say that $x$ \textit{covers} $y$ if 
	$x\geq y,x\neq y$ and there are no
	other elements between $x$ and $y$.
	Denote this relation by $x :>y$.
\end{dfn}

In the following we assume that all posets $(L,\leq)$ admit an infinum, 
since it is unique and we denote 
it by $0=\inf(L)\in L$.
In case $\sup(L)$ may not exists, we call such $(L,\leq)$ a \textit{local} poset. 

\begin{dfn}
	The elements $\Atoms(L):=\{x\in L\mid x:>0\}$ are called the
	\textit{atoms} of $(L,\leq)$. 
\end{dfn}

\begin{dfn}
	A poset $(L,\leq)$ is \textit{graded} 
	if it is equipped with a strict order-preserving
	function $r\colon (L,\leq)\to (\mb{Z},\leq)$
	such that $x<:y$ implies $r(y)=r(x)+1$.
\end{dfn}
In this case we refer to $r$ as the \emph{grading} or the \emph{rank function} of $(L,\leq)$.
It is easy to check that a poset admitting a strict order-preserving function 
admits a grading.
Clearly, if a such $r$ exists, it is determined by the value 
$r(0)$. We will assume that $r(0)=0$.
Note that $a\in \Atoms(L)$ iff $r(a)=1$.
\
\begin{dfn}\label{dfn:lattice:interval}
For given $x,y\in (L,\leq)$,
the subset $L[x,y]:=\{t\in L\mid x \leq t\leq y\}\subset L$ is called
the \textit{interval} between $x$ and $y$ in $(L,\leq)$.
\end{dfn}

\begin{dfn}
	A local poset $(L,\leq)$ is a \textit{sup-lattice} 
	if for any $S\subset L$ 
	there is $\sup(S)$ and for each $x\in L$ there are $a_1,\ldots,a_n\in \Atoms(L)$ such that
	$x=\sup(a_1,\ldots,a_n)$. 
	A \emph{lattice} is a poset which is both a sup- and inf- lattice.
	In each case for any $a,b\in L$ the operations
	$a\vee b:=\sup(a,b)$ and 
	$a\wedge b:=\inf(a,b)$ are well defined;
	they are called \emph{join} and \emph{meet} respectively.
\end{dfn}
\begin{rmrk}
	In other terminology this definition describes an \emph{atomistic sup-lattice}.
\end{rmrk}
\begin{dfn}
	A local poset $(L,\leq)$
	is a \textit{local} (graded) sup-lattice, 
	if for any $x\in L$,
	$L[0,x]$ is a (graded) sup-lattice.
	If $L$ admits a grading, 
	then we define the grading $r$ on $L$ by 
	setting $r(x):=r(L[0,x])$, assuming $r(0)=0$.
\end{dfn}
Clearly the grading $r$ of a local sup-lattice $L$ 
satisfies $r(x)=r(y)+1$ if $x:> y$ and 
is uniquely determined by the condition $r(0)=0$.
We will denote by $x\vee_{[0,t]}y$ for $x,y\le t\in L$
the join taken in the sup-lattice $L[0,t]$.

In subsequel all posets, if not stated otherwise, will be
graded (local) atomistic sup-lattices, 
we refer to them as (local) sup-lattices.

\begin{dfn}\label{dfn:lattice:geometric}
	A graded lattice $(L,\leq)$ is called \textit{geometric lattice}, 
	if 
$r(x\vee y)+r(x\wedge y)\leq r(x)+r(y)$ for all $x,y\in L$.
\end{dfn}

Following \cite{JCW}, one defines \emph{locally geometric
lattices} by imposing the corresponding 
conditions on all intervals.
\begin{dfn}
	Elements $x_i\le t\in L,i\leq n$ of 
	a graded local geometric geometric lattice
	$(L,\le)$ are 
	called \textit{independent} in $L[0,t]$,
	if $$r(\sup_{[0,t]}(x_1\vee\ldots \vee x_n))=
	\sum\limits^{i}_{1\leq i\leq n} r(x_i).$$
\end{dfn}
In this case every subset of independent elements is independent.
Elements 
$x_1,\ldots,x_n\le t$ are independent 
iff for all $k\leq n$ we have 
$\sup_{[0,t]}(x_1\vee \ldots \vee\widehat{x_k}\vee\ldots \vee x_n)<
	\sup_{[0,t]}(x_1\vee\ldots\vee x_n)$.
In particular for each $y\le t\in L$ there are $x_1,\ldots,x_r\le t\in \Atoms(L)$
such that $y=\sup_{[0,t]}(x_1\vee\ldots\vee x_r)$,
where $r=r(y)$.
In subsequel all posets will be assumed to be graded.

Assume $(L,\le),(L',\le)$ are local sup-lattices.
\begin{dfn}\label{dfn:lattices:homomorphism}
	An order preserving map  $f\colon (L,\leq)\to (L',\leq)$ 
	is a \textit{homomorphism of local sup-lattices} if
	\begin{enumerate}
		\item $f(0)=0$.
		\item $f(x\vee y)=f(x)\vee f(y)\in L'[0,f(t)]$ for 
			any $x,y\in L[0,t]$.
		\item $f^{-1}(L'[0,t])$ is an interval of $L$ for 
			all $t\in L'$.
	\end{enumerate}
	In the case $(L,\le),(L',\le)$ are locally geometric lattices,
	then $f$ is a \emph{homomorphism of locally geometric lattices}
	if in addition the equality
	$f(x\wedge y)=f(x)\wedge f(y)$ holds for each sublattice $x,y\in L[0,t]$.
\end{dfn}

\begin{dfn}\label{dfn:lattices:locally_injective}
	Say a map $f\colon \Atoms(L)\ar \Atoms(L')$ is \emph{locally injective}, if
	for all $t\in L$ and $a,b\in \Atoms(L[0,t])$ we have $g(a)=g(b)$ iff $a=b$.
\end{dfn}
\begin{rmrk}\label{rmrk:lattices:geom_homomorphism_preserves_r}
	If $f\colon (L,\le)\ar (L',\le)$ 
	is a homomorphism of graded geometric lattices,
	then for all $t\in L$ and each independent set $x_i\in L[0,t]$, 
	elements $f(x_i)\in L'[0,f(t)]$ are independent.
	In particular $f$ preserves the rank function:
	$r(f(x))=r(x)$ for all $x\in L$ and the restriction
	$f\colon \Atoms(L)\ar \Atoms(L')$ is a well-defined map which 
	is locally injective.
\end{rmrk}

\begin{dfn}\label{dfn:lattices:contraction}
	An order preserving map $f\colon (L,\leq)\to (L',\leq)$ is 
	a \textit{contraction}, if
	 $y<:x$ implies $f(y)<:f(x)$ or $f(y)=f(x)$.
\end{dfn}
Equivalently, $f$ is contraction iff $r(f(x))\le r(x)$ for all $x\in L$.

\begin{dfn}
	A subposet $L'\subset L$ is called \textit{complete} if
	for all $x,y\in L'$ and $t\in L'$ such that $x\leq t\leq y$ we have
	$t\in L'$.
\end{dfn}
So the preimage of complete subposet under an order preserving
map is complete.
\begin{dfn}
	We call an order-preserving map $f\colon (L,\le)\ar (L',\le)$ \emph{complete},
	if $f(L)\subset L'$ is complete.
\end{dfn}
\begin{rmrk}\label{rmrk:I_commutes}
If $f$ is a homomorphism preserving the rank function, then $f$ is complete
iff for all $a'\in \Atoms(L')$ and $x\in L$ such that $a'<f(x)$ we have $a'\in f(\Atoms(L))$.
In other words for all $x\in L$ the natural inclusion 
$f(\{a\le x\mid a\in \Atoms(L)\})\subset \{a'\le f(x)\mid a'\in \Atoms(L')\}$
is surjective.
\end{rmrk}

\subsection{\texorpdfstring{$(L,\leq)$}{L}--objects}
\label{sect:lattices:l_algebras}
Let $(L,\leq)$ be a graded poset.
Fix a symmetric monoidal abelian category $\Bbbk-\mc{A}$
enriched over $\Bbbk$, e.g.\ $\Bbbk-\Sh_X$.
The objects of $\ChF=\Ch(\Bbbk-\mc{A})$ will be referred to as \emph{complexes}.
In this section we will consider two main types of complexes graded by a poset $(L,\le)$
and study basic operations on them.

\begin{dfn}[$(L,\leq)$-object]
	An $(L,\leq)$-\textit{object} in a category $C$
	is a functor $F\colon (L,\leq)^{\op}\to C$.
	Its values are denoted by $F_s\in C,s\in L$.
	The corresponding morphisms $g^F_{yx}\colon F_x\to F_y,x\geq y\in L$
	are called the \textit{structure maps}.
\end{dfn}

\subsubsection{$(L,\le)$-algebras}
If $(L,\le)$ is a sup-lattice, then we can consider $-\vee-$
as a monoidal structure on the category $(L,\le)$.
Then an $(L,\le)$-\emph{algebra} in $\ChF$ is a lax 
monoidal functor
$((L,\le)^{\op},\vee^{\op})\ar (\ChF,\otimes)$.
Below we will explicitly describe $(L,\leq)$-algebras in the case when $(L,\le)$ is a local
sup-lattice.

Let $\mc{S}\subset \{(x,y,t)\mid x,y,t\in L\}$ be a subset.
We denote the corresponding indicator function by $\mc{S}^{t}_{xy}$, i.e.
$\mc{S}^{t}_{xy}\neq 0$ iff $(x,y,t)\in \mc{S}$.
We require  that for all $x,y,z,t\in L$ the following should should hold:
\begin{enumerate}
	\item $\mc{S}^{t}_{xy}=0$ if $t\not\in x\os{\circ}{\vee}y$.
	\item $\mc{S}^{t}_{xy}\neq 0$, then for each $x'\le x,y'\le y$ 
		and $t'=x\vee_{[0,t]} y$ 
		we have $\mc{S}^{t'}_{x'y'}\neq 0$.
	\item If $t\in \os{\circ}{\sup}(x,y,z)$, then 
		for $t_1= x\vee_{[0,t]}y,t_2=y\vee_{[0,t]}z$,
		 we have
		 $\mc{S}^{t_1}_{xy}\neq 0,\mc{S}^{t}_{t_1 z}\neq 0$
		iff $\mc{S}^t_{x t_2}\neq 0,\mc{S}^{t_2}_{yz}\neq 0$.
\end{enumerate}
Our main examples are the following:
\begin{enumerate}
	\item $\mc{S}=\{(x,y,t)\mid 
		t\in x\os{\circ}{\vee}y\}$.
	\item $\mc{S}=\{(x,y,t)\mid 
		t\in x\os{\circ}{\vee}y,r(t)=r(x)+r(y)\}$ 
		in the case of a local geometric lattice $(L,\le)$.
\end{enumerate}

For any $x'\leq x,y'\leq y$ there is a unique map
$$\gamma^{x'y'}_{xy}\colon \os{\circ}{\sup}(x,y)\to 
	\os{\circ}{\sup}(x',y')$$ 
such that $\gamma^{x'y'}_{xy}(t)\leq t$ for all $t\in \os{\circ}{\sup}(x,y)$. It is defined by
taking $\sup(x',y')$ in $L[0,t]$.
Note that $\mc{S}^t_{xy}\neq 0$ implies $\mc{S}^{t'}_{x'y'}\neq 0$.

\begin{dfn}
\label{dfn:L_algebra}
	An $\mc{S}$-\emph{partial}
	$(L,\leq)$-\emph{algebra}
	in $\ChF$ is a functor 
	$$A_{(-)}\colon (L,\leq)^{\op}\to \ChF$$ with morphisms
	$$m_{x y}\colon A_x\otimes A_y\to 
\bigoplus\limits^{t\in L}_{\mc{S}^{t}_{xy}\neq 0}A_t$$
	such that the following holds:
	\begin{enumerate}
		\item $m$ is associative, i.e.

	\begin{equation}\label{eq:dfn:associativity}
		\sum\limits^{t_1\in L}_{\mc{S}^t_{t_1 z},\mc{S}^{t_1}_{x y}\neq 0}
		m^{t}_{t_1,z}\circ (m^{t_1}_{x y}\otimes \id)=
		\sum\limits^{t_2\in L}_{\mc{S}^{t}_{x t_2},
		\mc{S}^{t_2}_{yz}\neq 0}
	m^{t}_{x, t_2}\circ(\id\otimes m^{t_2}_{yz})
	\end{equation}
	for all $t\in \os{\circ}{\sup}(x,y,z)$. 
	Here $m^t_{x y}$ are the components of $m_{x y}$.

		\item $m$ is functorial, i.e.\ for 
 for all $x'\le x,y'\le y$ the following square
			commutes:
			\begin{equation}\label{eq:dfn:L_algebra_functoriality}
\begin{tikzcd}[column sep = large, row sep = large]
	A_x\otimes A_y\arrow[r,"m"]
	\arrow[d,swap,"g^A_{x'x}\otimes g^A_{y'y}"]&	
	\bigoplus\limits^{t\in L}_{\mc{S}^t_{xy}\neq 0}A_t 
		\arrow[d,"g^A_{\gamma(t),t}"]\\
	A_{x'}\otimes A_{y'}\arrow[r,swap,"m"]
		&
		\bigoplus\limits^{t'\in L}_{\mc{S}^{t'}_{x'y'}\neq0 }A_{t'}.
\end{tikzcd}
\end{equation}
\item $m$ is unital, i.e. we have a morphism $1_{\ChF}\ar A_0$ such that
	the induced morphism $$m(1_{\ChF},-)\colon A_x\ar A_x$$ is identity.
\end{enumerate}

\end{dfn}
\begin{rmrk}
Note that in \eqref{eq:dfn:associativity} both sides contain at most
one term.
\end{rmrk}
\begin{dfn}
1. In the case $\mc{S}=\{(x,y,t)\mid t\in x\os{\circ}{\vee}y\}$ an $\mc{S}$-partial $(L,\le)$-algebra
will be called simply an $(L,\le)$-\emph{algebra}.

2. In the case $\mc{S}=\{(x,y,t)\mid t\in x\os{\circ}{\vee}y,r(t)=r(x)+r(y)\}$ and
$(L,\le)$ is locally geometric, an $\mc{S}$-partial $(L,\le)$-algebra
will be called a \emph{partial} $(L,\le)$-\emph{algebra}.
\end{dfn}

\begin{rmrk}
	Thus an $(L,\leq)$-algebra $A_\bullet$ in $\ChF$ 
	is provided with a richer structure then that of
	an unital $(L,\leq)$-object in $\Mon^{\circ}(\ChF)$.
	Such object will be called 
	$(L,\le)$-monoid in $\ChF$.
\end{rmrk}
\begin{dfn}
	The \textit{totalization} of an $(L,\leq)$-object 
	$A_\bullet$ in $\ChF$ is
	$\Tot(A_\bullet)=\oplus_{x\in L} A_x\in \ChF$.
\end{dfn}

The associativity in the definition of $(L,\leq)$-algebra in $\ChF$
is equivalent to requiring that $\Tot(A)$ with a product 
given by $m^t_{xy}$ should be an associative algebra.

\subsubsection{$(L,\le)$-chain algebras}
Assume $(L,\leq)$ is a local poset.
We define an $(L,\le)$-\textit{filtered complex} $K$ in $\ChF$
as a 
\begin{dfn}
	An $(L,\leq)$-\textit{filtered complex} $K$ in $\ChF$  
	is a data formed by
	\begin{enumerate}
		\item a collection of objects $K_x\in \ChF$ 
			with the inner differential $d_x$
			for all $x\in L$,
		\item a collection of morphisms
	$\partial_{yx}\colon K_x\to K_y[1]$ for all pairs $y\le x\in L$,
	such that $\partial_{xx}=d_x$.
	\end{enumerate}
	such that
	$$\sum\limits^{y,x}_{y\le x}\partial_{yx}$$
	defines a differential on 
	$K[0,z]:=\bigoplus\limits^{t}_{t\leq z\in L}K_t$ for each $z\in L$.
\end{dfn}
The morphisms $\partial_{yx}$ for $y<x$ and 
the differential $d=\sum \partial_{xx}$ 
are called the \textit{structure maps} of
$K$ and the \textit{inner differential} respectively.

Note that it is straightforward to define a \emph{morphism} between 
$(L,\le)$-objects (respectively $(L,\le)$-filtered complexes).

\begin{dfn}[Totalization]
	The \textit{totalization} of an $(L,\leq)$-filtered complex $K$
	is the complex $$\Tot(K)=\bigoplus\limits_{x\in L}K_x\in \ChF$$ with 
	the differential equal to the sum $\sum \partial_{yx}$.
\end{dfn}

\begin{rmrk}
Notice that the filtered complexes
$(K_x,V_{K_x})\in \ChF$ define 
an associated graded $(L,\leq)$-filtered complex 
$\Gr^\bullet_V K\in \Ch(\Bbbk-\mc{A})$.
In the case $\ChF=\ChF_{W}$ the structure maps 
$\Gr_W \partial_{yx}$ are \textit{trivial}
since $\partial_{yx}W^k(K_x)\subset W^{k}(K_y[1])=(W^{k+1}K_y)[1]$.
\end{rmrk}

Thus a 
$(L,\leq)$-filtered object is a $L$-graded complex with 
a perturbation of the differential which decreases the
$L$-grading.

\begin{dfn}[$(L,\leq)$-chains]
	An $(L,\leq)$-\textit{chain object}
	is an $(L,\leq)$-filtered object $K$ such that:
	\begin{enumerate}
		\item $\partial_{yx}\neq 0$ only in case $y<:x$ or $y=x$.
		\item $\partial:=\sum\limits_{y<:x} \partial_{yx}$ 
			is a differential.
	\end{enumerate}
\end{dfn}
Thus an $(L,\le)$-chain object is an $L$-graded complex 
together with an additional
\emph{combinatorial differential} $\partial:=\sum\limits_{y<:x}\partial_{yx}$ 
such that $[d,\partial]=0$.

Assume $(L,\leq)$ is a local sup-lattice.
\begin{dfn}[$(L,\leq)$-filtered algebra]
\label{dfn:L_filtered_algebra}
	An $(L,\leq)$-\textit{filtered algebra} in $\ChF$
	is an $(L,\leq)$-filtered complex $K$ with
	a collection of morphisms:
	$$m_{x,y}\colon K_x\otimes K_y\to 
	\bigoplus\limits^t_{t\in x\os{\circ}{\vee}y}K_t,$$
	with components $m^t_{x,y}$ such that 
	$m$ is associative and obeys the Leibniz rule, 
	i.e. $\Tot(K)$ defines a monoid in $\ChF$ 
	with the differential 
	$\partial+d$.
\end{dfn}
Similarly we define
$(L,\le)$-\emph{chain algebras}.  


\begin{dfn}
	An $(L,\leq)$-filtered algebra 
	is \textit{well-graded}, if 
	$L$ is graded and the multiplication satisfies
	$m^{t}_{xy}=0,t\in x\os{\circ}{\vee}y$ if 
	$r(t)\neq r(x)+r(y)$.
\end{dfn}

\subsubsection{Operations}
Clearly one has the notion of the \textit{tensor product} of 
two $(L,\leq)$-objects given by point-wise tensor product.
A direct check justifies the following definition.
\begin{dfn_prop}[Tensor product]\label{dfn:tensor_product}
	The \textit{tensor product} $A\otimes K$ of an $(L,\leq)$-complex 
	(algebra) $A$ and 
	an $(L,\leq)$-filtered complex (algebra) $K$ is an
	$(L,\leq)$-filtered complex (algebra) given by:
	\begin{enumerate}
		\item Components $(A\otimes K)_x=A_x\otimes K_x,x\in L$
		\item Structure maps $\partial^{A\otimes K}_{yx}:=
			g^A_{yx}\otimes \partial^K_{yx}$ for each $y<: x\in L$,
			as a morphism 
			$(A\otimes K)_x\to (A\otimes K)_{y}$.
	\end{enumerate}
\end{dfn_prop}
\begin{rmrk}\label{rmrk:lattices:product_with_partial}	
Note that in case $K$ is well-graded 
$(L,\le)$-chain algebra while $A$ is a partial algebra,
then $(A\otimes K,d)$ equipped with the inner differential 
$d$ only, is $(L,\le)$-chain algebra.
It is not true in general that $(A\otimes K,d+\partial)$
is an $(L,\le)$-chain algebra. 
It turns out, however, that in certain cases 
$(A\otimes K,d+\partial)$ admits 
a perturbation of the multiplication satisfying the Leibniz rule.
We will return to this point in \ref{sect:cech_model}.
\end{rmrk}

According to Proposition \ref{prop:vect_to_ch} we can tensor $A$ by an 
$(L,\leq)$-filtered complex in $\Ch(\Vect_{\Bbbk})$. 
The result will again be an $(L,\leq)$-filtered complex.

\
The following notion is clear.
\begin{dfn_prop}[Restriction]
	If $L\subset L'$ is a complete sublattice and $K'$ is 
	a $(L',\leq)$-filtered complex,
	then the \textit{restriction} $K'|_{L}$ is a
	$(L,\leq)$-filtered complex.
\end{dfn_prop}
	\
	Let $K$ be an $(L,\leq)$-chain complex
	and assume $f\colon (L,\leq)\to (L',\leq)$ is a contraction.
	\
	Let
	$$(f_*K)_{x'}=\Tot(K|_{f^{-1}x'}).$$
	Define the structure maps:
	$$\partial^{f_*K}_{y'x'}\colon (f_*K)_{x'}\to (f_*K)_{y'}$$ for 
	$y'<x'$
	equal to 
	$\sum\partial_{yx}$
	where the sum is over all pairs $y<x$ such that $f(x)=x',f(y)=y'$.
	
\begin{prop}
	The objects $(f_*K)_{x'},x'\in L'$ naturally form an 
	$(L',\leq)$-chain complex.
	There is a natural isomorphism $\Tot(f_*K)\simeq \Tot(K)$ of complexes. 
	In addition, if $f$ is a homomorphism and
	$K$ is $(L,\leq)$-chain algebra, then $f_*K$ is
	$(L',\leq)$-chain algebra.
\end{prop}
\begin{proof}
Each $(f_*K)_{x'}$ is a totalization. It is a complex
with its own inner differential $d'_{x'}$ given by:
$$d'_{x'}=\sum\limits^x_{x\in f^{-1}x'}d_x+
	\sum\limits^{y,x}_{y<x\in f^{-1}x'}\partial_{yx}=:d_{x'}+
	\partial_{x'}.$$ 

Let $\partial'=\partial^{f_*K}$ and $d'=\sum d'_{x'}$.
By the construction,
$$\left(\bigoplus\limits_{x'\in L'}(f_*K)_{x'},d'+\partial'\right)=
\left(\bigoplus\limits_{x\in L}K_x,d+\partial\right).$$

Thus $(d'+\partial')^2=0$. To show that $\partial'$ are the
structure maps of an $(L,\leq)$-chain complex,
it is enough to prove 
$[d',\partial']_{y'x'}=0$ for all $y'<x'$. Here and below $f_{yx}=f|_{yx}$ denotes
the $yx$-component of a morphism $f$ with the source and the target 
components given by $x,y$ respectively.
The restriction of $[d',\partial']_{y'x'}=
d'_{y'}\partial'_{y'x'}+\partial'_{y'x'}d'_{x'}$ onto its $yx$-component
for $f(y)=y',f(x)=x'$ is equal to
\begin{equation}\label{eq:yx_comp}
\left(d_y\partial_{yx}+\sum_{z}\partial_{yz}\partial_{zx}\right)+
\left(\partial_{yx}d_x+\sum_w \partial_{yw}\partial_{wx}\right)=
\sum_{z}\partial_{yz}\partial_{zx}+
\sum_w \partial_{yw}\partial_{wx}
\end{equation}
where $z$ runs through $\{z\mid y<z<x,f(z)=y'\}$ and
$w$ through $\{w\mid y<w<x,f(w)=x'\}$.
Then \eqref{eq:yx_comp} is equal to
$$
\partial^2|_{yx}-\sum_t \partial_{yt}\partial_{tx}=
-\sum_t \partial_{yt}\partial_{tx}
$$
where $t$ runs through the set 
\begin{equation}\label{eq:yx_comp_run_set}
	\{t\mid y<t<x, y'<f(t)<x'\}.
\end{equation}

Since $K$ has range $1$ and $f$ is a contraction,
$r(x')>r(y')+1$ implies $r(x)>r(y)+1$ and hence $\partial'_{y'x'}= 0$. 
Thus it is enough to cover the case $x':>y'$.
In this case \eqref{eq:yx_comp_run_set} is empty, therefore 
\eqref{eq:yx_comp}  vanishes and hence $f_*K$ is an $(L,\leq)$-chain complex.

If $K$ is $(L,\leq)$-chain algebra and $f$ is a homomorphism of
local lattices, then 
$$\Tot\left(f_*(K)|_{L[0,t]}\right)=\Tot\left(K|_{f^{-1}L[0,t]}\right)$$
is a dg-algebra with the multiplication 
preserving the $L$-grading. It defines the desired $(L',\leq)$-chain algebra 
structure on $f_*K$.
\end{proof}
The above proposition justifies the following
\begin{dfn}[Pushforward]\label{dfn:lattice:direct_image}
The \textit{pushforward} of an $(L,\leq)$-chain complex 
(respectively chain algebra) $K$ under a
contraction (respectively contraction homomorphism) 
$f\colon (L,\leq)\to (L',\leq)$ 
is the $(L',\leq)$-chain complex (respectively chain algebra)
$f_*K$.
\end{dfn}
\
Since the composition of two contraction maps is a contraction,
the pushforward map satisfies $$f_*\circ g_*=(f\circ g)_*$$ i.e.\ 
the construction is functorial.
Note that if $\pi\colon (L,\leq)\to(pt,\leq)$ is the projection to the point, then
$$\Tot(K)=\pi_*K.$$

\begin{dfn}[Pullback]\label{dfn:lattice:pullback}
	If $f\colon (L,\leq)\to (L',\leq)$ is an order preserving map
	and $A'$ is an $(L',\leq)$-object, then
	its \textit{pullback} under $f$ is the $(L,\leq)$-object $f^*A'$
	with values $(f^*A')_{x}:=A'_{f(x)}$ and the obvious
	structure maps.
\end{dfn}
The following is clear.
\begin{prop}
	If $f$ above is a homomorphism 
	and $A'$ is an $(L',\leq)$-algebra 
	then $f^*A'$ has
	a natural structure of an $(L,\leq)$-algebra.
	Similarly, if $f$ is a homomorphism of locally geometric lattices and
	$A'$ is partial, then $f^*A'$ is partial.
\end{prop}

\begin{rmrk}
	If $f$ preserves the rank function, then 
	the pushforward of a well-graded 
	$(L,\leq)$-chain algebra is well-graded. 
	Similarly a pullback of a partial $(L,\le)$-algebra 
	is partial if $r(f(x))=r(x)$ for all $x\in L$.
	We defined partial algebras only for locally geometric lattices, 
	so if $f$ is a homomorphism of locally geometric lattices,
	 $r(f(x))=r(x)$ holds automatically by 
	Remark \ref{rmrk:lattices:geom_homomorphism_preserves_r}.
	Note that there is no general recipe for pulling back
	$(L,\leq)$-chain objects even for homomorphism of 
	lattices, e.g.~the restriction is well defined only
	for complete sublattice.
\end{rmrk}

\begin{lemma}\label{lemma:lattice:on_endomorphism}
	If $\phi\colon L\to L$ is an order preserving map such that
	$\phi(x)\geq x$ for all $x\in L$, and $A$ is an $(L,\leq)$-object,
	then there is a natural homomorphism 
	$u_\phi\colon \phi^*A\to A$.
\end{lemma}
\begin{proof}
	Define $u_\phi\colon (\phi^*A)_x=A_{\phi(x)}\to A_x$ as 
	the corresponding structure map. Then for any pair 
	$x'\leq x$ the following diagram commutes:
	\begin{tikzcd}
		\phi^*(A)_x\arrow[r,"u_\phi"]\arrow[d]& 
			A_x\arrow[d]\\
		\phi^*(A)_{x'}\arrow[r,"u_\phi"] &
			A_{x'}
	\end{tikzcd}.
\end{proof}

Let $p\colon (L,\leq)\to (L',\leq)$ be a contraction homomorphism. 
Consider an $(L',\leq)$-complex $A'$ and
an $(L,\leq)$-chain complex $K$.
We have
\begin{lemma}[Projection formula]\label{lemma:lattice:projection_formula}
	There is a natural
	isomorphism of $(L,\leq)$-chain complexes:
	$$p_*(p^*A'\otimes K)\simeq A'\otimes p_*K.$$
	This isomorphism is an isomorphism of $(L,\le)$-chain algebras  
	in case when $A'$ is an 
	$(L',\le)$-algebra
	and $K$ is an $(L,\le)$-chain algebra.
\end{lemma}
\begin{proof}
	We have
	$$(p_*(p^*A'\otimes K))_{x'}=
	\bigoplus\limits^x_{x\in p^{-1}x'}(p^*A)_x\otimes K_{x}=
	\bigoplus\limits^x_{x\in p^{-1}x'}A_{x'}\otimes K_x=
	(A'\otimes p_*K)_{x'}.$$
	Then it is straighforward to check that
	this identification commutes with the structure maps
	$$\partial'_{y'x'}\colon (p_*(p^*A'\otimes K))_{x'}\to 
	(p_*(p^*A'\otimes K))_{y'}$$
	and
	$$\partial''_{y'x'}\colon 
		(A'\otimes p_*K)_{x'}\to (A'\otimes p^*K)_{y'}$$
	The multiplicativity is clear.
\end{proof}

Let $\pi:(L,\leq)\to (pt,\leq)$ be the trivial morphism.
Consider an $(L,\leq)$-object (algebra) $A$ 
	in $\ChF$ and an $(L,\leq)$-chain object 
	(algebra) $K$ in $\ChF$ or in $\Ch(\Vect_{\Bbbk})$ (thanks to Proposition \ref{prop:vect_to_ch}).
	
\begin{dfn}[Convolution]\label{dfn:convolution}
	The \textit{convolution} $A*K$ of $A$ and $K$ is 
	the complex (algebra) equal to
	$$\pi_*(A\otimes K)=\Tot(A\otimes K)\in \ChF.$$
\end{dfn}
Thus $\Tot(K)$ is the convolution of $K$ with the 
trivial algebra 
(i.e.~ the algebra with the value at all vertices equal to 
the unital object $\mathrm{1}_C\in \ChF$), and
$$A*K=\Tot(A\otimes K)=
	\left(\bigoplus\limits_{x\in L}A_x\otimes K_x,d+\partial\right).$$
Assume $K$ is an $(L,\leq)$-chain complex in $\ChF$.
\begin{dfn}[Rank filtration]\label{dfn:rank_filtration}
	The complex $\Tot(K)\in \Ch(\Bbbk-\mc{A})$ admits 
	an increasing filtration 
$W'_{p}\Tot(K):=\bigoplus\limits^{x\in L}_{r(x)\leq p}K_x$.
	The filtration $W'$ is called the \textit{rank filtration}.
\end{dfn}
If $K$ is well-graded, then $\Tot(K)$ becomes
a bicomplex with $W'$ being a split
multiplicative filtration.
Note that for a contraction order-preserving map $g\colon (L,\le)\ar (L',\le)$
an $(L,\le)$-chain complex $K$ we have a natural $W'$-filtered morphism
\begin{equation}
	(\Tot(K),W')\ar (\Tot(g_* K),W').
\end{equation}

\subsection{\texorpdfstring{$(L,\leq)$}{L}--mixed Hodge complexes}
\label{sect:lattices:mhc}
In the sequel we use the lower bullet to indicate 
the fact that complexes we consider are graded by the elements of a lattice $L$.

By Remark \ref{rmrk:ChF} there is a functor 
$$\MHC(X)\ar \ChF_{F,W}(\Sh_X)\ar \ChF_{F}(\mb{C}-\Sh_X)\times \ChF_{W}(\mb{Q}-\Sh_X)$$
which associates to an MHC $(K,W,F)$ the corresponding pair 
$((K_{\mb{C}},F), (K_{\mb{Q}},W))$.
We define
an $(L,\leq)$-object $(K,F,W)_\bullet$ in $\MHC(X)$
as a pair of $(L,\leq)$-objects $(K_\mb{C},F)_\bullet$ in 
$\ChF_{F}(\mb{C}-\Sh_X)$ 
and $(K_{\mb{Q}},W)_\bullet$ in $\ChF_{W}(\mb{Q}-\Sh_X)$ respectively which are 
compatible in the sense of Definition \ref{dfn:mhc}.
Similarly we define $(L,\leq)$-chain complex in $\MHC(X)$.


Assume $(L,\leq)$ is a graded local sup-lattice, e.g.\ that it is
given by an arrangement.
Here we prove two lemmas which allow us to produce mixed Hodge complexes 
of sheaves starting from an $(L,\leq)$-chain object in $\MHC(X)$.

Assume $(A,F,W)_\bullet$ is an $(L,\leq)$-object in $\MHC(X)$ and 
$K$ is an $(L,\leq)$-chain object in $\Ch(\Vect_{\mb{Q}})$.
\begin{lemma}\label{lemma:convolution_is_hodge}
	If $A$ and $K$ are as above, then
	$A\otimes K$ is an $(L,\leq)$-chain object in $\MHC(X)$.
	The convolution $(A*K,F,W)$ is
	a mixed Hodge complex of sheaves.
\end{lemma}
\begin{proof}
	Let $K'=A\otimes K$. By definition, we have 
	$\partial_{yx}\colon K'_x\to K'_y[1]$.
Notice that the filtered complexes
$(K_x,W_{K_x})$ in $\ChF_{W}(\Sh_X)$ define 
an associated graded $(L,\leq)$-filtered complex 
$\Gr^\bullet_W K'\in \Ch(\Sh_X)$.
Then the structure maps 
$\Gr_W \partial_{yx}$ are \textit{trivial}
since $\partial_{yx}W^k(K'_x)\subset W^{k}(K'_y[1])=(W^{k+1}K'_y)[1]$.
In particular $\Gr^n_W K$ is an $(L,\leq)$-pure Hodge complex of weight $-n$ 
(recall our indexing conventions in Section~\ref{sec:notation_etc}).
So $(K',F,W)_\bullet$ is an $(L,\leq)$-object in $\MHC(X)$.

In particular $(\Tot(K'),F,W)$ defines an object in $\MHC(X)$.
\end{proof}
\begin{dfn}
The mixed Hodge complex $(A*K,F,W)$ is called
the \textit{natural} mixed Hodge complex.
\end{dfn}

Now assume $(A,F,W)_\bullet$ is an $(L,\leq)$-pure Hodge 
complex of weight $0$ and
$K$ is an $(L,\leq)$-chain object in $\Ch(\Vect_{\mb{Q}})$ 
such that $H^*(K_x)$ is concentrated in degree $-r(x)$.
Recall that $A*K$ admits the rank filtration $W'$
(definition \ref{dfn:rank_filtration}).
\begin{lemma}\label{lemma:two_weight_filtrations}
	For $A$ and $K$ as above,
	$(A*K,F,W')$ is a mixed Hodge complex of sheaves.
	The filtrations induced by $W$ and $W'$ coincide on
	$\mb{H}^*(A*K)$.
\end{lemma}
\begin{proof}
1.
	First we shall prove that $\left(\Gr^p_{W'}(A*K),F\right )$ is
	a pure Hodge complex
	of weight $-p$.
	Note that 
	$$\Gr^p_{W'}(A*K)=\bigoplus\limits_{r(x)=-p}^{x\in L}A_x\otimes K_x$$
	is a sum of complexes, and each term is the tensor product 
	$A_x\otimes K_x$ of complexes. Thanks to the natural quasi-isomorphism 
	$$K_x \sim H^{-r(x)}(K_x)[r(x)]$$ we have
	$\mb{H}^n(A_x\otimes K_x)\simeq \mb{H}^{n-p}(A_x)\otimes H^{p}(K_x)$.
	Since $A_x$ has weight $0$, 
	the Hodge filtration $F$ is $(n-p)$-conjugate
	on $\mb{H}^n(\Gr^p_{W'}(A*K))$, 
	thus $(A*K,F,W')$ is a mixed Hodge complex.

2.
	Unfortunately we do not have a direct comparsion map between 
	$W$ and $W'$ on the level of complexes.
	To prove that $W$ and $W'$ coincide on 
	$H^*:=\mb{H}^*(A*K)$ 
	is to prove $$\Gr^p_W \Gr^{p'}_{W'}H^*=0$$ unless $p= p'$.
	Note that the induced filtration $W$ on 
	$\Gr^p_{W'} H^n={}_{W'}E^{pq}_2$
	coincides with $W_{dir}$ on ${}_{W'}E^{pq}_2$.
	by lemma \ref{lemma:three_filtrations}.
	Namely, $W_{dir}$ on 
	$E^{p',n-p'}_2(A*K,W')=\Gr^{p'}_{W'}(H^n)$
	is equal to $W_{ind}$ if $d_1$ is 
	$W_{dir}$-strict on $E_1(A*K,W')$.
	Thus it is enough to show that 
	$${}_{W}\Gr^p E^{p',n-p'}_1(A*K,W')=0$$ 
	unless $p=p'$. 
	The Hodge filtration $F$ on 
	$$V^n:=E^{p',n-p'}_1(A*K,W')=\mb{H}^n(\Gr^{p'}(A*K,W'))$$ is 
	$(n-p')$-conjugate by the previous remarks.
	So, since $(\Gr^{p'}_{W'}(A*K),F,W)$ is a mixed Hodge complex
	and the corresponding mixed 
	Hodge structure
	$(V^n,F,\Dec W)$ admits non-trivial weights
	only in degree $n-p'$, the group 
	$$\Gr^{\Dec\ W'}_{n-p} V^n=
	\Gr^{p-n}_{\Dec W'} V^n=\Gr^p_{W'}V^n$$
	vanishes unless $p'=p$ (see remark \ref{rmrk:decalage}).
\end{proof}

\begin{dfn}\label{dfn:rank_hs}
	The mixed Hodge complex $(A*K,F,W')$
	is called the \textit{rank} Hodge complex.
\end{dfn}
Hence the natural and the rank mixed Hodge structures 
$(\mb{H}^*(A*K),F,\Dec W)$,
$(\mb{H}^*(A*K),F,\Dec W')$ coincide.

Assume $K$ is an $(L,\le)$-chain of MHC such that $K_x$ is pure of weight $-r(x)$.
\begin{lemma}
	Given an order-preserving map $g\colon (L,\le)\ar (L',\le)$ such that
	$r(g(x))=r(x)$ for all $x\in L$,
	the direct image $g_*K$ is an $(L',\le)$-chain of MHC such that $(g_*K)_{x'}$
	is pure of weight $-r(x')$.
\end{lemma}
\begin{proof}
	It enough to note that $(g_*K)_{x'}$ is a direct sum of pure MHC's of weight $-r(x')$. 
\end{proof}
\subsection{Orlik-Solomon algebras}
\label{sect:lattices:os}
In this section we assume $(L,\leq)$ to be a local graded lattice.
Following Orlik-Solomon's work \cite{OS} we define:
\begin{dfn}[OS-complex]\label{dfn:os_complex}
	The OS-complex $\mc{M}$ (over $\Bbbk$) of a local graded lattice
	$(L,\leq)$ is an $(L,\leq)$-chain complex in $\Ch(\Bbbk-\Mod)$
        such that the following is true:
	\begin{enumerate}
		\item $\mc{M}_0=\Bbbk[0]$.
		\item $\mc{M}_s$ is a free $\Bbbk$-module sitting 
                in degree $\ -r(s)$.
		\item $\Tot(\mc{M}|_{[0,s]})$ is acyclic for 
			all $s>0\in L$.
			In other words, 
			the negatively graded complex 
			$$\bigoplus\limits_{t\leq s}\mc{M}_t$$
			with $\partial$ defined by:
			$$\partial v_y=\sum\limits^x_{y:>x} \partial_{xy}v_y$$
			is an acyclic complex for each $s>0\in L$.
	\end{enumerate}
\end{dfn}

\begin{rmrk}
The OS-complex $\mc{M}$ is the categorification of 
the M\"obius function $\mu$ of $L$: 
$$\mu[0,s]=(-1)^{r(s)}\rk_{\Bbbk} \mc{M}_s.$$
\end{rmrk}
Assume $(L,\le)$ is a
locally graded poset.
We will denote by $[L,\Mu]$ the data given by a poset $(L,\le)$ and
an OS-complex $\Mu$.

We have the following categorical characterization of the OS-complex $\Mu$:
\begin{prop}\label{prop:os_complex_categorically}
	Given $[L,\Mu]$, for an $(L,\le)$-chain complex $K$ such that $K_x\in \Ch^{\ge -r(x)}$,
there is a natural isomorphism  $$\Hom(K,\Mu)\ar[\sim]\Hom_{\Ch(\Bbbk-\Mod)}(K_0,\Mu_0).$$
\end{prop}
\begin{proof}
It is enough to extend any $\phi_0\colon K_0\to \Mu_0$ to 
a morphism $\phi_\bullet\colon K\ar \Mu$ of
$(L,\le)$-chain complexes.
Assume $K_x\ar \Mu_x$
for any $x$ with $r(x)<r$ is given.
Let $y\in L$ be such that $r(y)=r$.
We have a morphism of complexes $\phi_{<y}\colon \Tot(K_{<y })\ar \Tot(\Mu_{<y})$
where $\Tot(K_{<y})=\oplus_{t<y}K_t$.
In the following diagram
\begin{equation}
\begin{tikzcd}
K_y\arrow[r,dashed,"\phi_y"]\arrow[d,"\partial"]& \Mu_y\arrow[d,"\partial",hook]\\
\oplus_{x<:y}K_x[1]\arrow[r]& \oplus_{x<:y}\Mu_x[1],
\end{tikzcd}
\end{equation}
by the definition of the $(L,\le)$-chain complex $K$,
the image of the left vertical arrow in $\Tot(K_{<y})$ is 
closed with respect to
the combinatorial differential.
Since $\Mu_{\le y}$ is acyclic, 
there is a unique $\phi_y\colon K_y\ar \Mu_y$ making the above diagram commutative.
By the assumption $K_y\in \Ch^{\ge -r(y)}$, hence $\phi_y$
commutes with the inner differential.
\end{proof}

\begin{dfn}\label{dfn:os_algebra}
	An OS-algebra of a local sup-lattice $(L,\leq)$ is
	an OS-complex $\mc{M}$ with a structure of an $(L,\leq)$-chain
	algebra. 
\end{dfn}
By dimension reasons the multiplication $m^t_{xy}\colon \Mu_x\otimes \Mu_y\ar \Mu_t$ in an 
OS-algebra $\Mu$ is non-zero only on independent components, i.e.\ $m^t_{xy}=0$
if $t\not\in x\os{\circ}{\vee}y$ or $r(t)\neq r(x)+r(y)$.
In other words $\Mu$ is a well-graded $(L,\le)$-chain algebra.

\begin{thrm}[Orlik-Solomon; {\cite[\S 2]{OS}}]
\label{thrm:os_existence}
	If $(L,\leq)$ is a locally geometric lattice, then 
	there is a natural OS-algebra 
	$\OS(L)$ over $\mb{Z}$ with underlying lattice $(L,\le)$.
\end{thrm}
\begin{proof}
        For all $x\in L$, $L([0,x])$ is geometric. 
        First we assume the construction $OS(L[0,x])$ is known.

        Following \cite{JCW} we put
	$\OS(L)_x:=\OS(L[0,x])_x$. 
        Then $\Tot(\OS(L)|_{[0,x]})=\Tot(\OS([0,x]))$ is acyclic
        for all $x>0\in L$. Take $x,y\in L$ and
        and $t\in x \os{\circ}{\wedge}y$.
        Define $m^t_{xy}\colon \OS(L)_x\otimes \OS(L)_y\to \OS(L)_t$
        to be equal to the multiplication 
        $\OS([0,t])_x\otimes \OS([0,t])_y\to \OS([0,t])_t$. 

        It is clear that $m^t_{xy}$ defines an associative structure.
        The functoriality in Definition \ref{dfn:L_algebra}
        follows by naturality of $\OS(L[0,x])$. Namely the inclusion
        $L[0,x]\subset L[0,x']$ induces a natural isomorphism 
	$\OS(L[0,x])\simeq \OS(L[0,x'])|_{[0,x]}$.

        It remains to provide a natural construction of $\OS(L)$ in case 
	of a geometric lattice
        $(L,\le)$. We refer the reader to the original paper \cite[\S 2]{OS}
        and survey \cite{Yu}. 
        
        Below we sketch the construction following \cite{OS}
         (see also Remark \ref{rmrk:os_over_z}). 
    
        Set $\OS(L)=\Lambda^*\langle d\nu_i,i\in \Atoms(L)\rangle/\mc{J}$
	where $\mc{J}$ is an ideal in the Grassmann algebra that
	we will now descibe.
	The \emph{Kozsul differential} $\partial$ on the Grassmann algebra 
	is equal to $\sum\limits_i \iota_i$.
	Let $\langle dep\rangle$ denote the ideal 
        equal to the linear $\mb{Z}$-span by 
	$\nu_I$ where
	$I\subset \Atoms(L)$ is dependent, i.e.~$r(\sup(I))<|I|$.
	Then as an ideal, $\mc{J}$ is generated by 
	$\partial\langle dep\rangle$.
	One can check that $\mc{J}=
		\langle dep\rangle+\partial\langle dep\rangle$
	as vector spaces and that $\mc{J}$ is $L$-graded.
	Considering the complex $\Tot(\OS(L[0,x]))$,
 we have $[\partial,d\nu_i\wedge-]=\id$ for any $i<x\in L$, hence 
    $(\OS(L[0,x]),\partial)$ is acyclic unless $x=0\in L$.
\end{proof}

\begin{prop}[Functoriality of OS-complexes]\label{prop:morphism_of_OS_complexes}
	Let $[L,\Mu],[L',\Mu']$ be locally graded posets with OS-complexes
	$\Mu,\Mu'$ respectively.
	For a rank preserving map $g\colon (L,\le)\ar (L',\le)$,
	there is a unique homomorphism $g_*\Mu\ar \Mu'$ of $(L',\le)$-chain
	complexes inducing $\id\colon g_*\Mu_0=\Bbbk\to \Mu_0=\Bbbk$.
\end{prop}
\begin{proof}
	By Proposition \ref{prop:os_complex_categorically}, 
	$\Hom(g_*\Mu,\Mu')=\Hom((g_*\Mu)_0,\Mu'_0)$.
	Clearly $(g_*\Mu)_0=\Mu'_0=\Bbbk$.
\end{proof}
In particular, if $(L,\le)$ admits an OS-complex $\Mu$, then it is
unique up to a canonical isomorphism.
One can say that  $g\colon L\ar L'$ induces a natural \emph{morphism} 
$[L,\Mu]\ar[] [L',\Mu']$.

\begin{cor}\label{cor:morphism_of_OS_algebras}
	Assume $g\colon L\ar L'$ is a 
	homomorphism of locally geometric lattices.
	Then there is a natural homomorphism of $(L',\le)$-chain algebras
	$g_*\OS(L)\ar \OS(L')$.
\end{cor}
\begin{proof}
	It is sufficient to consider the case of globally geometric lattices.
	We have the induced homomorphism 
	$G\colon \Lambda\langle \Atoms(L)\rangle\ar[g] 
		\Lambda\langle \Atoms(L')\rangle.$
	Since $g$ maps $\Atoms(L)$ to $\Atoms(L')$ 
	(see Remark \ref{rmrk:lattices:geom_homomorphism_preserves_r}), 
	$G$ commutes with the Koszul differential $\partial$. 
	Following the proof of Theorem \ref{thrm:os_existence}, let 
	$\mc{J}=\langle dep\rangle+\partial\langle dep\rangle$, and 
	similarly for $\mc{J}'$.
	On the level of algebras we have 
	$$\Tot(g_*\OS(L))=\Tot(\OS(L))=\Lambda\langle \Atoms(L)\rangle/\mc{J},$$
	and $\Tot(\OS(L'))=\Lambda\langle \Atoms(L')\rangle/\mc{J}'$.
	Clearly $g$ preserves dependent subsets of atoms, hence
	$g\mc{J}\subset \mc{J}'$. 
	We obtain a homomorphism of cdga's 
	$G\colon \Tot(g_*\OS(L))\ar \Tot(\OS(L'))$.
	On the other hand, $G$ is $L'$-graded, so
	it follows that $G$ is induced by
	a morphism of $(L',\le)$-chain algebras 
	$G_\bullet\colon g_*\OS(L)\ar \OS(L')$. 
	Since $G_\bullet$ satifies $G_0=\id_{\Bbbk}$, this
	completes the proof.
\end{proof}
Hence a homorphism $g\colon (L,\le)\ar (L',\le)$ of locally geometric lattices
induces a \emph{multiplicative} morphism $[L,\OS(L)]\ar[] [L',\OS(L')]$.

\begin{exmpl}
\begin{enumerate}
	\item Consider an arrangement $(X,Z_*)$ from 
		Definition \ref{dfn:arrangement}
		and let $\mc{C}=2^N$ to be the cubical poset
		of all subsets in the index set $N$ parametrizing $Z_*$.
		The corresponding OS-algebra $\mc{M}=\OS(\mc{C})$ is the 
	Grassmann algebra
	generated by free variables 
	$d\nu_*,\deg(d\nu_*)=-1$ corresponding to $Z_*$.
	The differential $\partial=\sum\limits_i \iota_i$ is 
	the full Koszul differential. 
	The algebra is clearly $L$-graded and 
	$L$-filtred with respect to the differential.
	To simplify the notation, let us order $\Atoms(\mc{C})$. Then 
	the vertex 
	$I\in \mc{C}$ of the cube
	corresponds 
	to the free module $\mb{Z}\cdot d\nu_I$
	generated by the monomial
	$d\nu_I=
	d\nu_{I_1}\wedge\ldots\wedge d\nu_{I_k}\in \mc{M}_I$ 
	of degree $-|I|$
	for
	$I_1<\ldots<I_k$ that make up $I$. 
	The product $d\nu_I\wedge d\nu_J$
	is zero if $I\cap J\neq \emptyset$, otherwise it is given
	by the usual Koszul rule $d\nu_I\wedge d\nu_J=\pm d\nu_{IJ}$.
	
	Assume $X$ admits an action of a group $G$, such that $g(Z_i)$
	is an element of arrangement for all $i\in N$ and $g\in G$, 
	i.e.\ we can write $g(Z_i)=Z_{g(i)}\subset X$, where $G\ar \Sigma(N)$
	is the corresponding representation in the permutation group.
	Then $G$ acts on the poset $\mc{C}$, and hence on $\OS(\mc{C})$
	providing a morphism $g_*(\OS(\mc{C}))\ar \OS(\mc{C})$ 
	which is given by Corollary \ref{cor:morphism_of_OS_algebras}.
	(In Subsection \ref{sect:lattices:arrangement_poset} 
	we say that the corresponding cubical arrangement 
	poset $[X,\mc{C}]$ is $G$-equivariant.)
	Clearly the construction of $\OS(\mc{C})$ and of
	the morphism $g_*\OS(\mc{C})\ar \OS(\mc{C})$ is independent 
	of the order on $N$.

\item For a topological space $X$, the intersection poset $\Delta$ of 
	all diagonals $\Delta_*\subset X^n$
	is a geometric lattice with atoms $\Delta_{\{ab\}}$ given 
	by the diagonals $\{(x_1,\ldots,x_n)\in X^n\mid x_a=x_b\}$.
        The complement is the usual configuration space $F(X;K_n)$,
        where $K_n$ is a complete graph on $n$ vertices.
	The maximal element $\infty\in \Delta$ corresponds to the diagonal
	$x_1=\ldots=x_n$ and
	 the minimal element $0\in \Delta$ is $X^n$.
	Then OS-algebra $\mc{M}=\OS(L)$ 
	is recovered by $\Delta$-graded algebra
	$$\Tot(\OS(L))=\Lambda\langle \Delta_{\{ab\}}\mid
            a\neq b\in \{1,\ldots ,n\}\rangle/\mc{J},$$
	where $\deg(\Delta_{\{ab\}})=-1,d\Delta_{\{ab\}}=1$, and $\mc{J}$ is 
	an ideal given by Arnold's relations: 
		$$\Delta_{\{ab\}}\Delta_{\{bc\}}+\Delta_{\{bc\}}\Delta_{\{ca\}}+
			\Delta_{\{ca\}}\Delta_{\{ab\}}=0.$$
	This algebra admits a natural 
	$\Delta$-grading with pieces 
	$\mc{M}_s,s\in \Delta$ spanned by the
	Grassmann monomials $\Delta_{\{a_1 b_1\}}\wedge\ldots\wedge \Delta_{\{a_k b_k\}}$ 
	such that $\Delta_{\{a_1 b_1\}}\cap\ldots\cap \Delta_{\{a_k b_k\}}=\Delta_s\subset X^n$. 
	This defines a $(\Delta,\leq)$-chain algebra, 
	which we call $\OS(\Delta)$.
        In Section \ref{sect:mv_totaro} we generalize
        this description and obtain a spectral sequence
        converging to the cohomology of a chromatic configuration space $H^*(F(X;G))$. 
\end{enumerate}
\label{exmpl:os_algebras}
\end{exmpl}

\begin{dfn}\label{dfn:cubical_lattice_of_L}
	The {\it locally cubical lattice} of $(L,\leq)$ is 
	the poset $(\mc{Q},\leq)$ formed by pairs 
	$\tilde{I}=(x,I)$ for $I\subset \Atoms(L)$ such that 
	$x\in \os{\circ}{\sup}(I)$.
	Define $(x,I)\leq (y,J)$ iff $x\leq y$ and $I\subset J$.
\end{dfn}

In particular the poset $(\mc{Q},\le)$ 
is a locally geometric lattice.
If $(L,\leq)$ is a sup-lattice, then 
$\mc{Q}$ is the cubical lattice $(2^{\Atoms(L)},\subset)$.

For $i\in \Atoms(L)=\Atoms(\mc{Q})$ we write $i\in \tilde{I}=(x,I)$
if $i\in I$.
We have a natural map 
$p\colon (\mc{Q},\leq)\to (L,\leq)$ 
given by $p(\tilde{I})=x$.

\begin{prop}\label{prop:p_contraction_homomorphism}
	The map $p$ is a contraction homomorphism of local sup-lattices.
\end{prop}
\begin{proof}
	One can check that $(x,I)<:(y,J)$ iff $|J-I|=1$.
	Thus $r(x,I)=|I|$ defines a grading and $p$ is contractible.
	Then $(x,I)\vee (y,J)=
		(x\vee y\in L[0,t],I\cup J)\in [0,(t,K)]$ for
	any $(t,K)$ with $(x,I),(y,J)\leq (t,K)$. 
\end{proof}

Assume $(L,\leq)$ is locally geometric lattice.
Let $\OS(\mc{Q})$ denote the natural OS-algebra of 
$\mc{Q}$. 
\begin{rmrk}\label{rmrk:os_of_locally_cubical}
If $L$ is geometric, $\OS(\mc{Q})$ is given by the $L$-grading components 
of 
the usual Grassmann algebra
$\Lambda\langle d\nu_i,i\in \Atoms(L)\rangle$ where $\deg(d\nu_i)=-1$. 
The Koszul differential $\sum\limits_{i\in \Atoms(L)}\iota_i$ defines 
the structure maps of $\OS(\mc{Q})$ (Example \ref{exmpl:os_algebras}).

In case $L$ is locally geometric, the restriction $\OS(\mc{Q})|_{[0,\tilde{I}]}$
is determined by $\OS(\mc{Q}[0,\tilde{I}])$. 
Namely, similarly to Example \ref{exmpl:os_algebras} 
we can chose an order on $\Atoms(L)$, then
$\OS(\mc{Q})_{\tilde{I}}=\mb{Z}\cdot d\nu_{\tilde{I}}$, where
$d\nu_{\tilde{I}}$ is a Grasmmann monomial $d\nu_{I}$ of degree $-|I|$ equipped with
a label $x\in \os{\circ}{\sup}(I)\in L$.
The product structure is given by 
$d\nu_{\tilde{I}}\wedge d\nu_{\tilde{J}}=
\sum^{\tilde{K}\in\os{\circ}{\sup}(\tilde{I},\tilde{J})}_{|K|=|I|+|J|}\pm d\nu_{\tilde{K}}$,
where the sign is determined by the shuffle $(I,J)$ in $K=I\sqcup J$.
Similarly to the cubical case we have
$\partial^{\mc{Q}}_{\tilde{I}-\{i\},\tilde{I}}=\iota_i$.
\end{rmrk}

The following is an extension to the local case 
of the notion of the relative atomic
complex.
\begin{dfn}[Atomic complex; \cite{FY}]\label{dfn:atomic_complex}
	The \textit{atomic complex} $\mc{D}$ of $(L,\leq)$ is an 
	$(L,\leq)$-chain algebra equal to the direct image 
	$p_*\OS(\mc{Q})$ under the natural map
	$p\colon \mc{Q}\to L$.
\end{dfn}

Explicitely $(\mc{D}_x,\partial')$ is the complex spanned by 
Grassman monomials $d\nu_{\tilde{I}},\tilde{I}\in \mc{Q}$.
The differential $\partial'$ acts as follows:
$$\partial'd\nu_{\tilde{I}}=
\sum\limits^{i\in \tilde{I}}_{\sup(\tilde{I}-\{i\})=x}
\iota_i d\nu_{\tilde{I}}.$$

The structure morphisms 
$\partial^{\mc{D}}_{yx}\colon \mc{D}_x\to \mc{D}_y[1],x:>y$
are equal to the remaining part of the Koszul differential.
The multiplication is clear.

We will need the following homological interpretation of $\OS(L)$.
Consider $\OS(L)$ as an $(L,\leq)$-chain algebra with trivial inner 
differentials $(\OS(L)_x,d_x=0)$. 
There is a map $\mc{D}\to \OS(L)$ of $(L,\leq)$-chain
algebras induced by the identity: locally it sends
$d\nu_I$ to $d\nu_I\in \OS(L)=\Lambda^*\langle d\nu_i\rangle/\mc{J}$, and
since $\partial'd\nu_I=0\in \OS(L)$, it is a morphism of
$(L,\leq)$-chain complexes.
We have the following formality result for the atomic complex $\mc{D}(L)$:
\begin{thrm}[Feichtner-Yuzvinsky; {\cite[Theorem 3.1]{FY}}]
\label{thrm:atomic_complex}
	The morphism $\mc{D}=p_*\OS(\mc{Q})\to \OS(L)$
	is a quasi-isomorphism of $(L,\leq)$-chain algebras.
	So the atomic $(L,\le)$-chain algebra $\mc{D}$ is formal, 
	i.e.\ one has a natural 
	quasi-isomorphism $\mc{D}_x\ar \OS(L)_x\in \Ch(\Bbbk-\Mod)$
	commuting with the structure maps $\partial_{yx}$.
\end{thrm}
In other words there is a natural morphism 
$p\colon [\mc{Q},\OS(\mc{Q})]\ar[] [L,\OS(L)]$ inducing a quasi-isomorphism 
$p_*\OS(\mc{Q})\ar\OS(L)$.

\begin{rmrk}\label{rmrk:os_over_z}
For a locally geometric lattice $(L,\le)$, the fact that 
$H^*(\mc{D}(L)_x;\mb{Z})$ is free and 
concentrated in degree $-r(x)$
follows also from \cite[Theorems 3.1, 4.1]{Fo}. So
the spectral sequence $E^{pq}$ of the bicomplex 
$(\Lambda\langle d\nu_I\rangle,\partial'+\partial'')$ 
defined by the rank grading obeys $E^{pq}_1=0$ for all $q\neq 0$. 
Hence $E^{pq}_2=E^{pq}_\infty$
and acyclicity of the Koszul complex
implies that the totalization of an $(L,\le)$-chain algebra 
$H^*(\mc{D}(L)_\bullet;\mb{Z})\simeq E^{**}_1$ is acylic.
Applying this for intervals $L[0,x]$ for all $x\neq 0\in L$ we obtain that 
$H^*(\mc{D}(L)_\bullet;\mb{Z})$ satisfy the definition 
of an OS-complex. The evident multiplicative structure provides a different 
from Theorem~\ref{thrm:os_existence} proof of 
the existence of an OS-algebra over 
$\mb{Z}$ underlying $(L,\le)$.
\end{rmrk}

The definition of atomic complex make sense is more general setting
and is related to the notion of order complex of a lattice. 
Following \cite{Petersen} denote by $\tilde{H}_{*}(L(0,x))$ 
the reduced homology of the open interval 
$L(0,x)$ considered as a nerve of the corresponding category,
where by definition $\tilde{H}_*(L(0,0))$ is $\mb{Z}$
concentrated in homological degree $-2$.

Suppose $(L,\le)$ is a local sup-lattice, not necessary graded.
Let $\mc{D}:=p_*\OS(\mc{Q})$.
\begin{thrm}[Yuzvinsky; Lemma 2.3\cite{YuzAtomic}]\label{order_complex}
	For all $x\in L$ there is a natural isomorphism:
	$$H^*(\mc{D}_x)\simeq \tilde{H}_{-*-2}(L(0,x)).$$
\end{thrm}
If $(L,\le)$ is graded then $\mc{D}$ is a $(L,\le)$-chain algebra,
the isomorphism above is compatible with shuffle-type multiplication on the 
right hand side \cite{YuzAtomic}.

Otherwise $\mc{D}$ is not even an $(L,\le)$-filtered complex.
It will be useful to consider the following construction for arbitrary local sup-lattice.
Assume $A_\bullet$ is an $(L,\le)$-object in complexes, and $r\colon L\ar \mb{Z}$ 
such that $r(x)<r(y)$ if $x<y$. We call $r$ a \emph{weak} grading.
Since we do not assume that $r(y)=r(x)+1$ if $x<:y$,
such $r$ always exists.
The complex $T:=\Tot(p^*A_\bullet\otimes \OS(\mc{Q}))$ admits a filtration $W$
corresponding to $r$:
$$W^p T=\bigoplus\limits^{x\in L}_{r(x)\ge p}A_x\otimes (p_*\OS(\mc{Q}))_x.$$
Immediately we have
\begin{prop}\label{prop:atomic_complex_ss}
	Given $r$ defines a spectral sequence
	$$E^{pq}_1=\bigoplus\limits^{x\in L,i,j}_{r(x)=-p,i+j=p+q}H^i(A_x)\otimes H^j(\mc{D}_x),$$
	converging to $\Gr^p_W H^*(\Tot(p^*A_\bullet\otimes \OS(\mc{Q})))$.
\end{prop}
Of course the spectral sequence is multiplicative if $r$ satisfies $r(x\vee y)\le r(x)+r(y)$ for all $x,y\in L$.

\subsection{Arrangement posets}
\label{sect:lattices:arrangement_poset}
Here $\mc{T}=\Top,\Var$ etc. is a subcategory  of 
$\Top$ closed under finite coproducts and products. 
Assume $\mc{S}$ is a fibered category
over $\mc{T}$, i.e.\ 
a functor $\mc{S}\colon \mc{T}^{\op}\ar \Cat$ with values (the fibers) denoted by 
$\mc{S}_X\in \Cat$ for each $X\in \mc{T}$.
Essentially we need a notion of a pullback. 
Namely a morphism
$f\colon X'\ar X\in \mc{T}$ provides 
a functor $f^*:=\mc{S}(f)\colon \mc{S}_{X'}\ar \mc{S}_X$ between the fibers.

Examples of interest are $\mc{T}=\Top,\Var$,
$\mc{S}=\ChF(\Sh)$, in which case we have 
$\mc{S}_X=\ChF(\Sh_X)$.
In this subsection we call the objects of $\mc{T}$ simply 
{\it spaces} and the objects of $\mc{S}_X$ {\it sheaves} on $X$.

Let $X\in \mc{T}$ be a space.
Denote by $2^X$ the poset of subspaces 
of $X$ ordered by 
inclusion.
Let $[L,\Mu]$ be a graded local sup-lattice $(L,\le)$ equipped with an OS-complex $\Mu$. 
\begin{dfn}\label{dfn:arrangement_poset}
	An \textit{arrangement} 
	$[X,L,\mc{F}^L,\Mu]$ (in $\mc{T},\mc{S}$)
	is an $(L,\leq)$-subspace of 
	$X\in \mc{T}$, i.e.\ a 
	functor 
	$$L_{(-)}\colon (L,\leq)^{\op}\to {2^X}$$ with values
	$L_p\subset X$ for $p\in L$ such that
	\begin{enumerate}
		\item $L_0=X$;
		\item for each $I\subset \Atoms(L)$ we have
			$L_{I}:=
			\bigcap\limits_{i\in I}L_i=
			\bigsqcup\limits_{t\in \os{\circ}{\sup}(I)} 
			L_t\subset X$;
		\item for each $x\in X$ the full subposet 
			$\{p\in L\mid L_p\ni x\}\subset L$ is 
			 an interval of the form 
	$L[0,t],t\in L$;
		\item 
			$\bigcap\limits^{a\leq x}_{r(a)=1} L_a=L_x\subset X$,
	\end{enumerate}
	together with the additional data:
	\begin{enumerate}
		\item an $(L,\le)$-object $\mc{F}^L$ in $\mc{S}_X$;
		\item an OS-complex or OS-algebra $\Mu$ underlying $(L,\le)$.
	\end{enumerate}
\end{dfn}
In the case when $\mc{F}^L$ is an $(L,\le)$-algebra and 
$\Mu$ is an OS-algebra, 
we call $[X,L,\mc{F}^L,\Mu]$ \emph{multiplicative}.

\begin{dfn}\label{dfn:arrangement_posets:geometric}
	An arrangement is called \emph{locally geometric}	if $(L,\le)$ is a
	locally geometric lattice and for all $I\subset \Atoms(L)$
	we have $|\os{\circ}{\sup}(I)|\le 1$.
\end{dfn}

If $\mc{F}^L=\pi^*F$ is the constant $(L,\le)$-object, where
$\pi\colon (L,\le)\ar (0,\le)$ and $F\in \mc{S}_X$, we write 
$[X,L,F,\Mu]$ for $[X,L,\mc{F}^{L},\Mu]$.

Let us note that an \emph{arrangement} as defined in
Definition \ref{dfn:arrangement} is a cubical multiplicative arrangement 
$[X,\mc{C},\ul{\mb{Q}}_X,\OS(\mc{C})]$ 
where $\ul{\mb{Q}}_X$
is the constant $\mc{C}$-sheaf.

%
We call $U:=X-\bigcup\limits_{x>0\in L} L_x$ the 
\emph{complement} of the arrangement.
\begin{dfn}\label{dfn:arrangement_posets:morphism}
	A \emph{morphism} of (multiplicative) arrangements 
	$$\phi\colon [X',L',\mc{F}^{L'},\Mu']\ar[] [X,L,\mc{F}^L,\Mu]$$
	is the following data:
\begin{enumerate}
	\item a morphism $f\colon X'\ar X\in \mc{T}$;

	\item a rank-preserving homomorphism $g\colon (L,\le)\ar (L',\le)$
		such that 
		$f^{-1}L_{p}\subset L'_{g(p)}$ for each $p\in L$;
	\item a morphism 
		$h\colon f^*\mc{F}^L\ar g^*\mc{F}^{L'}$ of $(L,\le)$-objects 
		in $\mc{S}_{X'}$.
\end{enumerate}
In the multiplicative case we require in addition that
$h$ should be a morphism of partial $(L,\le)$-algebras and
the natural morphism $g_*\Mu\ar \Mu'$ provided by 
Proposition \ref{prop:morphism_of_OS_complexes} should be 
a morphism of $(L',\le)$-chain algebras.
\end{dfn}
Note that the morphism $g\colon L\ar L'$ satisfies $g^{-1}(0)=0$ and 
the restriction $g\colon \Atoms(L)\ar \Atoms(L')$ is a well-defined 
locally injective map (Definition \ref{dfn:lattices:locally_injective}).

The morphism $\phi$ restricts to a map of complements: 
$f|_{U'_{g(p)}}\colon U'_{g(p)}\ar U_p$, 
where $U_p=X-L_p$ for each $p\in L$.
\begin{rmrk}\label{dfn:arrangement_posets:category}
	There are various subcategories of the category of all arrangements.
	Working additively, by the functoriality of OS-complexes, $\Mu$ is unique
	if it exists
	(Proposition \ref{prop:morphism_of_OS_complexes}).
	We can omit $\Mu$ from the notation and 
	consider the category of arrangements of the form 
	$[X,L,\mc{F}^L]:=[X,L,\mc{F}^L,\Mu]$ or $[X,L,F]:=[X,L,F,\Mu]$.
	Similarly, in the case $\mc{S}_X=\Sh_X$ and $\mc{F}^L$ is 
	a constant $(L,\le)$-object equal to the constant sheaf 
	$\ul{\Bbbk}_X$, 
	we can consider the category of objects denoted by $[X,L,\Mu]:=[X,L,\ul{\Bbbk}_X,\Mu]$.

	Other main examples of interest are multiplicative and 
	formed by arrangements $[X,L,F,\Mu]$ or $[X,L,\ul{\Bbbk}_X,\Mu]$,
	locally geometric posets $[X,L,F]=[X,L,F,\OS(L)]$ or 
	$[X,L,\mc{F}^L]=[X,L,\mc{F}^L,\OS(L)]$, etc.
	These are well defined thanks to the functoriality of $\OS(-)$, see 
	Corollary \ref{cor:morphism_of_OS_algebras}.

 
\end{rmrk}
With this in mind, it should be clear from the context which category of arrangements
we consider.
We will denote it simply 
$\AP(\mc{T},\mc{S})$.

Let $\mc{E}\ar \mc{T}$ be a functor. 
\begin{dfn}\label{dfn:arrangement_posets:functorial_equipment}
	A \emph{functorial equipment} with values in $\mc{E}$ 
	is a functor
	$\Phi\colon \AP(\mc{T},\mc{S})\ar \mc{E}$
	over $\mc{T}$. 
\end{dfn}
\begin{rmrk}\label{rmrk:target_of_equipment}
	Let $\mc{T}=\Var$.
	Usually $\mc{E}$ will be a category of pairs 
	$(X,F_X)$, where $F_X$ are in $\ChF(\Sh_X)$.
	By definition a morphism $(X',F_{X'})\ar (X,F_X)$ is a pair 
	$f\colon X'\ar X$ and $f^*F_X\ar F_{X'}\in \ChF(Sh_{X'})$.
	In the case of mixed Hodge complexes $F_X\in \MHC(X)$ and $F_{X'}\in \MHC(X')$, $f^*F_X$
	is not an MHC, but
	$f^*F_X\ar F_{X'}\in \ChF_{F,W}(\Sh_X)$ induces
	$F_X\ar R_{TW}f_* F_{X'}\in \MHC(X)$ (see Remark \ref{rmrk:pullback_for_mhc} in \ref{sect:tws_functor}).
\end{rmrk}

A \emph{morphism} of functorial equipments $\Phi\ar \Phi'$
is defined in the obvious way.

\begin{exmpl}\begin{enumerate}
	\item An arrangement $(X,Z_*)$ parametrized by an index set $N$
		defines the \emph{cubical} arrangement $[X,\mc{C}]$, where
		the cubical lattice $\mc{C}=2^N$ is ordered by inclusion and
		$\mc{C}_I=Z_I\subset X$ provides a map $(\mc{C},\le)^{\op} \ar 2^X$.
		By definition $[X,\mc{C},\ul{\mb{Z}}_X,\OS(\mc{C})]$ 
		is a multiplicative geometric arrangement. 
	
	\item
	Recall that the \emph{intersection poset} $P$ of $(X,Z_*)$
	is the sublattice of $2^X$ of all \emph{different} intersections 
	$Z_I=\cap_{i\in I} Z_i\subset X$. 
	If $P$ is graded and admits an OS-complex $\Mu$, 
	it tautologically gives an arrangement $[X,P,\Mu]$.
	There is a natural morphism $p\colon [X,\mc{C},\OS(\mc{C})]\ar[] [X,P,\Mu]$.

	\item The intersection poset of vector hyperplanes is always a geometric
		lattice with the rank function given by codimension. 
		More generally an intersection poset $(L,\le)$ of hypersurfaces in a manifold $M$
		with clean 
		intersections,
		i.e.\ locally homeomorphic to a hyperspace arrangement,
		is locally geometric lattice.
		This provides a multiplicative arrangement $[M,L,\ul{\Bbbk}_M,\OS(L)]$.
	\item  
		The configuration space of $n$-ordered points in $X\in \Top$ is
	by definition the complement
	$$F(X,n):=X^{\times n}-\cup_{ab}\Delta_{\{ab\}},$$ where
	$\Delta_{\{ab\}}:=\{(x_1,\ldots,x_n)\in X^{\times n}\mid x_a=x_b\}.$
	The intersection poset of $\Delta_{ab}\subset X^{\times n}$ for all 
	$a,b\le n$
	provides a arrangement $[X^{\times n},\Delta_*]$ 
	with
	$\Atoms(\Delta_*)=\{\Delta_{ab}\mid a,b\le n\}$.
	Since the poset $\Delta_*$ is independent of $X$, one
	can take $X=\mb{C}$ and see by the previous part that
	$[X^{\times n},\Delta_*]$ is geometric.
	The corresponding subspaces
	$\Delta_{(-)}\subset X^{\times n}$, called diagonals, 
	are given by
	$$\Delta_{\{S_1,\ldots,S_k\}}=
	\{(x_1,\ldots, x_n)\in X^{\times n}\mid i,j\in S_t\Rightarrow x_i=x_j\},$$ 
	parametrized by an
	unordered 
	collection of disjoint non-empty subsets
	$\sqcup_{i\le k} S_i=\{1,\ldots,n\}$.
        This is the \emph{partition lattice}
        of $\{1,\ldots,n\}$.
	The rank function $r$ satisfies 
	$r(\Delta_{\{S_1,\ldots,S_k\}})=\sum\limits_{i\leq k} |S_i|-1$.
	The permutations $\Sigma_n$ acting on $X^{\times n}$
	induce a natural action on the arrangement $[X^{\times n},\Delta_*,\OS(\Delta_*)]$.

\item	Consider a vector subspace arrangement given by 
	two planes and one line 
	$\Pi_1,\Pi_2,l\subset \mb{C}^3$ in general position 
    meeting at the origin, which we denote $p$. 
	The corresponding intersection poset consisting of 
	$$\mb{C}^3,\Pi_1,\Pi_2,l,l_{12}=\Pi_1\cap \Pi_2,p
	\subset \mb{C}^3$$ 
	has no rank function:
	
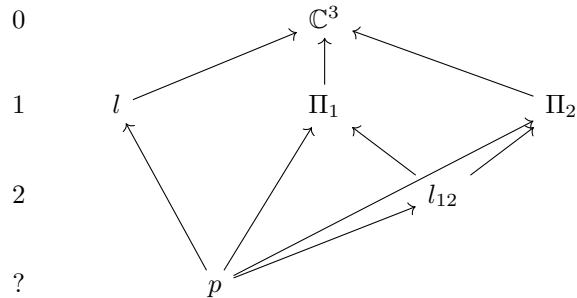
\begin{equation*}
	\begin{tikzcd}
	0&&	&\mb{C}^3	&	&\\
	1&l\arrow[urr]&&\Pi_1\arrow[u]&&\Pi_2\arrow[ull]\\
	2&&&&l_{12}\arrow[ul]\arrow[ur]&\\
	?&&p\arrow[uul]\arrow[uur]\arrow[uurrr]\arrow[urr]&&&
	\end{tikzcd}
	\end{equation*}
	Some paths 
	from $\mb{C}^3$ to $p$ have length 2 and others have length 3, 
	so there is no way to define a grading at $p$.

	One can always consider the corresponding cubical lattice 
	to resolve the issue:
	
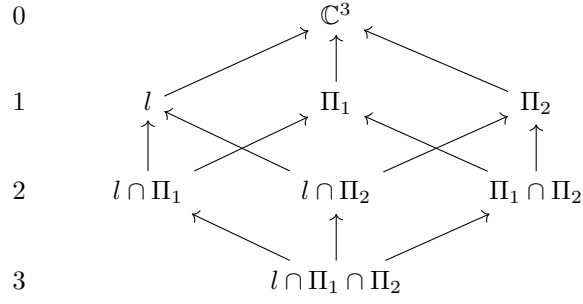
\begin{equation*}
	\begin{tikzcd}
		0&&\mb{C}^3	&\\
		1&l\arrow[ur]&\Pi_1\arrow[u]&\Pi_2\arrow[ul]\\
		2&l\cap \Pi_1\arrow[u]\arrow[ur]&
			l\cap \Pi_2\arrow[ul]\arrow[ur]&
			\Pi_1\cap \Pi_2\arrow[u]\arrow[ul]\\
		3&&l\cap \Pi_1\cap \Pi_2\arrow[ul]\arrow[u]\arrow[ur]&
	\end{tikzcd}
	\end{equation*}
	This result is a geometric arrangement.

	Another option is to duplicate $p$ by
	adding $p'$ to the arrangement 
	such that $L_{p'}=L_p$:
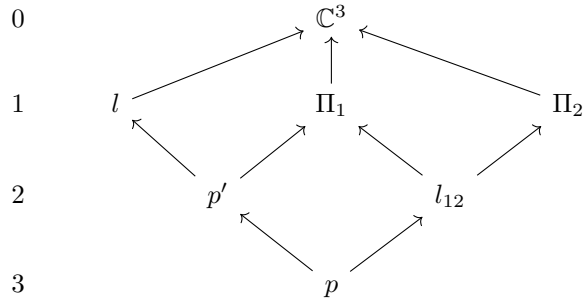
	\begin{equation*}
	\begin{tikzcd}
	0&&	&\mb{C}^3	&	&\\
	1&l\arrow[urr]&&\Pi_1\arrow[u]&&\Pi_2\arrow[ull]\\
	2&&p'\arrow[ul]\arrow[ur]&&l_{12}\arrow[ul]\arrow[ur]&\\
	3&&&p\arrow[ul]\arrow[ur]&&
	\end{tikzcd}
	\end{equation*}
	Note that this arrangement is not geometric
	and does not satisfy the second condition of 
	Definition \ref{dfn:lattice:geometric}.
\end{enumerate}
\label{exmpl:arrangementss}
\end{exmpl}

\newpage
\section{Natural mixed Hodge diagram 
\texorpdfstring{$(\mc{K}_{U/X},F,W)$}{}}\label{sect:aznar_mhd}
Let $j\colon U\subset X$ be an open in smooth compact algebraic variety. 
The main result we will use in \ref{sect:compl:mvhc}
is Corollary \ref{cor:functorial_mhd},
which provides a functorial mixed Hodge diagram $\mc{K}_{U/X}\in \MHD(X)$ 
corresponding to $Rj_*\mb{Q}_U$. 
This is done in a straightforward way 
by reducing to the case of an NCD-embedding $U/\tilde{X}$
and then taking colimit over all NCD-embeddings 
$U/\tilde{X}$ over $U/X$. The former case is
Theorem \ref{thrm:aznar_functorial_mhd} due to Navarro Aznar~\cite{Aznar}.
It uses the Thom-Whitney-Sullivan functor $s_{TW}$, which 
refines totalization functor $s\colon \cs\Vect_{\mb{Q}}\ar \Ch(\Vect_{\mb{Q}})$
to an equivalent lax \emph{symmetric} monoidal functor 
$s_{TW}\colon (\cs\Vect_{\mb{Q}},\otimes)\ar (\Ch(\Vect_{\mb{Q}}),\otimes)$.
A construction similar in Corollary \ref{cor:functorial_mhd} appears also
in \cite{CH20}[6.1].

\subsection{Adapted sheaves}
Below we will work in the category $\Sm$ of smooth 
manifolds with the usual topology
given by open embeddings. 
In particular for any $X\in \Sm$,
the restriction $\PSh_{\Sm}\to \PSh_{X}$
commutes with the sheafififaction $\PSh_*\to \Sh_*$ 
sending  $F\in \PSh_*$ to $\ul{F}\in \Sh_*$.

\begin{dfn}[Sullivan polynomial forms, {\cite[\S 10]{FHT}}]
	Let $A^*_{\PL}\in \CMon(\Ch(\mb{Q}-\PSh_{\Sm}))$ be 
	the presheaf of Sullivan polynomial forms, i.e.
	for each $X\in \Sm$, $A^*_{\PL}(X)\in \CDGA_{\mb{Q}}$
	is the cdga of polynomial forms on 
	the singular simplicial set
	of $X$.
\end{dfn}
\begin{prop}[{\cite[\S 5, Theorem 5.5]{Aznar}}]
	The sheaffication $A^*_{\PL}\ar \ul{A}^*_{\PL}$  
	is a quasi-isormorphism of preasheaves in 
	$\CMon(\Ch(\mb{Q}-\PSh_{\Sm}))$.
\end{prop}

Let $\Bbbk\supset \mb{Q}$ be a field.
\begin{dfn}
	We say a cdga $R^*\in \CDGA_{\Bbbk}$ \textit{describes} the 
	$\Bbbk$-homotopy type of $X$
	if there is a quasi-isomorphism 
	$A^*_{\PL}(X)\otimes_{\mb{Q}} \Bbbk\zgqis R^*$ in $\CDGA_{\Bbbk}$.
\end{dfn}


Recall that (Section~\ref{sec:notation_etc}) for an inclusion of an open subset $j\colon U\subset X$ and 
a sheaf $\ul{F}$ on $X$, we denote the sheaf 
$j_* j^*\ul{F}$ by $(j_* j^*)_U\ul{F}$. 
Similarly, we denote $(Rj_* j^*)\ul{F}$ by 
$(Rj_* j^*)_U\ul{F}$. 

\begin{dfn}
	Let $U\subset X$ be open. 
	A sheaf $\ul{A}^*_X\in \CMon(\Ch(\Bbbk-\Sh_X))$ over $X$ 
	is $U$-\textit{adapted}, if
	there is a quasi-isomorphism
		$$\ul{A}_X^*\zgqis 
			(j_*j^*)_U \ul{A}^*_{\PL}\otimes \Bbbk\in 
				\CMon(\Ch(\Bbbk-\Sh_X)),$$
     		which we consider as part of the structure.
\end{dfn}

In the sequel an $X$-adapted sheaf $\ul{A}_X$ over $X$ will be called an adapted sheaf over $X$.
Similarly, we define adapted sheaves over $\Sm$ and call them adapted.
The de Rham complex $(\Lambda^*,d_{\dR})$ over $\Bbbk=\mb{R}$
and $\ul{A}^*_{\PL}\otimes \Bbbk$ over $\Bbbk$ are examples 
of soft adapted sheaves \cite[Lemma 5.4]{Aznar}.

\subsection{Thom-Whitney-Sullivan totalization functor}
\label{sect:tws_functor}
Recall that we use $\ChF(\mc{A})$ to denote any of several similar categories of (multi)filtered complexes in $\mc{A}$, see \ref{dfn:chf}. In particular, an object of $\ChF$ comes equipped with one or several filtrations, and we will denote either of these by $V$.  
If $\mc{A}$ is abelian symmetric monoidal, then $\ChF(\mc{A})$ is additive symmetric monoidal.
We will now discuss the Thom-Whitney-Sullivan functor $s_{TW}$. 

As a motivation consider the Moore complex functor
$s\colon \cs{\ChF(\Sh_X)}\to \ChF(\Sh_X)$. 
The Alexander-Whitney formula 
makes $s$ a lax (non-symmetric) monoidal functor, in particular we get a functor
$$s\colon \cs\CMon(\ChF(\Sh_X))\to \Mon(\ChF(\Sh_X)),$$
which takes values in non-commutative monoids.
The functor $s_{TW}$ can be though as a symmetric monoidal refinement of $s$.
Below we summarize its main properties.
\begin{prop}[{\cite[2.2, 3.3, 6.12]{Aznar}}]
\label{prop:s_tw}
There is a lax symmetric monoidal functor
$$s_{TW}\colon (\cs\ChF(\Vect_{\mb{Q}}),\otimes)\to (\ChF(\Vect_{\mb{Q}}),\otimes),$$ 
and a natural filtered quasi-isomorphism 
$s_{TW}\zgqis s$.
The functor $s_{TW}$ is exact on chain complexes. 
This induces a symmetric  lax
monoidal functor
$$s_{TW}\colon \cs\ChF(\mb{Q}-\Sh_X)\to \ChF(\mb{Q}-\Sh_X)$$
defined by $(s_{TW}\mc{F})(U)=s_{TW}(\mc{F}(U))$ for each 
$\mc{F}\in \ChF(\mb{Q}-\Sh_X)$ and any open $U\subset X$.
\end{prop}

For continous $f\colon X\to Y$ the exactness of $s_{TW}$ implies $s_{TW}\circ f^*=f^*\circ s_{TW}$
and $s_{TW}\circ f_*=f_*\circ s_{TW}$.

Let $G_{\cs}(-)\colon \ChF(\Sh_X)\to \cs\ChF(\Sh_X)$ denote the Godement 
cosimplicial construction (\cite[Appendice]{Godement}). 
Then $G=s\circ G_{\cs}$ is the Godement resolution.
If $f\colon X\to Y$ is continuous, then there is natural morphism 
$G_{\cs}f_*\to f_* G_{\cs}$ and by the exactness of $f^*$, 
a natural quasi-isomorphism $f^*G_{\cs}\to G_{\cs}f^*$.

\begin{dfn}
	Set
	\begin{align*}
		G_{TW}&=s_{TW}\circ G_{\cs},\\
		R_{TW}f_*&=s_{TW}\circ f_*\circ G_{\cs}=f_*\circ G_{TW},\\
		\Gamma_{TW}&=s_{TW}\circ \Gamma\circ G_{\cs}=\Gamma\circ G_{TW}.\\
	\end{align*}
\end{dfn}
Note that in case when $j$ is an open inclusion 
we have a natural isomorphism $j^*  G_{\cs}\simeq G _{\cs}j^*$,
in particular $R_{TW}j_*j^*\simeq j_*j^*G_{TW}$.

Let us call a complex $K^*\in \ChF(\Sh_X)$ \emph{filtered soft}, if 
for any filtration $V$ defined by $\ChF$,
the filtered components $V^i K^n\in \Sh_X$ are soft
for all $n,i$.
\begin{prop}[{\cite[4.5, 4.7, 6.13]{Aznar}}]
	The functor
	$$G_{TW}\colon (\ChF(\mb{Q}-\Sh_X),\otimes)\to (\ChF(\mb{Q}-\Sh_X),\otimes)$$
	is lax symmetric monoidal and provides 
	a \textit{filtered} soft resolution, 
	i.e.\ there is a natural \emph{filtered} quasi-isomorphism
	$$K\ar[\sim]G_{TW}(K)\in \ChF(\Sh_X)$$
	for all $K\in \ChF(\mb{Q}-\Sh_X)$.
\end{prop}
\begin{rmrk}\label{rmrk:pullback_for_mhc}
	In particular given morphism $f^*F\ar F'\in \ChF(\Sh_X)$ induces a natural morphism 
	$F\ar R_{TW}f_* F'\in \ChF(\Sh_{X'})$
given by the composition $$F\ar G_{TW}(F)=s_{TW}G_{\cs}(F)\ar s_{TW}f_* G_{\cs}(F')=R_{TW}f_* F',$$
where the middle arrow is induced by $f^*(G_{\cs}(F))\ar G_{\cs}(f^*(F))\ar G_{\cs}(F')$.
\end{rmrk}
\begin{rmrk}
	The quasi-isomorphism $s_{TW}\zgqis s$ in Proposition \ref{prop:s_tw} 
	is monoidal only up to homotopy
	(see \cite[Theorem 3.3]{Aznar}). 
	However for $F\in \CMon(\Ch(\mb{Q}-\Sh_X))$ there is a natural
	zig-zag in $\DGA_{\mb{Q}}$:
	$$\Gamma \left[G(F)\ar G(G_{TW}(F))\al G_{TW}(F)\right]$$
	induced by the resolutions $F\ar G_{TW}(F)$ and $G_{TW}(F)\ar G(G_{TW}(F))$ respectively. Since 
	$G_{TW}(F)$ is soft, these morphisms provide a zig-zag of quasi-isomorphisms.
\end{rmrk}

This canonical resolution allows one to control the rational homotopy type on the level of sheaves
of cdga's, namely we have:
\begin{cor}\label{cor:global_sections_of_adapted_sheaf}
	
	If $(\ul{A}^*_X,V)\in \CMon(\ChF(\Bbbk-\Sh_X))$ is a $U$-adapted sheaf 
	$\ul{A}^*_X$ with a multiplicative filtration $V$, then the natural map
	$(\ul{A}^*_X,V)\ar G_{TW}(\ul{A}^*_X,V)$ is a filtered resolution
	of $\ul{A}^*_X$
	by a soft $U$-adapted sheaf $G_{TW}(\ul{A}^*_{X})$ 
	with the filtration $G_{TW}(V)$. 
	The same is true in the multifiltered case.
	We have a quasi-isomorphism 
		$$\Gamma_{TW}(\ul{A}^*_X)\zgqis A_{\PL}(X)\otimes \Bbbk\in 
			\CDGA_{\Bbbk}$$
\end{cor}

Finally we have a natural functor $R_{TW}f_*\colon \MHD(X)\to \MHD(Y)$.
In particular,
\begin{thrm}[{\cite[8.15]{Aznar}}]\label{thrm:mhd_sheaves_sections}
	If $(\mc{K},F,W)\in \MHD(X)$, then 
		$\Gamma_{TW}(\mc{K},F,W)\in \MHD$
	provides a mixed Hodge diagram in the sense of Morgan.
 If $\mc{K}$ is $U$-adapted for an open $U\subset X$, 
 then there is a quasi-isomorphism $\Gamma_{TW}(\mc{K}_{\mb{Q}})\zgqis A_{\PL}(U)\in\CDGA_{\mb{Q}}$. 
\end{thrm}

\subsection{Logarithmic complex}
Assume $j\colon U\hookrightarrow X$ is an NCD-embedding in 
a smooth compact algebraic variety $X$.
Such embeddings form a category $\Var^{NCD}_{-/-}$ with objects denoted by $U/X$.
Let $\MHC_{-}$ be the category of pairs $(X,\mc{F})$ for
$X\in \Var$ and $\mc{F}\in \MHC(X)$, where
by definition a morphism $(X,\mc{F})\ar (X',\mc{F}')$ is
a pair of morphisms $f\colon X\ar X'\in \Var$ and 
$g\colon f^*\mc{F}'\ar \mc{F}\in \ChF_{F,W}(\Sh_X)$,
so by Remark \ref{rmrk:pullback_for_mhc}, $g$ induces $\mc{F}'\ar[g] R_{TW}f_* \mc{F}\in \MHC(X')$.
Exterior tensor product $-\boxtimes-$ in $\MHC_{-}$ 
defines a symmetric monoidal category $(\MHC_{-},\boxtimes)$ such that
$(\MHC_{-},\boxtimes)\ar[(X,\mc{F})\ar X] (\Var,\times)$ is a
fibered category.

Let 
\begin{equation}\label{eq:dfn:log_subsheaf}
\Omega^*(\ln X-U)\subset j_*\Omega^*_{U}
\end{equation}
denote the \textit{logarithmic}
(holomorphic) de Rham complex over $X$.
It is a sheaf of cdga's.
\begin{dfn}[{\cite[\S 3.1]{D}}]
	The logarithmic complex $\Omega^*(\ln X-U)$
	admits two 
	filtrations:
	\begin{enumerate}
		\item The \textit{Hodge} decreasing filtration 
			$F^p \Omega^*(\ln X-U)=
				\Omega^{\geq p}(\ln X-U).$
		\item The \emph{weight} increasing filtration $W_n \Omega^*(\ln X-U)=
			\Omega^*_{X}\wedge \Omega^{n}(\ln X-U)[-n]+
			\Omega^{\leq n}(\ln X-U).$
	\end{enumerate}
\end{dfn}
In other words, the Hodge filtration $F$ is the dumb decreasing 
truncation, while
the weight filtration $W_n\Omega^*(\ln X-U)$ is formed by the log-forms 
containing at most $n$ singular terms, e.g.\ we have $W_0\Omega^*(\ln X-U)=\Omega^*_{X}$.
Both filtrations are multiplicative.

Note that the inclusion in \eqref{eq:dfn:log_subsheaf} is functorial 
in the following sense.
If $f\colon U/X\ar U'/X'$ is a morphism, then 
Then $f^{-1}(X'-U')\subset X-U$.
If $\omega\in \Omega^*(\ln X'-U')$ has a logarithmic 
singularity along $X'-U'$, then $f^*\omega$
has an at most logarithmic singularity along $X-U$.
Thus 
\begin{prop}
\label{prop:log_complex_functoriality}
	The complex 
	$(\Omega^*(\ln X-U),F,W)$ is functorial in 
	$U/X\in \Var^{NCD}_{-/-}$, 
	i.e.\ if $f$ is as above, then
	there is an $f^*\colon \Omega^*(\ln X'-U')\longrightarrow
	\Omega^*(\ln X-U)$ such that $f^* W_n\subset W_n$ and 
	$f^*F^p\subset F^p$.
\end{prop}

\begin{prop}[{\cite[Proposition 3.1.8]{D}}]\label{prop:sheaves:log_complex}
	There is a natural zig-zag of filtered quasi-isomorphisms
	$$(\Omega^*(\ln X-U),W)\longleftarrow(\Omega^*(\ln X-U),\tau_{\le})
		\longrightarrow (j_*\Omega^*_U,\tau_{\le}).$$
\end{prop}
Recall that  $j_*j^*G(\ul{\mb{Q}})$ is a natural flabby resolution of 
$Rj_* j^* \ul{\mb{Q}}$. We have
\begin{thrm}[\cite{D}]\label{thrm:deligne_mhc}
	Mixed Hodge complex of sheaves
	$$(j_* j^* G(\ul{\mb{Q}})_X,\tau_{\leq}),
	(\Omega^*(\ln X-U),F,W))$$ 
	provides a functor $\Var^{NCD}_{-/-}\ar \MHC_{-}$
    over $\Var$.
\end{thrm}
In fact we have
$\Omega^*(\ln X-U)\zgqis j_*j^* G(\ul{\mb{C}}_X)$ in $\Mon(\Ch(\mb{C}-\Sh_X))$.
Using the Thom-Whitney functor $s_{TW}$ it is possible to upgrade 
this construction
to the level of cdga's over $\Bbbk\supset \mb{Q}$.
We have a soft resolution 
$j_*j^*G_{TW}(\ul{\mb{Q}})\sim R_{TW}j_*j^*\ul{\mb{Q}}\in \CMon(\Ch(\mb{Q}-\Sh_X))$
of  $Rj_*j^*\ul{\mb{Q}}$,
together with a natural quasi-isomorphism 
$j_*j^* G_{TW}(\ul{\mb{Q}})\otimes \mb{C}\zgqis \Omega^*(\ln X-U)\in \CMon(\Ch(\mb{C}-\Sh_X))$.

\begin{thrm}[{\cite[8.15]{Aznar}}]\label{thrm:aznar_functorial_mhd}
	There is a natural mixed Hodge diagram of $U$-adapted sheaves
	$$(\mc{K}^{NCD}_{U/X},F,W):=
	((j_* j^*G_{TW}(\ul{\mb{Q}})_X,\tau_{\leq}),
	(\Omega^*(\ln X-U),F,W))\in \MHD(X),$$
	providing a lax symmetric monoidal functor of categories fibered over $(\Var,\times)$:
	$$\mc{K}^{NCD}_{-}\colon (\Var^{NCD}_{-/-},\times)\ar (\MHC_{-},\boxtimes).$$
\end{thrm}
Forgetting the monoidal structures the functor $\mc{K}^{NCD}_{-}$ is naturally 
quasi-isomorphic to the Deligne MHC of Theorem \ref{thrm:deligne_mhc}.
Here the lax structure is the quasi-isomorphism 
$\mc{K}^{NCD}_{U/X}\boxtimes \mc{K}^{NCD}_{U'/X'}\ar 
	\mc{K}^{NCD}_{U\times U'/X\times X'}$.
Restriction to the diagonal induces a product 
$\mc{K}^{NCD}_{U/X}\otimes \mc{K}^{NCD}_{U/X}\ar 
\mc{K}^{NCD}_{U/X}\in \MHD(X)$.

We want to extend this construction to the case when $X-U$ 
is not necessary an NCD.
Consider the category $\Var_{-/-}$ of pairs $U/X$, where $X\in \Var$ is smooth 
proper and $U$ is open. It is fibered over $\Var$ with fibers
$\Var_{-/X}$.
An NCD compactification $p\colon U/\tilde{X}\ar U/X$ provides
$R_{TW}p_*\mc{K}^{NCD}_{U/\tilde{X}}\in \MHD(X)$.
Any two NCD compactifications of a given $U/X$
can be dominated be a third one (\cite{D}), i.e. $\Var_{-/X}$ is cofiltered.
Similarly, every morphism $U/X\ar U'/X'$ can be lifted to 
some NCD compactifications $U/\tilde{X}\ar U'/\tilde{X}'$.
Following {\cite{CH20}[6.1]} we extend the definition of 
$\mc{K}^{NCD}$ to $\Var_{-/-}$ by the formula
$$\mc{K}_{U/X}:=
	\colim_{p\colon (U/\tilde{X},U)\ar (U/X)}
	R_{TW}p_* \mc{K}^{NCD}_{U/\tilde{X}}\in \MHD(X),$$
where the colimit is taken over all NCD-compactifications of $U$ over 
$X$. This colimit is well defined since it is a colimit of quasi-isomorphic objects.
As a corollary we have
\begin{cor}\label{cor:functorial_mhd}
There is a commutative diagram of lax symmetric monoidal functors fibered over $(\Var,\times)$:
\begin{equation*}
\begin{tikzcd}
	(\Var_{-/-},\times) \arrow[rd,swap,"{U/X\rightarrow X}"]\arrow[rr,"{U/X\rightarrow (X,\mc{K}_{U/X})}"] &&
(\MHC_{-},\boxtimes)\arrow[ld,"{X\leftarrow(X,\mc{K}_{U/X})}"]\\
		&(\Var,\times)&
\end{tikzcd}
\end{equation*}
\end{cor}
In this sense $\mc{K}_{U/X}\in \MHD(X)$ is functorial in $U/X$ and
we will say that $\mc{K}_{U/X}$ provides a \emph{functorial equipment}
of $U/X$ in $\MHD(X)$.
\newpage
\section{Towards the complement}\label{sect:construction}
Throughout the section $[X,L,\Mu]$ 
is an arragement poset of closed subspaces 
of $X\in \Top$. 
We will define the Mayer-Vietoris complex 
$\MV$ which can be used for
calculating the cohomology of the
complement $U=X-\bigcup\limits^{x\in L}_{x>0}L_x$. 
It serves as a prototype for the Mayer-Vietoris Hodge diagram 
$(\MVHD,F,W')$
which we describe in \ref{sect:compl:mvhc} in order
to model the rational homotopy type of $U$ in case of an algebraic arrangement.

In \ref{sect:compl:mv_resolution}, following \cite[\S 2]{Looijenga},
for $F\in\Ch(\Sh_X)$
we define 
$\MV=\MV(X,L,(i_*i^!)_{L_\bullet} G(F),\Mu)$ 
as the convolution
$(i_*i^!)_{L_\bullet} G(F)*\mc{M}\in \Ch(\Sh_X)$. 
Here 
$G(F)$ is the cosimplicial Godement resolution of
$F$.
It is provided with a natural quasi-isomorphism
$\MV\sim Rj_*j^*F$ for the open inclusion 
$j\colon U\to X$.
The construction is multiplicative, albeit non-commutative, 
e.g.\ if $\mc{M}$ is an OS-algebra and 
$F\in \CMon(\Ch(\Sh_X))$, then
$\MV$ is naturally an object of $\Mon(\Ch(\Sh_X))$.

As a convolution $\MV$ is equipped with the rank filtration $W'$.
The corresponding spectral sequence
$E^{pq}_1(\MV,W')\Rightarrow H^*(U;j^*F)$ is described in Theorem \ref{thrm:ss_for_lattices}.
As an application to chromatic configuration spaces an explicit expression of
$(E^{**}_1,d_1)$ in terms of generators and relations is given in \ref{sect:mv_totaro}. 
This generalizes Totaro's spectral sequence \cite[Theorem 1]{Totaro}.

In \ref{sect:leray_ss} we relate this spectral spectral sequence
to the \emph{lattice spectral sequence} introduced independently by Peterson and Tosteson in \cite{Petersen}, 
\cite{To}. 
In \ref{GorMac} we sketch a proof of the Goresky-MacPherson formula
in order to relate the lattice spectral sequence to the Leray spectral sequence of the open embedding of the complement.

In \ref{sect:cech_model} 
we introduce the \v{C}ech construction in order 
to upgrade the non-commutative construction of $\MV$ 
in the case $\Bbbk\supset \mb{Q}$ to 
a genuine cdga. 
This construction will be used to obtain 
the Hodge diagrams $(\MVHD,F,W)$ and 
$(\MVHD,F,W')$ in Section \ref{sect:compl:mvhc}.

\subsection{Mayer-Vietoris model \texorpdfstring{
$\MV$ of
$Rj_*j^*\mc{F}_X$}{}
}
\label{sect:compl:mv_resolution}
Recall that $X\in \Top$ is assumed to be paracompact, 
Hausdorff and locally contractible topological space. 
Here
$[X,L,\Mu]$ is an arrangement of closed subspaces in $X$.
Recall that $\Mu$ is essentially unique due to 
Proposition \ref{prop:morphism_of_OS_complexes}. 
For each $s\geq t\in L$ we have the natural embeddings 
$i_{ts}\colon L_s\to L_t\subset X$ and
$j\colon X-\bigcup\limits_{s\in L}L_s\to X$.
Put $i_s:=i_{0s}$.

For a closed embedding $i\colon Z\to X$ we define the functor $i^!\colon \Sh_X\to \Sh_Z$ 
by taking the sections with support in $Z$. Any sheaf $F\in \Ch(\Sh_X)$
defines an $(L,\leq)$-sheaf $(i_*i^!)_{L_\bullet} F$ given by 
$$(i_*i^!)_{L_s} F:={i_s}_* {i^!_s}F.$$
The structure maps are given by the natural inclusions $\partial_{ts}$ induced by $i_{ts}$.
\begin{prop}\label{prop:ii_is_Lalgebra}
	If $[X,L,F,\Mu]$ is multiplicative, 
	then $(i_*i^!)_{L_\bullet} F$ is
	an $(L,\leq)$-algebra in $\Ch(\Sh_X)$.
\end{prop}
\begin{proof}
For all open $V\subset X$ we have a natural product
$$m_{xy}\colon (i_*i^!)_{L_x} F(V)\otimes (i_*i^!)_{L_y} F(V)\to (i_* i^!)_{L_x\cap L_y} F(V).$$
This together with the decomposition $L_x\cap L_y=\bigsqcup\limits_{t\in x\os{\circ}{\vee}y}L_t$
gives for each $t\in x\os{\circ}{\vee} y$, a well-defined component
$m^t_{xy}\colon (i_*i^!)_x F\otimes (i_* i^!)_y F\to (i_*i^!)_t F$.
The associativity follows from the associativity of the multiplication in $F$.
\end{proof}
Thus $[X,L,(i_*i^!)_{L_\bullet}F,\Mu]$ is an arrangement.
More generally, let $\mc{F}^L$ be an $(L,\le)$-object in $\Sh_X$.
\begin{dfn}[Mayer-Vietoris complex]\label{dfn:mv_complex}
	The Mayer-Vietoris complex $\MV(X,L,\mc{F}^L,\Mu)$
	is the convolution 
	$\Tot(\mc{F}^L\otimes \mc{M})\in \Ch(\Sh_X)$.
\end{dfn}
Recall that $\MV(X,L\mc{F}^L,\Mu)$ is an $L$-filtered object and 
admits the rank filtration $W'$ (Definition \ref{dfn:rank_filtration}):
$$W'_k \MV(X,L,\mc{F}^L,\Mu)=
	\bigoplus\limits^{x\in L}_{r(x)\leq k}\mc{F}^L_x\otimes \Mu_x.$$

	In case $\mc{F}^L=(i_*i^!)_{L_\bullet}F$
we will write $\MV(F)$ for short.
If $F\in \Mon(\Ch(\Sh_X))$ and
$\Mu$ is multiplicative, then
$\MV(F)$ naturally lands to $\Mon(\Ch(\Sh_X))$.
Since $\Mu$ is well-graded, the filtration $W'$ is multiplicative.
Assume $F\in \Sh_X$, then $\MV(F)\in \Ch^{\leq 0}(\Sh_X)$.
There is a natural augmentation map
	$$\eta\colon \MV(F)\to 
	j_* j^*F.$$
By definition $\eta$ restricted to $\MV(F)_s\subset \MV(F)\in \Ch^{\leq 0}(\Sh_X)$ 
is zero if $s>0\in L$, and for $s=0$ it is the natural morphism $\MV(F)_0=F\to j_*j^*F$.
Clearly $\eta$ is a morphism of complexes of sheaves, since 
	$\partial (\MV(F)_t)$ for $t>0$ is formed by sections supported on 
	the complement $X-\bigcup\limits^{s\in L}_{s>0}L_s$.

\begin{dfn}
A {\it Godement} sheaf $F$ is a sheaf of  the form
$F(U)=\prod_{x\in U}\tilde{F}_x$, for some sheaf $\tilde{F}\in \Sh_X$.
\end{dfn}

\begin{prop}\label{prop:mv_resolution}
	Assume $F$ is a Godement sheaf. Then
	\begin{enumerate}
		\item There is a natural inclusion 
			$\eps\colon j_* j^*F\to F=\MV^0(F)\subset \MV^*(F)$ 
			of complexes. The composition 
			$$j_*j^*F\ar[\eps] \MV(F)\ar[\eta] j_*j^*F$$
			is the identity.
		\item The augmentation makes $\MV(F)\in \Ch^{\leq 0}(\Sh_X)$ 
			a flabby 
			resolution of $j_*j^*F$.
	\end{enumerate}
	In addition, if $\Mu$ is an OS-algebra and 
	$F$ is a sheaf of dg-algebras, 
	then $\eta$ is multiplicative.
\end{prop}
\begin{proof}
	Take an open $V\subset X$.
	Since $\Gamma F|_V=F(V)=\prod_{x\in V}\tilde{F}_x$ for some sheaf $\tilde{F}$, we have a
	natural identification of complexes:
	$$\Gamma \MV(F)|_V =
	\prod_{x\in V}\bigoplus\limits^{t\in L}_{L_t\ni x}
		\mc{M}_t\otimes \tilde{F}_x$$
	with the differential $\partial$ on the RHS induced by the structure maps 
	of $\mc{M}$. On the other hand,
	$$j_*j^*F\ (V)=
	\prod\limits^{x\in V}_{x\not\in \bigcup\limits_{s>0\in L}L_s}\tilde{F}_x.$$

 	Let $S_x:=\{t\in L\mid L_t\ni x\}$.  
	By the definition of $L$,
	$S_x\subset L$ a sublattice with the maximal element $s_x\in L$, 
	thus $S_x=L[0,s_x]$. Moreover, $s_x>0$ iff 
	$x\not\in X-\bigcup\limits_{s>0} L_s$.
	So
	$$\left(\bigoplus\limits^{t\in L}_{L_t\ni x}\mc{M}_t\otimes \tilde{F}_x,
		\partial\right)=
	(\mc{M}[0,s_x],\partial)\otimes \tilde{F}_x$$
	is naturally quasi-isomorphic to $\tilde{F}_x$ if 
	$x\in X-\bigcup\limits^{s\in L}_{s>0} L_s$, and
	is acyclic otherwise.
	
	Recall that $\mc{M}_0=\mb{Z}$.
	Thus $\Gamma \MV_\bullet(F)|_V$ is the product over $x\in V$ 
	of complexes 
	\textit{equal} to $\tilde{F}_x$ for 
	$x\in V-\bigcup\limits^{s\in L}_{s>0} L_s$, and 
	of acyclic complexes for 
	$x\in \bigcup\limits^{s\in L}_{s>0} L_s$.
	This defines the inclusion in the obvious manner.
	Clearly both maps are quasi-isomorphisms.
	The multiplicativity follows by degree reasons
\end{proof}
Then for any sheaf $F\in \Ch(\Sh_X)$ we 
can apply the canonical Godement resolution $G(F)$ to obtain 
a quasi-isomorphism
\begin{equation}\label{eq:mv_resolution}
	\eps\colon j_* j^* G(F)\ar \MV(X,L,G(F),\Mu).
\end{equation}
In the case $F\in \Mon(\Sh_X)$, the quasi-isomorphism 
is multiplicative.

\subsection{Functoriality of \texorpdfstring{$\MV(X,L,(i_*i^!)^L G(F))$}{}}
We want to show that the quasi-isomorphism \eqref{eq:mv_resolution} is functorial 
with respect to arrangements. 
Let $\phi\colon [X',L',F',\Mu']\ar[] [X,L,F,\Mu]$ be a morphism of arrangements.
By Definition \ref{dfn:arrangement_posets:morphism}, 
such $\phi$ is given by a 
pair of maps $f\colon X'\ar X$, $g\colon (L,\le)\ar (L',\le)$ and a morphism $f^*F\ar F'$,
together with a natural morphism $g_*\Mu\ar \Mu'$.
In particular, we have the restriction $f\colon U'\ar U$,
which defines a morphism
$f^* (j_*j^*)_U G(F)\ar (j_*j^*)_{U'}G(F')$.
So $(j_*j^*)_U G(F)$ is functorial in the arrangement $[X,L,F,\Mu]$. In other words, 
we have a functorial equipment $[X,L,F,\Mu]\lto (j_*j^*)_U G(F)\in \Ch(\Sh_X)$
in the sense of Definition \ref{dfn:arrangement_posets:functorial_equipment}.

\begin{prop}\label{prop:MV_functorial}
	There is a quasi-isomorphism
	of functorial equipments 
	$$(j_*j^*)_U G(F)\ar[\eps]\MV(X,L,(i_*i^!)_{L_\bullet} G(F),\Mu)\in \Ch(\Sh_X)$$
	which is natural in the arrangement $[X,L,F,\Mu]$.
	For multiplicative arrangements
	this quasi-isomorphism is multiplicative.
\end{prop}
\begin{proof}
	First we show that $[X,L,F,\Mu]\lto \MV(X,L,(i_*i^!)_{L_{\bullet}}G(F),\Mu)$
	naturally produces a functorial equipment.
	Given $[X',L',F',\Mu']\ar[] [X,L,F,\Mu]$, we shall construct a natural $W'$-filtered 
	morphism
	$$f^*\MV((i_*i^!)^L G(F))\ar \MV((i_*i^!)^{L'}G(F'))$$
	which induces a commutative diagram in $\Ch(\Sh_{X'})$:
	\begin{equation}\label{eq:mv_funct}
\begin{tikzcd}
	f^*(j_*j^*)_U G(F)\arrow[r]\arrow[d,"f^*\eps"]	& 
		(j_*j^*)_{U'}G(F')\arrow[d,"\eps"]\\
	f^*\MV((i_*i^!)^L G(F))\arrow[r]& 
		\MV((i_*i^!)^{L'}G(F')).
\end{tikzcd}
\end{equation}
Note that since $g$ is a contraction, i.e.\ 
	$r(g(p))\le r(p)$ for all $p\in L$, 
	we have a natural $W'$-filtered morphism
	$\Tot(f^*(i_*i^!)^L G(F)\otimes \Mu)\ar 
		\Tot(g_*\big(f^*(i_*i^!)^L G(F)\otimes \Mu\big))$.
	On the other hand, for each $p\in L$, since
	$f^{-1}(L_p)\subset L'_{g(p)}$, we have 
	a natural inclusion $$f^*(i_*i^!)^L_{p}G(F)\ar 
		(i_*i^!)^{L'}_{g(p)}f^* G(F).$$
	It gives us a morphism of $(L,\le)$-objects:
	$f^*(j_*j^*)^L G(F)\ar g^* (j_*j^*)^{L'}f^* G(F).$
	By the projection formula \ref{lemma:lattice:projection_formula} 
	we obtain a composition:
	\begin{align*}
	\MV(f^*(i_*i^!)^L G(F))=\Tot(f^*(i_*i^!)^L G(F)\otimes \Mu)\ar
		\Tot(g_*(f^*(i_*i^!)^L G(F)\otimes \Mu))\ar \\
		\Tot(g_*(g^* (i_*i^!)^{L'}f^*G(F)\otimes \Mu))\ar
	\Tot((i_*i^!)^{L'}f^* G(F)\otimes g_*\Mu)\ar \\
		\Tot((i_*i^!)^{L'}f^*G(F)\otimes \Mu')=\MV((i_*i^!)^{L'}f^*G(F)),
	\end{align*}
	where the last arrow is induced by 
	the natural morphism $g_*\Mu\ar \Mu'$ 
	(Proposition \ref{prop:morphism_of_OS_complexes} or Corollary \ref{cor:morphism_of_OS_algebras}).
	Note that all morphisms in the composition are $W'$-filtered,
	and in the multiplicative case they are multiplicative too, so
	this gives us the required morphism and also a functorial equipment 
	$[X,L,F,\Mu]\lto \MV((i_*i^!)_{L_\bullet}G(F))$.
	It remains to note that the $0$-component of this morphism is 
		$\id\colon f^*(G(F))\ar f^*(G(F))$.
	Applying the natural quasi-isomorphism 
	$f^*G(F)\ar G(f^*F)$ followed by $G$ applied to $f^*F\ar F'$ we obtain the commutativity
	of diagram \eqref{eq:mv_funct}.
\end{proof}

So we obtain a natural $W'$-filtered morphism
$$\phi^*\colon f^*\MV((i_*i^!)^L G(F))\ar \MV((i_*i^!)^{L'}G(F')).$$
Passing to global sections gives :
	$$\phi^*\colon \Gamma_X(\MV((i_*i^!)^L G(F)))\ar 
	\Gamma_{X'}(\MV((i_*i^!)^{L'}G(F'))).$$
On the level of cohomology we obtain the expected morphism $H^*(U;j^*_U F)\ar H^*(U';j^*_{U'}F')$.
To refine this consider the natural morphisms
$g_{yx}\colon H^*(X;(i_*Ri^!)_x F)\to H^*(X;(i_*Ri^!)_y F)$
for $y\leq x$.
\begin{thrm}\label{thrm:ss_for_lattices}
The spectral sequence 
$${}_{W'}E^{pq}:={}_{W'}E^{pq}(\MV(X,i_*i^!)_\bullet G(F),\Mu)\Rightarrow H^*(U;j^*_U F)$$ is
functorial in $[X,L,F,\Mu]$ and satisfies the following:
\begin{enumerate}
    \item we have
	    $${}_{W'}E^{pq}_1=\bigoplus\limits^{x\in L}_{r(x)=-p} 
	    	H^q(X;(i_*Ri^!)_{L_x}F)\otimes_{\mb{Z}} \Mu_x[-r(x)];$$
    \item 
    the differential $d_1=\sum\limits_{x<:y}g_{yx}\otimes \partial_{xy}$.
\end{enumerate}
The terms ${}_{W'}E_r$ are functorially multiplicative for
multiplicative arrangements $[X,L,F,\Mu]$.
\end{thrm}

\begin{rmrk}\label{rmrk:ss_latice_equivariant}
Let $G$ be a finite group. An important application of the functoriality of the spectral sequence is a $G$-equivariant intersection poset
$[X,L,F,\OS(L)]$ formed by 
a geometric lattice $L$
and a $G$-equivariant sheaf 
$F$. 
Then the inclusion $U\ar X$
equips $H^*(U;j^* F)$ with a
$G$-invariant rank filtration $W'$.
In this case 
$G$ acts on $\OS(L)$ and on 
${}_{W'}E^{**}_1\Rightarrow \Gr_{W'}H^*(U;j^* F)$ 
in the obvious manner.
\end{rmrk}

\subsection{Lattice spectral sequence}
\label{sect:leray_ss}
More generally, if $(L,\le)$ is a local sup-lattice, not necessary graded, then
one can choose a weak grading $r\colon L\ar \mb{Z}$ and a contraction homomorphism 
$p\colon \mc{Q}\ar L$ from the local cubical lattice \ref{prop:p_contraction_homomorphism}.
In this case $(L,\le)$ may not admit an OS-algebra, so $\MV$ is not defined,
but the definiton of the arrangement poset $[X,L]$ still makes sense 
(omitting the additional data (2) in 
Definition \ref{dfn:arrangement_poset}), and following \ref{order_complex} 
we will use $\mc{D}:=p_*(\OS(\mc{Q}))$ instead of $\Mu$.
For example, one can take $[X,L]$ to be the intersection poset of the arrangement.
Proposition \ref{prop:atomic_complex_ss} implies that 
$\Tot(p^*((i_*i^!)_\bullet G(F))\otimes \OS(\mc{Q}))$ admits 
the \emph{lattice} filtration ${}_{L}W$ defined by $r$
and we immediately obtain
\begin{thrm}\label{thrm:lattice_ss}
	There is a spectral sequence
	$${}_{L}E^{pq}_1=\bigoplus\limits^{x\in L,i,j}_{r(x)=-p,i+j=p+q}H^i((i_*Ri^!)_{L_x}F\otimes H^j(\mc{D}_x)),$$
	converging to $\Gr^p_{{}_{L}W}H^{p+q}(U;Rj^*_U F)$.
\end{thrm}
Of course, if $r$ satisfies $r(x\vee y)\le r(x)+r(y)$ for all $x,y\in L$, then the spectral sequence ${}_{L}E_r$ is multiplicative.
By Proposition \ref{order_complex}
the terms $H^*(\mc{D}_x)$ are isomorphic to the homology of the order complex $L(0,x)$:
$$H^*(\mc{D}_x)=\tilde{H}_{-*-2}(L(0,x)).$$
After this identification we recover a spectral sequence studied in additive 
setting by Tosteson \cite{To}, and in dual terms by Petersen \cite{Petersen}.
We call ${}_{L}E_1$ the \emph{lattice spectral sequence}.

\subsection{Goresky-MacPherson formula and Leray spectral sequence}
\label{GorMac}
Consider a collection of real affine subspaces $L_i\subset V$.
The corresponding intersection poset $[V,L]$ admits a weak grading $r(x)=\codim(L_x/V)$ such that
$r(x\vee y)\le r(x)+r(y)$ for all $x,y\in L$, so $[V,L]$ is multiplicative.
In the case when $\codim L_i=c$ and $c$ divides $r(x)$ for all $x\in L$, i.e.\ it is an \emph{arrangement of 
$c$-equal codimension}, $(L,\le)$ is automatically locally geometric.

Since $H^*_{L_x}(V)$ is isomorphic to $\mb{Z}[-\codim(L_x)]$, the lattice spectral sequence provided by Theorem \ref{thrm:lattice_ss} reduces 
to a complex over~$\mb{Z}$ with vanishing differential:
\begin{equation}\label{eq:gormac}
E^{**}_1\simeq \bigoplus\limits^{x\in L} \tau_x\otimes H^{*}(\mc{D}_x)\Rightarrow H^*(V-\cup_i L_i),
\end{equation}
where $\tau_x$ is a formal variable 
of degree $\codim L_x$ corresponding to a generator in
$H^{\codim L_x}_{L_x}(X)$.
The convergence is multiplicative: the product of vertices $x,y$ is trivial 
if $r(x\vee y)\ne r(x)+r(y)$ (i.e.\ when $L_x,L_y$ are not transversal), and is induced by
the equality $\tau_x\cdot \tau_y=\pm \tau_{x\vee y}$ otherwise.

In terms of the order complex \ref{order_complex}, the expression for $E^{**}_1$ coincides with 
the (additive) Goresky-MacPherson formula: 
$$H^n(V-\cup_i L_i;\mb{Z})\simeq \bigoplus\limits_{x\in L}H^{n-\codim(L_x)}(\mc{D}_x).$$
Following \cite{To} it is not hard to prove the formula (which a priori stronger than the degeneration of \eqref{eq:gormac}).
Namely, by degree reasons there is a splitting of the canonical filtration $\tau^{\le}$
on the $(L,\le)$-object $p^*C^*_{L_\bullet}(V)$: all associated graded components are equivalent to 
$p^* H^i_{L_\bullet}(V)[-i]$, and are concentrated at the
vertices $x$ with $r(x)=i$, so all extensions are trivial. 
So there exists a quasi-isomorphism of $(L,\le)$-objects
$p^*C^*_{L_\bullet}(V)\zgqis p^*H^*_{L_\bullet}(V)$.
Hence
$\Tot(p^*H^*_{L_\bullet}(V)\otimes \OS(\mc{Q}))$ and $\Tot(p^*C^*_{L_\bullet}(V)\otimes \OS(\mc{Q}))$
are filtered quasi-isomorphic. 
It follows that 
$$E^{**}_1\simeq \bigoplus\limits_{x\in L} H^{*}(\mc{D}_x)[-\codim(L_x)],$$ 
degenerates and is additively (non canonically) isomorphic to $H^*(V-\cup_i Z_i;\mb{Z})$.
Similarly we obtain that the lattice filtration ${}_{L}W_p H^*(V-\cup_i Z_i)$ is given by the part
of \eqref{eq:gormac} with $r(x)\le p$.
Moreover $E^{**}_1$ in the form \eqref{eq:gormac} is multiplicatively isomorphic
to $\Gr^{{}_{L}W}_\bullet H^*(V-\cup_i L_i;\mb{Z})$.

We conclude this section with some observations on the Leray spectral sequence.
Consider an arrangement of smooth manifolds $Z_i\subset X$ with clean intersections,
i.e.\ locally $\{Z_i\subset X\}$ at each $p\in X$ is homeomorphic to a vector subspace arrangement.
As usual $j\colon U:=X-\cup_i Z_i\ar X$ is the complement.
Let $(L,\le)$ be the corresponding intersection poset with a weak grading $r(x)=\codim_{\mb{R}}(L_x/X)$.
In this case the stalk of $R^n j_* \ul{\mb{Z}}$ at $p$ is isomorphic to $q$-th cohomology
of the subspace complement. Moreover the sheaf 
$R^n j_*\ul{\mb{Z}}\simeq\mc{H}^n(\Tot(p^*(i_*Ri^!)_{L_\bullet})\otimes \OS(\mc{Q}))$ is equipped
with a filtration $W'$ given by Theorem \ref{thrm:lattice_ss}. 
Let $\nu_x$ denote the rank $1$ orientation local system of the normal bundle $L_x$ in $X$.
The Goresky-MacPherson formula \eqref{eq:gormac} 
provides a natural isomorphism of sheaves
$$\Gr^{p}_{W'} R^n j_*\ul{\mb{Z}}[-n]\simeq \bigoplus\limits^{x\in L}_{r(x)=-p}\tau_{x}\otimes H^{n-r(x)}(\mc{D}_x)[-n]$$
where $\tau_x=(i_x)_* \nu_x$ and $i_x\colon L_x\ar X$ is the inclusion.
Recall that $H^*(R^*j_*\ul{\mb{Z}})$ is isomorphic to the second page of the usual
Leray spectral sequence (omitting the differential).
So there is a spectral sequence $'E^{**}_1$ converging to the Leray spectral sequence:
$$'E^{pq}_1=H^{p+q}(\Gr^p_W \oplus_n R^nj_*\ul{\mb{Z}}[-n])=
\bigoplus\limits^{x\in L}_{r(x)=-p}H^{p+q-n}(L_x;\tau_x\otimes H^{n-r(x)}(\mc{D}_x)).$$
The natural Thom isomorphism 
$H^{*+r(x)}(L_x;\tau_x)\simeq H^*_{L_x}(X;\ul{\mb{Z}})$ gives:
$$'E^{pq}_1\simeq
\bigoplus\limits^{x\in L}_{r(x)=-p}H^{p+q-n+r(x)}_{L_x}(X;H^{n-r(x)}(\mc{D}_x)).
$$

We conclude that $'E^{ij}_1\simeq {}_{L}E^{ij}_1$.
Of course the differentials in both spectral sequences are different, and the spectral sequences converge to different groups.
We will see that in the case of arrangements of algebraic origin with geometric intersection poset, 
the weight considerations, together with the formality of the atomic complex \ref{thrm:atomic_complex} imply that
${}_{L}E_2={}_{L}E_\infty$ over $\mb{Q}$. 
Then the theory of $l$-adic sheaves or mixed Hodge modules suggests that 
$R^n j_*\ul{\mb{Q}}$ is a sheaf of mixed Hodge structures, so its weight filtration 
splits \cite{CH20}. On the other hand, by Proposition \ref{weight_lattice_filtrations}
the weight filtration on 
the MHS $R^n j_*\ul{\mb{Q}}$ coincides with the lattice filtration.
Hence after forgetting the differentials $'E^{**}_1$ in fact isomorphic to
the second page of the usual Leray spectral sequence. 

One can say that in the above assumptions the lattice spectral sequence 
``accumulates'' all differentials $d_r$ of the Leray spectral sequence.
A basic example of this observation is provided by an arrangement of $c$-equal codimension. 
In this case weight considerations, see~\cite{Totaro} or more generally~\cite{Weber}, 
imply that the only term of the corresponding Leray spectral sequence 
with non vanishing differential is $E^{**}_{2c}$. On the other hand, 
$d_{2c}$ appears as the differential of the first term of the lattice spectral sequence ${}_{L}E^{**}_1$.

\subsection{\v{C}ech construction
\texorpdfstring{$\check{C}$}{}}
\label{sect:cech_model}

\subsubsection{The construction}
Assume $\mc{Q}$ is a local cubical lattice and  
$A$ is a $\mc{Q}^{\op}$-object in $\CMon(\ChF)$.
We are going to define a partial $\mc{Q}$-algebra in $\ChF$, 
denoted by $\check{C}(A)_\bullet$.
Using this we will describe M\"obius inversion formula, by means of
an equivalence $\hat{A}^\bullet\zgqis A^\bullet$ of $\mc{Q}^{\op}$-monoids. 
The replacement $A$ by $\hat{A}$ 
is suitable for the construction
of mixed Hodge diagram in \ref{sect:compl:mvhc}.

Recall that the elements of $\mc{Q}$ are the pairs $\bar{I}=(x,I)$ 
where $x\in L,I\subset \Atoms(\mc{Q})$ are
such that $x\in \os{\circ}{\sup}(I)$.
To simplify the notation let us describe $\check{C}(A)$ for cubical lattices 
$\mc{C}$. 
Then it will be clear that all constructions commute with the restriction to $\mc{Q}[0,x]$.
Since each interval $\mc{Q}[0,x]$ is a cubical lattice, we will 
formulate the corresponding statements for an arbitrary local cubical lattice $\mc{Q}$.

Let $\mc{C}$ be a cubical lattice. Let 
$N=\Atoms(\mc{C})$ and
denote by $g_{JI}\colon A^{I}\to A^J$ the structure maps for each $I\subset J$.
For simplicity let us assume that $N$ is ordered.
Let $d\tau^K$ be a monomial in the Grassmann algebra 
$\Lambda\langle d\tau_i\mid i\in N\rangle$ with
$\deg d\tau^K=|K|$.
Equip
\begin{equation}
	\check{C}(A)_I=\bigoplus\limits^K_{K\subset I} A^K\cdot d\tau^K
\end{equation}
with the \v{C}ech differential 
\begin{equation}\label{eq:dfn:cech_differential}
	\delta_{\Cech}=\sum\limits^k_{K+\{k\}\subset I}
	g_{K+\{k\},K}(-)\otimes d\tau^k\wedge -.
\end{equation}
Using the notion of the tensor product (\ref{dfn:chf_tensor_product}) 
we have $A^K\cdot d\tau^K\simeq A^K[-|K|]\in \ChF$, and 
$\check{C}(A)_I$ is an iterated cone (Definition \ref{dfn:iterated_cone}) 
in $\ChF$:
$$\check{C}(A)_I=
	\left[A^\emptyset{\longrightarrow}
		\left[\bigoplus\limits^{K\subset I}_{|K|=1}
				A^K{\longrightarrow}\left[
	\bigoplus\limits^{K\subset I}_{|K|=2}A^K{\longrightarrow}
		\left[\ldots\right]\right]\right]\right]
	\in \ChF.$$
For $I\subset I'$, we define the structure map
$$\check{g}_{II'}\colon \check{C}(A)_{I'}\to \check{C}(A)_I$$ 
by mapping the component 
$A^{K'}\cdot d\tau^{K'}\subset \check{C}(A)_{I'}$, $K'\subset I'$
identically if $K'\subset I$, and to zero otherwise.

For all disjoint $I,I'\subset N$ define multiplication tensors
$$\check{m}^{II'}_{I,I'}\colon 
	\check{C}(A)_I\otimes \check{C}(A)_{I'}\to 
		\check{C}(A)_{II'},$$ 
which act on the components
$A^K\cdot d\tau^K,K\subset I$ and 
$A^{K'}\cdot d\tau^{K'},K'\subset I'$
by first applying the inclusion of 
$g_{KK',K}\otimes g_{KK',K'}$ (with the usual sign rule) and then 
the multiplication  $m\colon A^{KK'}\otimes A^{KK'}\to A^{KK'}$.

A straightforward check shows the following:
\begin{prop}\label{prop:cech_construction}
	This construction provides $\check{C}(A)_\bullet$ with
	a structure of a commutative partial $\mc{Q}$-algebra in $\ChF$.
	It is independent of the ordering of $N$.
\end{prop}

\subsubsection{Homological M\"obius inversion}
For each $\bar{N}\in \mc{Q}$ let
\begin{equation}\label{eq:mobius_inversion}
\hat{A}^{\bar{N}}:=
	\check{C}(A)|_{\mc{Q}[0,\bar{N}]}*\OS(\mc{Q}[0,\bar{N}])
	\in \ChF
\end{equation}
be the convolution equipped with the rank filtration $W'$.
Note that by Remark \ref{rmrk:lattices:product_with_partial},
$\hat{A}^{\bar{I}}$ is equipped with a commutative algebra structure
which however \emph{does not} satisfies the Leibniz rule. On the other hand,
we naturally have $\Gr^{W'}\hat{A}^{\bar{I}}\in \CMon(\ChF)$.

Below we will equip $\hat{A}^{\bar{N}}$ with a natural multiplication 
$\hat{m}$ obeying the Leibniz rule.
Without loss of generality one can replace $A$ by its restriction to
$\mc{Q}_{[0,\bar{N}]}$ and assume that $A$ is cubical.
Let $d\mu_i,i\in N$ be formal variables of degree $-1$
and let $\iota_i$ denote the corresponding derivation of degree $1$.
For a pair $I\subset K\in \mc{C}$ denote by $(d\mu_I)^K$ 
the monomial $d\mu_I$ that corresponds to $K$.

Additively, we have the following isomorphism 
\begin{equation}\label{eq:hat_A_expression}
	\hat{A}^{N}=
\bigoplus\limits_{I\subset K}^{I,K\subset N}
	A_{I}\otimes \mb{Z}\cdot d\tau^I\cdot d\nu^K
\simeq 
\bigoplus\limits_{I\subset K}^{I,K\subset N}
	A^{K\sm I}\otimes \mb{Z}\cdot (d\mu_I)^K,
\end{equation}
where the RHS is equipped with a differential 
$\delta+d=\sum \delta_{I'\subset K'}^{I\subset K}+d$
	described as follows.
	The differential $d$ is the inner differential corresponding 
	to the complexes $A^{K\sm I}\in \ChF$,
	while
	the only non-zero components of $\delta$,
	\begin{equation}\label{eq:dif_cech}
 \delta_{I'\subset K'}^{I\subset K}\colon A^{K\sm I}\otimes \mb{Z}\cdot (d\mu_I)^K\ar
		A^{K'\sm I'}\otimes \mb{Z}\cdot (d\mu_{I'})^{K'},
  \end{equation}
	are given by
$$
\delta_{I-p\subset K}^{I\subset K}\colon 
	A^{K\sm I}\cdot (d\mu_I)^K\ar[g\otimes \iota_p] 
		A^{K\setminus I-p}\cdot (\iota_p d\mu_{I})^K,$$
$$\delta_{I-p\subset K-p}^{I\subset K}\colon 
	A^{K\sm I}\cdot (d\mu_I)^K\ar[g\otimes \iota_p] 
	A^{K\sm I-p}\cdot (-\iota_p d\mu_{I})^{K-p}.$$

	One can describe the components of 
	the identification \eqref{eq:hat_A_expression} 
	by sending 
	$d\tau^I\cdot d\nu_K$ to $(-1)^K (\iota_I d\mu_K)^K$.
	Imposing the usual sign rule we recover the definition of
	$\hat{A}^N$ in terms of \eqref{eq:mobius_inversion}, 
	in particular we have $[d,\delta]=0$.

We will define the product $\hat{m}$ by specifying its components 
\begin{equation}
\label{eq:components_hat_m}
\hat{m}_{I\subset K,I'\subset K'}\colon A^{K\sm I}\cdot (d\mu_I)^K\otimes 
	A^{K'\sm I'}\cdot (d\mu_{I'})^{K'}\ar
	A^{KK'\sm II'}\cdot (d\mu_{I}d\mu_{I'})^{KK'}.
\end{equation}
Note that $K\sm I\scup K'\sm I'\supset KK'\sm II'$ .
If $K\sm I\scup K'\sm I'=KK'\sm II'$ then we define the component
$\hat{m}_{I\subset K,I'\subset K'}$ by using the sign rule and the natural composition
$$A^{K\sm I}\otimes A^{K'\sm I'}\ar[g\otimes g] 
A^{K\sm I \scup K'\sm I'}\otimes A^{K\sm I\scup K'\sm I'}\ar[m] 
A^{K\sm I\scup K'\sm I'}=A^{KK'\sm II'},$$
where $m$ is the multiplication in $A$.
Otherwise we set $\hat{m}_{I\subset K,I'\subset K'}=0$.

This gives us a product $\hat{m}\colon \hat{A}^{\bar{N}}\otimes \hat{A}^{\bar{N}}\ar \hat{A}^{\bar{N}}$.
The proof of the following technical lemma is given in 
Appendix \ref{lemma:mob_inversion_proof}.
\begin{lemma}\label{lemma:mobius_inversion_multiplication}
	The product $\hat{m}$
	provides
$\hat{A}^{\bar{N}}$ with a $\mc{Q}|_{[0,\bar{N}]}$-chain algebra structure.
\end{lemma}
\QEDB

\begin{rmrk}
It seems likely that the cdga $\bigoplus\limits_{I\subset K}^{I,K\subset N}\mb{Z}\cdot (d\mu_I)^K$ from the proof of Lemma~\ref{lemma:mobius_inversion_multiplication} has a description which makes the proof of the associativity and Leibniz rule more straightforward. We tried but failed to find such an interpretation,
\end{rmrk}

In particular $\hat{A}^{\bar{I}}\in \CMon(\ChF)$ admits a multiplicative rank filtration $W'$.
It follows from the construction that the multiplication on $\Gr^{W'}\hat{A}^{\bar{I}}$
induced by $\hat{m}$
coincides with the one induced by the convolution \eqref{eq:mobius_inversion}.
By the naturality of restrictions it follows also that $\hat{A}=\check{C}(A)\otimes \OS(\mc{Q})$ 
is a $\mc{Q}$-chain algebra.
\begin{rmrk}
By definition, the differential on the complex $\hat{A}^{\bar{I}}$ 
is the sum $\delta_{\Cech}+\partial+d$ of supercommuting differentials. Here
$\partial$ is the Koszul differential which comes from the convolution with $\OS(\mc{Q}[0,I])$,
and $d$ is the inner differential of $A^\bullet$.
Note that by Remark \ref{rmrk:lattices:product_with_partial}
the partial algebra structure on $\check{C}(A)$ 
\emph{does not}
induce a monoid structure on the convolution $\hat{A}^{\bar{I}}$.
Lemma \ref{lemma:mobius_inversion_multiplication} gives another natural
multplication which satisfies the Leibniz rule with respect to 
the differentials $\delta_{\Cech}+\partial$ and $d$ respectively.

Let us illustrate this by a basic example.
Consider the cubical diagram $A^0\ar[f] A^1$ on a single atom. 
Then $\hat{A}^0=A^0$ and
$\hat{A}^1=[A^0\ar A^1]\cdot d\nu\oplus  A^0$ where $[A^0\ar A^1]=A^0\oplus A^1\cdot dx$.
Hence $\hat{A}^1=[A^0\ar[f\oplus -\id] A^1\oplus A^0]=A^0[1]\oplus(A^1\oplus A^0)[0]$.
The usual multiplication on this convolution vanishes on $A^1[0]$ and 
does not satisfies the Leibniz rule:
given $\bar{a}\in A^0[1],b\in A^1[0]$ we have $\bar{a}b=0$, but 
$d\bar{a} b-(-1)^{\deg a}\bar{a}db=(-1)^{\deg a}f(a)b\in A^1[0]$.
On the other hand one can consider $(A^1\oplus A^0)[0]$ as the trivial extension of $A^0$ by 
the algebra $A^1$ over $A^0$. Then $A^0$ is a quotient algebra of 
$(A^1\oplus A^0)[0]$.
The resulting dg-algebra can be described in terms of 
Lemma \ref{lemma:mobius_inversion_multiplication}
as 
$A^0\cdot \mu\oplus (A^0\oplus A^1)$ with the differential 
$\id\otimes \iota-f\otimes \iota$ where $\mu$ is a formal
variable with $\deg \mu=-1$ and $\iota$ is the corresponding contraction.

\end{rmrk}

The following theorem can be thought of as a categorification of 
the M\"obius inversion formula in the local cubical case.
\begin{thrm}\label{thrm:mobius_trick}
	The natural inclusions
	$$g^{\hat{A}}_{\bar{I}\bar{J}}\colon \hat{A}^{\bar{J}}\ar 
		\hat{A}^{\bar{I}}$$ 
	for $\bar{J}\le \bar{I}\in \mc{Q}$ make 
		$\hat{A}^\bullet$ into a $\mc{Q}^{\op}$-object in 
	$\CMon(\ChF)$.
	There is a natural quasi-isomorphism 
	of $\mc{Q}^{\op}$-objects in 
	$\CMon(\ChF)$:
	$$A^\bullet\zgqis \hat{A}^\bullet.$$
\end{thrm}
\begin{proof}
	It is enough to show this 
	for all intervals $\mc{C}\subset \mc{Q}$ 
	containing $0\in \mc{Q}$.
	If $N\subset M\in \Atoms(\mc{C})$,  then
	we have a natural morphism of monoids $\hat{A}^N\ar \hat{A}^M$,
	which is easily seen from \eqref{eq:hat_A_expression}.
	For the second assertion we construct a zig-zag.
	Fix $J\subset \Atoms(\mc{C})$.
	Let $A'^\bullet$ be the constant $\mc{C}^{\op}$-object in $\ChF$ 
	with value $A^J$ and the identity structure maps.
	Consider the following zig-zag in $\CMon(\ChF)$:
	$$\hat{A}^J\ar[\alpha] \hat{A}'^J 
		\al[\beta] A^J.$$
	Here  $\alpha$ is $J$-component of the natural morphism 
	$\hat{A}|_{[0,J]}\ar[\alpha] \hat{A'}|_{[0,J]}$,
	and $\beta$ is given by the inclusion 
	$A^J=A^J\cdot (d\mu_{\emptyset})^{\emptyset}\subset \hat{A}'^J$.
	It is clear that $\alpha$ and $\beta$ are morphisms in 
		$\CMon(\ChF)$.

	Showing that $\alpha,\beta$ are quasi-isomorphisms is straightforward.
	It will be easier to work using the additive expression 
	\eqref{eq:mobius_inversion}.
	By definition,
	$$\hat{A}^J=\bigoplus\limits^{I}_{I\subset J}\check{C}_I\cdot d\nu_I=
	\bigoplus\limits^{K,I}_{K\subset I\subset J}
		A^K\cdot d\tau^K\cdot d\nu_I.$$
	The differential acting on $\hat{A}^J$ is equal to the sum 
	$d_A+\delta_{\check{C}}+\partial_{K}$ of super-commuting differentials
	where $d_A$ is the inner differential on 
	$A^\bullet$, $\delta_{\check{C}}$ is the 
	\v{C}ech differential \eqref{eq:dfn:cech_differential}, 
	$\partial_K$ is the Koszul differential
	coming from $\OS(\mc{C}[0,J])$.

		We have $$\hat{A}'^J=\check{C}(A')|_{[0,J]}*\OS(\mc{C}[0,J])=
		\bigoplus\limits^{K,I}_{K\subset I\subset J}
		A^J\cdot d\tau^K\cdot d\nu_I.$$

	Let $F$ and $G$ be the filtrations corresponding to the monomial 
	degrees of 
		$d\nu_*$ and $d\tau^*$ respectively.
	From
	$$\hat{A}'^J=A^J\cdot 
		\bigoplus\limits^I_{I\subset J}
			\bigoplus\limits^K_{K\subset I} d\tau^K\cdot d\nu_I$$
	we have that
	$\beta\colon A^J\to A^J=\Gr^0_{F} \hat{A}'_J$ is the identity 
	while $Gr^{\neq 0}_F \hat{A}'_J$ is acyclic. Hence
	$\beta$ is a quasi-isomoprhism.
	Similarly, from
	$$\hat{A}^J=\bigoplus\limits_{K}A^K\cdot d\tau^K\cdot 
		\bigoplus\limits^I_{K\subset I\subset J}d\nu_I$$
	we immediately see that
	$\Gr^{\neq |J|}_G \hat{A}^J$ as well as 
	$\Gr^{\neq |J|}_G \hat{A}'^J$ are acyclic.
	On the other hand, the induced morphism 
	$\alpha\colon \Gr^{|J|}_G \hat{A}^J\to \Gr^{|J|}_G \hat{A}'^J$ 
	is the identity
	map on $A^J\cdot d\tau^J\cdot d\nu_J$. Hence $\alpha$ 
	is a quasi-isomorphism.
	
	Finally, if $T\hookrightarrow J$, then the above construction gives
	a commutative diagram:
	\begin{equation*}
	\begin{tikzcd}
		\hat{A}^J\arrow[r,"\alpha"]\arrow[d]& 
			\hat{A}'^J\arrow[d] 		&	
				\arrow[l,swap, "\beta"] A^J\arrow[d,"g"]\\
		\hat{A}^T\arrow[r]             & 
			\hat{A}'^T                    &
				\arrow[l]     A^T.\\ 
	\end{tikzcd}
	\end{equation*}
	This provides a natural quasi-isomorphism 
	$A^\bullet\zgqis \hat{A}^\bullet$ of $\mc{C}^{\op}$-objects
	in $\CMon(\ChF)$.
\end{proof}
\begin{rmrk}
	Clearly all results above hold additively or 
	without commutativity assumption.
\end{rmrk}

Finally, let us note that the construction is functorial 
with respect to a homomorphism $g\colon \mc{Q}\ar\mc{Q}'$ of 
locally cubical lattices such that $r(g(\bar{I}))=r(\bar{I})$ for all $\bar{I}\in \mc{Q}$.
We have the tautological order-preserving map of posets 
$g^{\op}\colon \mc{Q}^{\op}\ar \mc{Q}'^{\op}$.
Let $B'$ be a $\mc{Q}'$-object in $\CMon(\ChF)$.
\begin{prop}\label{prop:cech:functoriality_along_lattices}
	There is a natural morphism of partial $\mc{Q}$-algebras in $\ChF$:
	$$G\colon \check{C}({g^{\op}}^*B')\ar g^* \check{C}(B').$$
\end{prop}
\begin{proof}
	As above, it suffices to consider the case of a morphism of cubical lattices 
	$g\colon \mc{C}\ar \mc{C}'$.
	More generally, let $B$ be a $\mc{C}^{\op}$-object in $\CMon(\ChF)$, and let 
	$\psi\colon B\ar {g^{\op}}^* B'$ be a morphism.
	By our assumption on $g$,
	the restriction $g\colon \Atoms(\mc{C})\ar \Atoms(\mc{C}')$ is a well-defined 
	injective map.
We define $G$ component-wise using the following diagram
\begin{equation*}
\begin{tikzcd}
	\check{C}(B)_I\arrow[r,"="]& 
	\bigoplus\limits^{K}_{K\subset I} B^{K}\otimes_{\mb{Z}} \mb{Z}\cdot d\tau^K\arrow[d,"G_I"]\\
	g^*\check{C}(B')_I\arrow[r,"="]&
	\bigoplus\limits^{K'}_{K'\subset g(I)}B'^{K'}\otimes_{\mb{Z}} \mb{Z}\cdot d\tau^{K'},\\
\end{tikzcd}
\end{equation*}
where $G_I$ is determined by $\psi^K\colon B^K\ar B'^{g(K)}$, and 
the morphism of algebras
$$\Lambda\langle d\tau^i\mid i\in \Atoms(\mc{C})\rangle\ar
	\Lambda\langle d\tau^j\mid j\in \Atoms(\mc{C}')\rangle$$
 induced by $g\colon \Atoms(\mc{C})\ar \Atoms(\mc{C}')$.
 So the component $B^{K}\otimes \mb{Z}\cdot d\tau^K$ maps to $B'^{K'}\otimes \mb{Z}\cdot d\tau^{K'}$ 
for $K'=g(K)$ as $\pm \id$.
In this case it is straightforward to check that the components $G_I$ commute 
with $\delta_{\Cech}$ and together give us a morphism $G$ of partial $\mc{C}$-algebras.
\end{proof}

Note that the results of this section are independent on the ordering of $\Atoms(\mc{Q})$ and also 
have straightforward additive versions, 
which do not involve multiplication.

\subsection{\v{C}ech model for \texorpdfstring{
$i_*Ri^!\mc{F}$}{}}
Here we apply the previous construction to $\mc{A}=\Sh_X$ 
in the following form.
Let $[X,\mc{Q},\Mu]$ be an arrangement.
We have a collection of open subsets
$U_i=X-\mc{Q}_i\subset X, i\in\Atoms(\mc{Q})$. 
Set $U_{\bar{I}}:=U_I=\cap_{i\in I}U_i$, and consider 
$U_\bullet$ as a $\mc{Q}$-variety.
We write $j_{\bar{I}}\colon U_{\bar{I}}\ari X$
and $i_{\bar{I}}\colon Z_{\bar{I}}:=X-U^I\ari X$. 
As in Section \ref{sect:compl:mv_resolution} 
for an open $V\subset X$ and a closed $Z\subset X$
we write ${(j_*j^*)}_{V}={j_V}_* j^*_V$,  $(i_*i^!)_Z={i_Z}{i^!_Z}$.

Pick an
$F\in \CMon(\ChF(\Sh_X))$.
Then
$(Rj_*j^*)^{\mc{Q}^{\op}}_{U_\bullet}F$
form a $\mc{Q}^{\op}$-monoid in $D(\Sh_X)$. 
The natural transformations
$(Rj_*j^*)_{U_{\bar{I}}}\to (Rj_* j^*)_{U_{\bar{J}}}$ 
for all $\bar{I}\le \bar{J}$ define the structure maps.
We want to model 
$(Rj_*j^*)^{\mc{Q}^{\op}}_{U_\bullet} F$ on the level of complexes.

Namely, we will consider a $\mc{Q}^{\op}$-object $\mc{F}^{\mc{Q}^{\op}}$ in 
$\CMon(\ChF(\Sh_X))$ satisfying some conditions.
Set $\mc{F}^{\mc{Q}}:=\check{C}(\mc{F}^{\mc{Q}^{\op}})$ and 
let $\hat{\mc{F}}^{\mc{Q}^{\op}}_\bullet$ be the $\mc{Q}^{\op}$-monoid 
provided by Theorem \ref{thrm:mobius_trick}.
We have
\begin{cor}\label{cor:mobius_trick_sheaves}
 There is a quasi-isomorphism of $\mc{Q}^{\op}$-objects in $\CMon(\ChF(\Sh_X))$:
	$$\mc{F}^{\mc{Q}^{\op}}_\bullet\zgqis 
	\hat{\mc{F}}^{\mc{Q}^{\op}}_\bullet.$$	
\end{cor}

Assume that for each $\bar{I}\in \mc{Q}^{\op}$ the following holds:
\begin{enumerate}\label{dfn:FCop:conditions}
	\item $\mc{F}^{\mc{Q}^{\op}}_0=F$.
	\item The natural morphism 
		$\mc{F}^{\mc{Q}^{\op}}_{\bar{I}}\to 
(Rj_* j^*)^{\mc{Q}^{\op}}_{U_{\bar{I}}}\mc{F}^{\mc{Q}^{\op}}_{\bar{I}}$
		is a quasi-isomorphism.
	\item $\eps$ induces a quasi-isomorphism 
		$(Rj_*j^*)^{\mc{Q}^{\op}}_{U_{\bar{I}}}\mc{F}\to 
		(Rj_*j^*)^{\mc{Q}^{\op}}_{U_{\bar{I}}}\mc{F}^{\mc{C}^{\op}}_{\bar{I}}$.
\end{enumerate}
\begin{prop}\label{prop:cech_for_sheaves}
	For $\mc{F}^{\mc{Q}^{\op}}$ and $F=\mc{F}^{\mc{Q}^{\op}}_0$ as above, 
	there are natural quasi-isomorphisms 
	\begin{enumerate}
	\item of $\mc{Q}^{\op}$-objects in $\Mon(\ChF(\Sh_X))$:
		$$(j_*j^*)_{U_\bullet} G(F)\zgqis 
			\mc{F}^{\mc{Q}^{\op}}_\bullet;$$
	\item of $\mc{Q}^{\op}$-objects in $\CMon(\ChF(\mb{Q}-\Sh_X))$:
		$$(R_{TW}j_*j^*)_{U_\bullet} (F)\zgqis 
			\mc{F}^{\mc{Q}^{\op}}_\bullet;$$
	\item of partial $\mc{Q}$-algebras in $\ChF(\Sh_X)$:
		\begin{equation}\label{eq:cech_iso}
			(i_*i^!)_{Z_\bullet} G(F)\zgqis 
				\check{C}(\mc{F}^{\mc{Q}^{\op}})_\bullet.
		\end{equation}
	\item of $\mc{Q}$-chain algebras in $\ChF(\Sh_X)$:
		\begin{equation}
			(i_*i^!)_{Z_\bullet}G(F)*\OS(\mc{Q})
			\zgqis  
			\check{C}(\mc{F}^{\mc{Q}^{\op}})*\OS(\mc{Q}).
		\end{equation}
	\end{enumerate}
\end{prop}
\begin{proof}
	In the following diagram
	\begin{equation}\label{eq:domik}
	\begin{tikzcd}
(j_*j^*)_{U_{\bar{I}}} G(F)\arrow[rd]&
	&\mc{F}^{\mc{Q}^{\op}}_{\bar{I}}\arrow[ld]\\
	&(j_*j^*)_{U_{\bar{I}}}G(\mc{F}^{\mc{Q}^{\op}})_{\bar{I}}&
	\end{tikzcd}
	\end{equation}
	both arrows are quasi-isomorphisms for each $\bar{I}\in \mc{Q}^{\op}$.
	The natural tranformations 
	$(j_*j^*)_{U_{\bar{I}}}\to (j_*j^*)_{U_{\bar{J}}}$ for all 
	$\bar{I}\le \bar{J}$ 
	make the terms 
	$(j_*j^*)_{U_{\bar{I}}}G(\mc{F}^{\mc{Q}^{\op}}_{\bar{I}})$ into 
	$\mc{Q}^{\op}$-monoids. 
	This gives us a quasi-isomorphism of $\mc{Q}^{\op}$-objects
	in $\Mon(\ChF(\Sh_X))$:
	$$(j_*j^*)_{U_\bullet} G(F)\zgqis \mc{F}^{\mc{C}^{\op}}_\bullet.$$
	This proves the first assertion. 
	Replacing $G$ by $G_{TW}$ and repeating the argument 
	we obtain the second. 
	In particular we have a quasi-isomorphism of partial 
	$\mc{Q}$-algebras:
	$$\mc{F}^{\mc{Q}}\zgqis 
		\check{C}((j_*j^*)^{\mc{Q}^{\op}}_{U_\bullet} G(F)).$$
	
	For the third assertion, consider the natural morphism of 
		partial $\mc{Q}$-algebras
	\begin{equation}\label{eq:supp_to_cech}
	(i_*i^!)^{\mc{Q}}_{Z_\bullet} G(F)\ar 
		\check{C}((j_*j^*)^{\mc{Q}^{\op}}_{U_\circ} G(F))_\bullet,
	\end{equation}
	induced by 
	$$(i_*i^!)^{\mc{Q}}_{Z_I} G(F)\hookrightarrow 
		G(F)\otimes \mb{Z}\cdot d\tau^0\subset 
	\check{C}((j_*j^*)^{\mc{Q}}_\circ G(F))_{\bar{I}}.$$
	To prove the assertion it is enough to show 
	for an arbitrary Godement sheaf
	$\mc{F}\in \Sh_X$, that the morphism
	$(i_*i^!)_{Z_I}F\ar 
		\check{C}((j_*j^*)^{\mc{Q}^{\op}}_{U_\circ} \mc{F})_{\bar{I}}$
	is a quasi-isomorphism.
	We proceed similarly to the proof of 
		Proposition \ref{prop:mv_resolution}.
	Assume $F$ is a Godement sheaf, 
	so $F(V)=\prod\limits_{p\in V}\tilde{F}_p$ for some sheaf 
		$\tilde{\mc{F}}$.
	Let
	$$C:=\check{C}((j_*j^*)^{\mc{Q}^{\op}}_{U_\bullet} F)_{\bar{I}}=
		\bigoplus\limits^K_{K\subset I}
		(j_*j^*)_{U_K} F\otimes \mb{Z}\cdot d\tau^K.$$

	We will show that the composition
	\begin{equation}\label{eq:mobius_trick_prop_1}
		(i_*i^!)_{Z_I} F\hookrightarrow F\otimes 
		\mb{Z}\cdot d\tau^{\emptyset}\subset C^0\subset C
	\end{equation}
	is a quasi-isomorphism.
	Since $(j_*j^*)_{U_K}F\ (V)=
	\prod\limits_{p\in V\cap U_K}\tilde{F}_p$, 
	we have the decomposition
	$$C(V)=\prod\limits_{p\in V}\left(\tilde{F}_p\otimes
	\bigoplus\limits^{K}_{\substack{K\subset I\\ p\in U_K}} 
	\mb{Z}\cdot d\tau^K\right )\in \Ch(\mb{Z}-\Mod)$$
	as a direct product of Koszul complexes.
	The complexes 
	$\tilde{F}_p\otimes 
		\bigoplus\limits^K_{\substack{K\subset I\\ p\in U_K}} 
			\mb{Z}\cdot d\tau^K$
	are acyclic for $p\in U^I$, and are equal to
	$\tilde{F}_p\cdot d\tau^\emptyset$ for $p\in Z_I=X-U^I$.	
	\
	Hence the extension by zero of \eqref{eq:mobius_trick_prop_1} 
	$$(i_*i^!)_{Z_I}F\ (V)\simeq 
	\prod\limits_{p\in Z_I\cap V}\tilde{F}_p\otimes 
		\mb{Z}\cdot d\tau^\emptyset
		\hookrightarrow \prod\limits_{p\in V}
		\left(\tilde{F}_p\otimes
		\bigoplus\limits^{K\subset I}_{p\in U_K} 
		\mb{Z}\cdot d\tau^K\right)=C(V)$$
	is a quasi-isomorphism for all $V\subset X$.

	The last assertion follows by the fact that
	\eqref{eq:supp_to_cech} induces a quasi-isomorphism of
	$\mc{Q}$-chain algebras
	$$(i_*i^!)^{\mc{Q}}_{Z_\bullet}G(F)*\OS(\mc{Q})\ar
	\check{C}((j_*j^*)^{\mc{Q}^{\op}}G(F))*\OS(\mc{Q}).$$
	Together with \eqref{eq:domik} this proves the claim.
\end{proof}
An easy commutative example over $\mb{Q}$ is given by
$\mc{F}^{\mc{Q}^{\op}}=(R_{TW}j_*j^*)^{\mc{Q}^{\op}}_{U_\bullet}F$.

\subsection{Constructions
\texorpdfstring{
	$\MVCC(X,L,\mc{F}^{\mc{C}^{\op}},\Mu)$ 
	and 
$\MVC(X,L,\mc{F}^{\mc{C}^{\op}},\Mu)$}{}}
\label{sect:mv_commutative}

Let
$\phi\colon [X',L',F',\Mu']\ar[] [X,L,F,\Mu]$
be a morphism of arrangements corresponding to a morphism
$f\colon X'\ar X$, 
a homomorphism $g\colon L'\ar L$ preserving the rank function
 and a morphism 
$T_0\colon f^*F\ar F'$ (Definition \ref{dfn:arrangement_posets:morphism}).
Let $\mc{Q},\mc{Q}'$ be the local 
cubical lattices
corresponding to $L,L'$ respectively.
The homomorphism $g\colon L\ar L'$ induces a homomorphism
$\tilde{g}\colon \mc{Q}\ar \mc{Q}'$ given by $\tilde{g}(\tilde{I})=(g(x),g(I)),\tilde{I}=(x,I)\in \mc{Q}$.
Following  \ref{dfn:FCop:conditions},
consider sheaves $\mc{F}^{\mc{Q}^{\op}}_\bullet$ and 
$\mc{F}^{\mc{Q}^{'\op}}_\bullet$ 
on $X$ and $X'$ respectively,
equipped with a morphism 
$T\colon f^*\mc{F}^{\mc{Q}^{\op}}\ar 
	\tilde{g}^{\op*} \mc{F}^{\mc{Q}^{'\op}}$ in 
	$\CMon(\Ch(\Sh_{X'}))$ such that 
	$T_0\colon f^*\mc{F}_{0}\ar \mc{F}'_{0}$ 
is the restriction of $T$.

Proposition \ref{prop:cech:functoriality_along_lattices} 
provides a composition 
\begin{equation}\label{eq:cech:functoriality_along_lattices}
f^*\check{C}(\mc{F}^{\mc{Q}^{\op}})=
	\check{C}(f^*\mc{F}^{\mc{Q}^{\op}})\ar 
	\check{C}({{\tilde{g}}^{\op*}}{\mc{F}}^{\mc{Q}^{'\op}})\ar
{\tilde{g}}^{*}\check{C}(\mc{F}^{\mc{Q}^{'\op}}).
\end{equation}
On the other hand we have 
\begin{equation}\label{eq:support:functoriality_along_lattices}
f^*(i_*i^!)^L G(\mc{F})\ar 
	g^*(i_*i^!)^{L'}f^*G(\mc{F})\ar[T] 
		g^*(i_*i^!)^{L'}G(\mc{F}'),
\end{equation}
where the first morphism exists by the definition of $\phi$ and 
the second is induced by $T$.

Let $I\colon (L,\le)\ar \mc{Q}$ be the order-preserving map 
given by $x\ar I(x)=(x,\{a\in \Atoms(L)\mid a\le x\})$.
Recall that $L_J:=\cap_{j\in J}L_j\subset X, J\subset \Atoms(L)$.
By the definition of the arrangement $L_{I(x)}=L_x\subset X$.
Let 
$p\colon \mc{Q}\to (L,\le)$ be the contraction homomorphism 
given by $p(x,I)=x\in L$
(Proposition \ref{prop:p_contraction_homomorphism}).
Similarly we define maps $I',p'$ for $(L',\le)$.

There are two straightforward ways to associate an $(L,\le)$-chain 
object to $\mc{F}^{\mc{Q}^{\op}}$. Set
\begin{equation}\label{eq:mvcc}
\MVCC(X,L,\mc{F}^{\mc{Q}^{\op}}) :=
	\Tot(p_*\big(\check{C}(\mc{F}^{\mc{Q}^{\op}})\otimes \OS(\mc{Q}))\big)) \in \CMon(\ChF(\Sh_X)),
\end{equation}
for a multiplicative locally geometric arrangement $[X,L,F]$,
and
\begin{equation}
	\MVC(X,L,\mc{F}^{\mc{Q}^{\op}}) :=
	\MV(X,L,I^*\check{C}(\mc{F}^{\mc{Q}^{\op}}),\Mu)=
	\Tot\big(I^*(\check{C}(\mc{F}^{\mc{Q}^{\op}}))\otimes \Mu\big) \in \ChF(\Sh_X),
\end{equation} 
for an arbitrary arrangement $[X,L,F,\Mu]$ respectively.
Note that ignoring the rank filtration $W'$ we have
$$\MVCC(X,L,\mc{F}^{\mc{Q}^{\op}})=\MVCC(X,\mc{Q},\mc{F}^{\mc{Q}^{\op}}).$$ 
Here the RHS is a monoid by Lemma \ref{lemma:mobius_inversion_multiplication}.
Since $p$ is a homomorphism, \eqref{eq:mvcc} defines a monoid equiped with a multiplicative filtration $W'$.

\begin{rmrk}\label{rmrk:I_pullback}
	The map $I\colon L\to \mc{Q}$ is rarely a homomorphism of lattices, 
	hence $\MVC$ in general is \emph{not} a monoid.
	Since $\Mu$ viewed as an OS-complex is essentially unique, 
	we do not include it 
	in the notation for $\MVC(-)$.
\end{rmrk}
Let us consider the data $[X,L,\mc{F}^{\mc{Q}^{\op}},\Mu]$ as an equipped arrangement.
By definition, a morphism $[X',L',\mc{F}^{\mc{Q'}^{\op}},\Mu']\ar[] 
	[X,L,\mc{F}^{\mc{Q}^{\op}},\Mu]$
is given by a morphism $\phi\colon [X',L',F',\Mu']\ar[] 
	[X,L,F,\Mu]$, where $F=\mc{F}^{\mc{Q}^{\op}}_0$,
and a morphism $T\colon f^*\mc{F}^{\mc{Q}^{\op}}\ar 
	\tilde{g}^*\mc{F}^{\mc{Q}^{'\op}}$ is a morphism of 
$\mc{Q}^{\op}$-objects in $\ChF(\Sh_{X'})$.

Using our notation we have
\begin{prop}\label{prop:mvc_is_functorial}
	The constructions $\MVCC(X,L,\mc{F}^{\mc{Q}^{\op}})$, 
		$\MVC(X,L,\mc{F}^{\mc{Q}^{\op}})$ are functorial.
	Namely, a homomorphism
	$\phi\colon [X',L',\mc{F}^{\mc{Q}^{'\op}}]\ar[] 
		[X,L,\mc{F}^{\mc{Q}^{\op}}]$
	induces $W'$-filtered morphisms:
	$$f^*\MVCC(X,L,\mc{F}^{\mc{Q}^{\op}})\ar 
	\MVCC(X',L',\mc{F}^{\mc{Q}^{'\op}})\in \CMon(\ChF(\Sh_{X'})),$$
	in the case when $\phi$ is multiplicative and 
	$(L,\le)$, $(L',\le)$ are locally geometric, and
	$$f^*\MVC(X,L,\mc{F}^{\mc{Q}^{\op}})\ar 
		\MVC(X',L',\mc{F}^{\mc{Q}^{'\op}},\Mu')\in\ChF(\Sh_{X'})$$
	in the case when $g\colon (L,\le)\ar (L',\le)$ is a homomorphism
	with complete image.
\end{prop}
\begin{proof}
	We argue as in the proof of Proposition \ref{prop:MV_functorial}.
	Let $\mc{F}^{\mc{Q}}=\check{C}(\mc{F}^{\mc{Q}^{\op}})$ and define 
		$\mc{F}^{\mc{Q}'}$ similarly.
	By \eqref{eq:cech:functoriality_along_lattices} we have 
	a morphism $f^*\mc{F}^{Q}\ar \tilde{g}^*\mc{F}^{\mc{Q}'}$.

	1. By considering a commutative square
	\begin{equation*}
	\begin{tikzcd}
		\mc{Q}\arrow[d,"p"]\arrow[r,"\tilde{g}"]& \mc{Q}'\arrow[d,"p'"]\\
		L\arrow[r,"g"]&L'
	\end{tikzcd}
	\end{equation*}
	we obtain a composition in $\ChF(\Sh_{X'})$
	\begin{multline}\label{eq:ugly_composition}
	f^*\Tot(p_*\big(\mc{F}^{\mc{Q}}\otimes \OS(\mc{Q})\big))=
	\Tot(p_*\big(f^*\mc{F}^{\mc{Q}}\otimes \OS(\mc{Q})\big))\ar[\eqref{eq:cech:functoriality_along_lattices}]
	\Tot(p_*\big(\tilde{g}^*\mc{F}^{\mc{Q}'}\otimes\OS(\mc{Q})\big))\ar\\
	\ar[\Tot\to \Tot\circ g_*]
		\Tot (g_*p_*\big(\tilde{g}^*\mc{F}^{\mc{Q}'}\otimes\OS(\mc{Q})\big))=
		\Tot(p'_*\tilde{g}_*\big(\tilde{g}^*\mc{F}^{\mc{Q}'}\otimes\OS(\mc{Q})\big))=
			\\\os{\substack{projection\\ formula}}{=}
	\Tot(p'_*\big(\mc{F}^{\mc{Q}'}\otimes \tilde{g}_*\OS(\mc{Q})\big)) \ar
	\Tot(p'_*\big(\mc{F}^{\mc{Q}'}\otimes \OS(\mc{Q}')\big)).
\end{multline}
	Here the last arrow is induced by $\tilde{g}$ 
		(Corollary \ref{cor:morphism_of_OS_algebras}).
	In order to show that the composition is multiplicative
	consider $K_\bullet:=f^*\mc{F}^{\mc{Q}}\otimes \OS(\mc{Q})$ as a $\mc{Q}$-chain algebra
	as in Lemma \ref{lemma:mobius_inversion_multiplication}.
	$$K_{(x,M)}=\bigoplus\limits_{I\subset K\subset M} f^*\mc{F}^{\mc{Q}^{\op}}_{K\sm I}\cdot (d\mu_I)^{K}.$$
	Similarly $K'_\bullet:=\mc{F}^{\mc{Q'}}\otimes\OS(\mc{Q}')$ is an $\mc{Q}'$-chain algebra.
	The components of \eqref{eq:ugly_composition} are given explicitely:
	$$f^*\mc{F}^{\mc{Q}^{\op}}_{K\sm I}\cdot (d\mu_I)^{K}\ar 
\mc{F}^{\mc{Q'}^{\op}}_{\tilde{g}(K)\sm \tilde{g}(I)}\cdot 
	(\tilde{g}(d\mu_{I}))^{\tilde{g}(K)}, I\subset K\subset M,$$
	where $\tilde{g}(d\mu_I)$ is the image of $d\mu_I$ 
	under the morphism $\Tot(\OS(\mc{Q}))\ar \Tot(\OS(\mc{Q}'))$.
	The morphism is clearly mutliplicative with respect to 
	the rank filtration $W'$ induced by $p$ and $p'$ respectively.
	Hence the composition \eqref{eq:ugly_composition} is 
	a morphism in $\CMon(\ChF)$.

2. Similarly, since $g\colon (L,\le)\ar (L',\le)$ has complete image, 
by Remark \ref{rmrk:I_commutes}
	the following square commutes
\begin{equation*}
\begin{tikzcd}
	L\arrow[d,"I"]\arrow[r,"g"]& L'\arrow[d,"I'"]\\
	\mc{Q}\arrow[r,"\tilde{g}"]& \mc{Q}',
\end{tikzcd}
\end{equation*}
and we obtain a composition of $W'$-filtered morphisms:
\begin{multline*}
	f^*\Tot(I^*\mc{F}^{\mc{Q}'}\otimes\Mu)\ar
	\Tot(I^*\tilde{g}^*\mc{F}^{\mc{Q}'}\otimes \Mu)=
	\Tot(g^*I'^*\mc{F}^{\mc{Q}'}\otimes\Mu)\ar\\
	\Tot(g_*\big(g^*I'^*\mc{F}^{\mc{Q}'}\otimes \Mu\big))
	\ar\Tot(I'^*\mc{F}^{\mc{Q}'}\otimes \Mu').
\end{multline*}
\end{proof}

By Part 1 of Proposition \ref{prop:cech_for_sheaves}, there is a natural quasi-isomorphism
	$\mc{F}^{\mc{Q}}\zgqis (i_*i^!)^{\mc{Q}}_{Z_\bullet}G(F)$
	where $F=\mc{F}^{\mc{Q}^{\op}}_0$ and $Z_{\tilde{I}}=Z_I$ for each $\tilde{I}\in \mc{Q}$,
	which provides a $W'$-filtered quasi-isomorphism:
	$$\MV(X,L,\mc{F}^{\mc{Q}^{\op}},\Mu)\zgqis 
	\MV(X,L,(i_*i^!)^L_{L_\bullet}G(F),\Mu)\in \ChF(\Sh_X)$$
		functorial in all equipped arrangements $[X,L,\mc{F}^{\mc{Q}^{\op}},\Mu]$ 
		in the sense of Proposition \ref{prop:mvc_is_functorial}.
	
\begin{prop}\label{prop:mvc_iso_mvc}
	There are functorial $W'$-filtered quasi-isomorphisms:
	\begin{enumerate}
		\item
$$\MVC(X,L,\mc{F}^{\mc{Q}^{\op}})\zgqis 
	\MVCC(X,L,\mc{F}^{\mc{Q}^{\op}})\in \ChF(\Sh_X),$$
				for equipped locally geometric arrangements $[X,L,\mc{F}^{\mc{Q}^{\op}},\OS(L)]$, and
				with respect to morphisms $\phi$ such that $g$ is complete;
		\item 
			$$\MVC(X,L,\mc{F}^{\mc{Q}^{\op}})\zgqis 
				\MV(X,L,(i_*i^!)^L_{L_\bullet}G(F),\Mu)\in \ChF(\Sh_X),$$
				for equipped arrangements $[X,L,\mc{F}^{\mc{Q}^{\op}},\Mu]$, and
				with respect to morphisms $\phi$ such that $g$ is complete;
		\item 
	$$\MVCC(X,L,\mc{F}^{\mc{Q}^{\op}})\zgqis 
		\MV(X,L,(i_*i^!)^L_{L_\bullet}G(F),\OS(L))\in \Mon(\ChF(\Sh_X)),$$
		for equipped multiplicative locally geometric arrangements $[X,L,\mc{F}^{\mc{Q}^{\op}},\Mu]$, with respect to arbitrary morphisms of such arrangements.
	\end{enumerate}
\end{prop}
\begin{proof}
	1.
	Let $\mc{F}^{\mc{Q}}:=\check{C}(\mc{F}^{\mc{Q}^{\op}})$.
	Applying Lemma \ref{lemma:lattice:on_endomorphism} to the order preserving map $I\circ p\colon \mc{Q}\ar \mc{Q}$
	we obtain a natural morphism $p^*I^*\mc{F}^{\mc{Q}}\ar \mc{F}^{\mc{Q}}$ of $\mc{Q}$-objects.
	This is a quasi-isomorphism, since by the definition of a 
	locally geometric arrangement we have 
	$L_{I(p(\tilde{K}))}=L_{\tilde{K}}\subset X$ for all $\tilde{K}\in \mc{Q}$.
	We obtain the required zig-zag of $W'$-filtered quasi-isomorphisms:
	\begin{align*}
	p_*(\mc{F}^{\mc{Q}}\otimes \OS(\mc{Q}))\al p_*(p^*I^*\mc{F}^{\mc{Q}}\otimes \OS(\mc{Q}))=
		I^*\mc{F}^{\mc{C}}\otimes p_*\OS(\mc{Q})\ar I^*\mc{F}^{\mc{C}}\otimes \OS(L),
	\end{align*}
	where the last arrow follows by the formality of the atomic complex $p_*\OS(\mc{Q})$ in 
	Theorem \ref{thrm:atomic_complex}.
	
	2. By Part 3 of Proposition \ref{prop:cech_for_sheaves} we have 
	a natural quasi-isomorphism $\check{C}(\mc{F}^{\mc{Q}^{\op}})\zgqis (i_*i^!)^{\mc{Q}}_{L_\bullet}G(F)$.
	On the other hand, since $L_{I(x)}=L_x\subset X$, we have $I^*(i_*i^!)^{\mc{Q}}_{L_\bullet}=(i_*i^!)^L_{L_\bullet}$,
	which gives us the required $W'$-filtered quasi-isomorphism
	$$\Tot(I^*\check{C}(\mc{F}^{\mc{Q}^{\op}})\otimes\OS(L))\zgqis 
	\Tot(I^*(i_*i^!)^{\mc{Q}}_{L_\bullet}G(F)\otimes \OS(L))=\Tot((i_*i^!)^L G(F)\otimes \OS(L)).$$

	3. By the definition of a locally geometric arrangement 
	we have $L_J=L_x$ if $x\in \os{\circ}{\sup}(J)$.
	So there is a natural isomorphism
	$(i_*i^!)^{\mc{Q}}_{L_\bullet}G(F)\ar 
	p^*(i_*i^!)^L_{L_\bullet}G(F)$. We obtain
	a zig-zag of $W'$-filtered quasi-morphisms
	\begin{multline}\label{eq:another_zigzag}
		\Tot(p_*(\check{C}(\mc{F}^{\mc{Q}^{\op}})\otimes\OS(\mc{Q})))\zgqis
	\Tot(p_*((i_*i^!)^{\mc{Q}}_{L_\bullet}G(F)\otimes\OS(\mc{Q})))\simeq
	\Tot(p_*(p^*(i_*i^!)^{L}G(F)\otimes\OS(\mc{Q})))=\\
	=\Tot((i_*i^!)^L G(F)\otimes p_*\OS(\mc{Q}))\ar \Tot((i_*i^!)^L G(F)\otimes \OS(L)),
	\end{multline}
	where the last arrow is the quasi-isomorphism given 
	by Theorem \ref{thrm:atomic_complex}.
	Here the first zig-zag is a quasi-isomorphism of $(L,\le)$-chain 
	algebras by Proposition \ref{prop:cech_for_sheaves} Part 4. 
	Since $(i_*i^!)^{\mc{Q}}G(F)$ 
	and $(i_*i^!)^L G(F)$ are algebras, not just partial algebras, 
	all other arrows are morphisms of algebras as well.
	Thus \eqref{eq:another_zigzag} is a multiplicative quasi-isomorphism.
\end{proof}

We finish this subsection by considering the case 
$\mc{F}\zgqis \ul{\Bbbk}_X\in \CMon(\Ch(\Bbbk-\Sh_X))$
for $\Bbbk\supset \mb{Q}$.

Assume $\mc{F}^{\mc{Q}^{\op}}$ is $U_\bullet/X$-adapted, i.e.
there are quasi-isomorphisms in $\CMon(\ChF(\Bbbk-\Sh_X))$:
$$\mc{F}^{\mc{Q}^{\op}}_I\zgqis (j_*j^*)_{U_I}G_{TW}(\ul{\Bbbk}_X).$$

Recall that $U_N$ is the complement of the arrangement.
\begin{cor}
\label{cor:mvhd_is_adapted}
	There is a quasi-isomorphism of cdga's 
	$$\Gamma_{TW}\MVCC(X,L,\mc{F}^{\mc{Q}^{\op}})\zgqis A^*_{\PL}(U_N)\otimes \Bbbk\in \CDGA_{\Bbbk}.$$
\end{cor}
\begin{proof}
	After forgetting the rank filtration we have $\MV_{com}(L)=\MV_{com}(\mc{Q})$. 
	Then by Corollary \ref{cor:mobius_trick_sheaves},
	$\MVCC(\mc{Q})=\check{C}(\mc{F}^{\mc{Q}^{\op}})*\OS(\mc{Q})$,
	which
	is quasi-isomorphic to $\mc{F}^{\mc{Q}^{\op}}_N\zgqis 
	(j_*j^*)_{U_N}\ul{\mb{Q}}|_X\otimes \Bbbk$, i.e.
	$\MVCC(L)$ is a $U_N$-adapted sheaf. It remains to apply 
	Corollary \ref{cor:global_sections_of_adapted_sheaf}.
\end{proof}


\newpage
\section{Rational homotopy type of the complement}
\label{sect:main_sect}
In this section we prove our main results. 
Here $[X,L,\Mu]$ 
is a smooth compact algebraic 
arrangement over $X\in \Var$ 
with an OS-algebra $\mc{M}$ underlying $(L,\le)$.
As usual, 
$U=X-\bigcup\limits_{s>0\in L} L_s$ 
is the complement of the arrangement.

In Section \ref{sect:compl:mvhc} we introduce
the (rank) Mayer-Vietoris Hodge complex 
$(\MVHC(X,L),F,W')$ such that 
${}_{W'}E^{pq}_1\MVHC(X,L)_{\Bbbk}$ is naturally isomorphic to 
${}_{W'}E^{pq}_1 \MV((i_*i^!)G(\ul{\Bbbk}_X),\Mu)$ for $\Bbbk=\mb{Q},\mb{C}$ (see \ref{sect:compl:mv_resolution}).
This will allow us to identify the associated pure Hodge structure 
$\Gr^{\Dec W} H^*(U)$ with ${}_{W'}E^{**}_2$ appearing in 
Theorem \ref{thrm:ss_for_lattices}.
As a corollary, we describe the multiplicative structure of 
$H^*(U;\mb{Q})$.
A similar result appeared in \cite{JCW}, but we note that our
proof does not require the intersections 
of the varieties $L_x\subset X$ to be clean. 

To model the rational homotopy type of $U$,
we define the cubical Mayer-Vietoris Hodge diagram
$(\MVHD(X,\mc{C}),F,W')$ whose
${}_{W'}E_1$-term is described by Theorem \ref{thrm:ss_for_lattices}.

\subsection{\texorpdfstring{Mixed Hodge complexes $(\MVHC(X,L),F,W')$}{} and 
\texorpdfstring{$(\MVHD(X,\mc{Q}),F,W')$}{}}
\label{sect:compl:mvhc}

\subsubsection{Definition}
Recall that, for a subset $I\subset \Atoms(L)$ we write 
$L^I=\bigcup\limits_{i\in I} L_i$ and 
$L_I=\bigcap\limits_{i\in I} L_i\subset X$.
Consider the corresponding locally cubical lattice
$\mc{Q}$ (Definition \ref{dfn:cubical_lattice_of_L}) and the cubical lattice $\mc{C}$.
Let $U_\bullet\subset X$ be the open $\mc{Q}$-subspace 
of $X$ with vertices
$U_{\tilde{I}}=X-L^I\subset X$ for $\tilde{I}=(x,I)\in \mc{Q}$. 
Recall the order-preserving map $I\colon (L,\le)\ar (\mc{Q},\le)$
is given by
$I(x):=(x,\{t\in \Atoms(L)\mid t\leq x\})\in \mc{Q}.$
By the definition of an arrangement, we have $L_{I(x)}=L_x\subset X$.

Since $U_\bullet\ar X$ is a diagram of open subsets over a smooth algebraic $X$,
Corollary \ref{cor:functorial_mhd} provides a natural $\mc{Q}^{\op}$-object in $\CMon(\MHC(X))$:
\begin{equation}\label{eq:cubical_hodge_complex}
\mc{K}^{\mc{Q}^{\op}}_\bullet:=(\mc{K}_{U_\bullet/X},F,W).
\end{equation}
This defines a $\mc{Q}$-object in $\MHC(X)$:
\begin{equation}
	\mc{K}^{\mc{Q}}:=\check{C}(\mc{K}^{\mc{Q}^{\op}}).
\end{equation}
By Proposition \ref{prop:cech_for_sheaves} and 
the naturality of mixed Hodge structures, for each $\bar{I}=(x,I)\in \mc{Q}$ 
we have an isomorphism
of mixed Hodge structures
$$(\mb{H}^*(\mc{K}^{\mc{Q}}_{\tilde{I}}),F,\Dec W)\simeq H^*_{L_x}(X)$$ for each $\tilde{I}\in \mc{Q}$. 
Since $L_x\subset X$ is smooth, $\mc{K}^{\mc{C}}_{\tilde{I}}$
is a pure Hodge complex of weight $0$.

For an arrangement $[X,L]$ we define
\begin{equation}
	\MVHC(X,L):=\MVC(X,L,\mc{K}^{\mc{Q}^{\op}})=I^*\mc{K}^{\mc{Q}}*\Mu.
\end{equation}
By Lemma \ref{lemma:convolution_is_hodge} we obtain $(\MVHC(X,L),F,W)\in \MHC(X)$.
On the other hand, the rank filtration $W'$ on the convolution and 
Lemma \ref{lemma:two_weight_filtrations} allow us to consider
the \emph{Mayer-Vietoris Hodge complex} $(\MVHC(X,L),F,W')\in \MHC(X)$.
We omit $\Mu$ from the notation for $\MVHC(X,L)$ since it is essentially 
unique by Proposition \ref{prop:morphism_of_OS_complexes}.

By the same argument as in the case of a locally geometric arrangement with 
$L=\mc{Q}$, i.e. $I(-)=\id$, the M\"obius inversion formula in
Lemma \ref{lemma:mobius_inversion_multiplication} gives us a commutative monoid
with an additional rank filtration $W'$:
\begin{equation}
	\MVHD(X,\mc{Q}):=\MVCC(X\mc{Q},\mc{K}^{\mc{Q}^{\op}})=\check{C}(\mc{K}^{\mc{Q}^{\op}})*\OS(\mc{Q}).
\end{equation}
Then both  $(\MVHD(X,\mc{Q}),F,W)$ and $(\MVHD(X,\mc{Q}),F,W')$ define 
\emph{Mayer-Vietoris Hodge diagrams} in $\MHD(X)$. 
\begin{rmrk}
Note that by our definition of a locally geometric arrangement, 
 $\mc{Q}$ is just a 
sublattice in the cubical lattice $\mc{C}$ on $\Atoms(L)$. 
Tautologically we have
$\MVHD(X,\mc{Q})=\MVHD(X,\mc{C})$.
\end{rmrk}

\subsubsection{Functoriality}
Assume $\phi\colon [X',L',\Mu']\ar[] [X,L,\Mu]$ 
is a morphism of
arrangements corresponding to a map $f\colon X'\ar X$,
a rank-preserving homomorphism $g\colon L\ar L'$ and
the natural morphism $g_*\Mu'\ar \Mu$.

Let $U_\bullet$ and $U'_\bullet$ be the corresponding complements considered as 
$\mc{Q}$- and $\mc{Q}'$-varieties respectively.
By definition, $f$ maps $U'_{g(x)}$ to $U_x$ for all $x\in L$,
hence there is a well-defined morphism $U'_{g(I)}\ar U_{I}$ for $I\subset \Atoms(L)$, i.e.
the morphism $f$ naturally restricts to a morphism 
$f_\bullet\colon \tilde{g}^*U'_\bullet\ar U_\bullet$ of 
$\mc{Q}$-varieties, 
where $\tilde{g}\colon \mc{Q}\ar \mc{Q}'$ is the induced morphism over $g\colon L\ar L'$.
By the naturality of $\mc{K}^{\mc{Q}^{\op}}$, there is 
a natural morphism 
$f^*\mc{K}^{\mc{Q}^{\op}}\ar \tilde{g}^{\op*}\mc{K}^{\mc{Q}^{'\op}}$,
which by Proposition \ref{prop:mvc_is_functorial} and Remark \ref{rmrk:pullback_for_mhc}
induces a morphism $$\MVHC(X,L)\ar R_{TW}f_* \MVHC(X',L')\in \MHC(X),$$
i.e. $\MVHC(X,L)\in \MHC(X)$ is functorial in the arrangement $[X,L]$.
Similarly, one can see that $\MVHD(X,\mc{Q})\in \MHD(X)$ is functorial in 
multiplicative arrangement $[X,\mc{Q},\OS(\mc{Q})]$.

Proposition \ref{prop:mvc_iso_mvc} implies that there is a natural $W'$-filtered quasi-isomorphism:
\begin{equation}
	\MVHC(X,L)\zgqis \MV(X,L,(i_*i^!)^L_\bullet G(\ul{\mb{Q}}_X),\Mu)\in \Ch(\Sh_X),
\end{equation}
while by Proposition \ref{prop:MV_functorial} the RHS is quasi-isomorphic to $(Rj_* j^*)_U\ul{\mb{Q}_X}$.
Finally, by Lemma \ref{lemma:two_weight_filtrations}, the filtrations induced by 
$W$ and $W'$ on $\mb{H}^*(\MVHC(X,L))\simeq H^*(U)$ coincide.
If $[X,\mc{Q},\OS(\mc{Q})]$ is locally geometric, 
which is the case when 
$[X,L]$ is locally geometric,
then the same holds for $\MVHD(X,\mc{Q})$.
On the other hand we have
\begin{lemma}\label{lemma:mvhc_natural}
	There are natural quasi-isomorphisms
	$$(\MVHC(X,L),F,W)\zgqis (\mc{K}_{U/X},F,W)\in \MHC(X),$$
	and
	$$(\MVHD(X,\mc{Q}),F,W)\zgqis (\mc{K}_{U/X},F,W)\in \MHD(X)$$
	when $[X,\mc{Q},\OS(\mc{Q})]$ is locally geometric.
\end{lemma}
\begin{proof}
	Here $\mc{K}:=(\mc{K}_{U/X},F,W)$ is the 
	natural Hodge diagram associated with the complement $U$.
	We have natural restriction morphisms
$r_I\colon \mc{K}^{\mc{Q}^{\op}}_{\tilde{I}}\ar \mc{K}$ 
for all $\tilde{I}\in\mc{Q}^{\op}$.
	Consider $q^*\mc{K}$ as a constant $\mc{Q}^{\op}$-object.
	Similarly to the proof of Theorem \ref{thrm:mobius_trick},
we construct a zig-zag

\begin{equation}\label{eq:mvhc:zigzag}
(\MVHC(X,L),F,W)=(I^*\check{C}(\mc{K}^{\mc{Q}^{\op}})*\Mu,F,W)\ar[r]
	(I^*\check{C}(q^*\mc{K})*\Mu,F,W)
	\al[i]
		(\mc{K},F,W)\in \MHC(X).
	\end{equation}
	Here the equality is given by the definition, 
	the left arrow is induced by the
	$r_I$'s, and the right arrow $i$ is the natural inclusion of
	$\mc{K}=I^*\check{C}(q^*\mc{K})_0\otimes \Mu_0$.  
	
	By Lemma \ref{lemma:convolution_is_hodge} 
	all three terms  \eqref{eq:mvhc:zigzag} are MHC's
	and are naturally 
	quasi-isomorphic to $Rj_*j^*\ul{\Bbbk}$.
	It remains to note that in the case $L=\mc{Q}$, i.e. $I(-)=\id$,
	the morphisms $r$ and $i$ are multiplicative.
	This gives the required zig-zag of quasi-isomorphisms.
\end{proof}
We conclude that the filtrations $F$ and $\Dec W'$ induce on 
$\mb{H}^*(\MVHC(X,L))\simeq H^*(U)$ the natural mixed Hodge structure,
which in turn coincides with the one induced by $F$ and $\Dec W$.

\subsection{Cohomology}
Here $[X,L,\Mu]$ is a multiplicative smooth 
compact algebraic arrangement.
By the above, $\MVHC(X,L)$ is functorial in 
$[X,L,\Mu]$.
Recall that $\Mu_x$ sits in cohomological degree $-r(x)$.
Let us summarize the previous results in the following
generalization of Theorem \ref{thrm:ss_for_lattices}:
\begin{thrm}\label{thrm:mv_spectral_sequence}
	The spectral sequence 
	${}_{W'}E^{pq}:={}_{W'}E^{pq}(\MV(X,L,(i_*i^!)^L G(\ul{\mb{Q}}),\Mu))$
	is functorial in the multiplicative arrangement $[X,L,\Mu]$
	and satisfies the following:
	\begin{enumerate}
		\item We have ${}_{W'}E^{**}_1=
  H^*_{L_\bullet}(X)*\Mu$, or explicitely:
			$${}_{W'}E^{pq}_1=
			\bigoplus\limits^{x\in L}_{r(x)=-p}
			H^q_{L_x}(X)\otimes \mc{M}_x[-r(x)].$$
		\item The differential $d_1$ is equal to 
			$d_1=\sum g_{x'x}\otimes \partial_{x'x}$,
			where $g_{x'x}\colon H^*_{L_x}(X)\to H^*_{L_{x'}}(X)$
			are the natural homomorphisms for $x'<:x\in L$ and
			$\partial_{x'x}$ are the structure maps of $\Mu$.

		\item $({}_{W'}E^{**}_r,d_r),r>0$ 
			is a complex of Hodge structures
			with $({}_{W'}E^{pq}_r,F),r>0$ 
			a pure Hodge structure of weight $q$.

		\item There is a natural isomorphism of 
			pure Hodge structures of weight $q$:
			$${}_{W'}E^{pq}_2=
		{}_{W'}E^{pq}_\infty\simeq \Gr^{\Dec W}_q H^{p+q}(U).$$ 

		\item The spectral sequence is 
				multiplicative with respect to the
		    mixed Hodge structures.

		\item There is a functorial isomorphism of algebras:
			$$({}_{W'}E^{**}_2,W')\simeq (H^*(U;\mb{Q}),W).$$

	\end{enumerate}
\end{thrm}
\begin{proof}
	Let $\MVHC:=\MVHC(X,L)$, $\MV:=\MV((i_*i^!)G(\ul{\mb{Q}}_X),\Mu)$.
	By Proposition 
 \ref{prop:mvc_iso_mvc}
 we have a natural $W'$-filtered quasi-isomorphism:
	$$(\MVHC,W')\zgqis (\MV,W')\in \Ch(\Sh_X).$$
	By Theorem \ref{thrm:ss_degeneration_mhc} we have
	${}_{W'}E^{**}_2\MVHC={}_{W'}E^{**}_{\infty}\MVHC$.
	The terms $E^{pq}_r(\MV,W'),r>0$ are described by 
	Theorem \ref{thrm:ss_for_lattices},	
	providing
	a natural isomorphism 
	${}_{W'}E^{pq}_{\infty}(\MV)\simeq \Gr^p_{W'} H^{p+q}(U;\mb{Q})$
	which is automatically \emph{multiplicative}.
	
	By Lemma \ref{lemma:mvhc_natural} the filtrations $W$ and $W'$
	on $H^{p+q}(U;\mb{Q})$ coincide, hence we have 
	a natural multiplicative isomorphism
	${}_{W'}E^{pq}_2(\MV)\simeq 
		\Gr^p_W H^{p+q}(U)=\Gr^{\Dec W}_q H^{p+q}(U)$
	which is compatible with the Hodge filtration $F$ on both sides.
	
	Finally, the splitting provided by 
	Corollary \ref{cor:cirici_splitting} implies a functorial
	multiplicative isomorphism 
	$\Gr^{\Dec W} H^*(U;\mb{Q})\simeq H^*(U,\mb{Q})$.
\end{proof}
\begin{rmrk}
Assertions 1-6 of the previous Proposition hold without assuming $\Mu$ has a multiplicative structure, 
i.e.\ additively it is sufficient for $\Mu$ to be an OS-complex.  

	This result is also stated in \cite{JCW}. 
	We want to note that the \emph{functoriality} of 
	the splitting over $\mb{Q}$ is a non-trivial 
	fact based on Theorem \ref{thrm:mhd:E_1_qis} due to Cirici and Horel.
	For example over $\mb{R}$ 
	the Deligne splitting of the weight filtration is not 
	conjugation invariant, 
	while the Morgan-Sullivan theorem \cite[Theorem 10.1]{M}
	a priori provides
	only a non-canonical splitting.
\end{rmrk}

\subsection{CDGA model}
Assume $[X,\mc{C},\OS(\mc{C})]$ is 
a smooth compact algebraic cubical arrangement
given by the subvarieties $Z_i\subset X$. As usual we set $Z_I=\cap_{i\in I}Z_{i}$ for 
$I\subset \Atoms(\mc{C})$. 
By definition
$$\MVHD(X,\mc{C})=\MVCC(X,\mc{C},\mc{K}^{\mc{C}^{\op}})=\check{C}(\mc{K}^{\mc{C}^{\op}})*\OS(\mc{C}),$$
so by Corollary \ref{cor:mvhd_is_adapted} we have a natural quasi-isomorphism:
\begin{equation}\label{eq:mvhd_qis_Apl}
\Gamma_{TW}(\MVHD(X,\mc{C}))\zgqis A_{\PL}(U)\in \CDGA_{\mb{Q}}.
\end{equation}

By the naturality of the construction,
$(\MVHD(X,\mc{C}),F,W')$ and $(\MVHD(X,\mc{C}),F,W)$ 
are functorial equipments in $[X,\mc{C},\OS(\mc{C})]$ with values in $\MHD(X)$.

Recall that $U=X-\cup_i Z_i$ is the complement.
\begin{thrm}[The main theorem]\label{thrm:main_cubical}
	There is a multiplicative spectral sequence $E^{pq}_r$ 
	with the following properties.
    \begin{enumerate}
		\item We have $E^{pq}_1=
	\bigoplus_{|I|=-p}H^{q}_{Z_I}(X)\otimes_{\mb{Z}}
		\mb{Z}\cdot d\nu_I$ where 
  		$d\nu_I\in \Lambda\langle d\nu_i\mid i\in N\rangle$ are  
  		Grassmann monomials with $\deg d\nu_I=-|I|$.

		\item The differential 
		$d_1=\sum g_{I-\{i\},I}\otimes \iota_i$
		where 
		$$g_{I'I}\colon H^{*}_{Z_I}(X)\to H^{*}_{Z_{I'}}(X)$$
	 	is the natural homomorphism
		induced by the inclusion $Z_I\to Z_{I'}$ for $I'\subset I$, 
		and $\iota_i$ is the derivation on the Grassmann 
		algebra satisfying $\iota_i d\nu_j=\delta_{ij}$.
	
		\item The multiplication
		$$H^q_{Z_I}(X)\otimes \mb{Z}\cdot d\nu_I \otimes 
		H^{q'}_{Z_J}(X)\otimes \mb{Z}\cdot d\nu_J\to 
		H^{q+q'}_{Z_{IJ}}(X)\otimes \mb{Z}\cdot d\nu_{I}\wedge d\nu_J$$
		with the usual sign rule makes $E^{pq}_1$ into a cdga.

		\item There is a functorial quasi-isomorphism 
			$$E^{**}_1\zgqis A_{\PL}(U)\in \CDGA_{\mb{Q}}.$$
	\end{enumerate}
\end{thrm}
\begin{proof}
	Let $E_r={}_{W'}E_r(\MVHD(X,\mc{C}))$.
	The description of $E_1$ is given by 
	Theorem \ref{thrm:mv_spectral_sequence} by setting $L:=\mc{C}$. This proves the first three assertions of the Theorem.
	Since $\Gamma_{TW}(\MVHD(X,\mc{C}),F,W')\in \MHD$, 
	Theorem \ref{thrm:mhd:E_1_qis} together with 
	\eqref{eq:mvhd_qis_Apl} proves the last assertion.
\end{proof}

Finally, assume $[X,L]=[X,L,\OS(L)]$ is a locally geometric arrangement.
We have
\begin{thrm}[The main theorem for lattices]\label{thrm:main_lattices}
For the multiplicative spectral sequence $E_1$ in Theorem~\ref{thrm:mv_spectral_sequence} 
	there is a quasi-isomorphism 
	$$E^{**}_1\simeq A_{\PL}(U)\in \CDGA_{\mb{Q}},$$
		which is functorial in the locally geometric arrangement $[X,L]$.
\end{thrm}
\begin{proof}
	Let $[X,\mc{Q},\OS(\mc{Q})]$ be the locally cubical 
	arrangement corresponding to $[X,L,\OS(L)]$.
	Recall that $|\os{\circ}{\sup}(I)|\le 1$ for $I\subset \Atoms(I)$, hence
	$\mc{Q}$ is in fact a sublattice of the cubical lattice $\mc{C}$ corresponding to $(L,\le)$.
	Thus $\Tot(H^*(X,X-\mc{Q}_\bullet)\otimes \OS(\mc{Q}))=\Tot(H^*(X,X-\mc{C}_\bullet)\otimes \mc{C})$.
	On the other hand we have a zig-zag of quasi-isomorphisms of 
	$(L,\le)$-chain algebras
	\begin{multline*}
		p_*(H^*(X,X-\mc{Q}_\bullet)\otimes\OS(\mc{Q}))
  \ar[r]
		 p_*(p^*H^*(X,X-L_\bullet)\otimes \OS(\mc{Q}))=
		H^*(X,X-L_\bullet)\otimes p_*\OS(\mc{Q})\ar \\
		\ar H^*(X,X-L_\bullet)\otimes \OS(L),
	\end{multline*}
 where $r$ is an isomorphism,
	the equality is given by the projection formula (Lemma \ref{lemma:lattice:projection_formula}), and the last arrow is a quasi-isomorphism provided by the formality of the atomic complex $p_*\OS(\mc{Q})$ (Theorem \ref{thrm:atomic_complex}).
	Applying $\Tot(-)$ we obtain a quasi-isomorphism of cdga's 
	$${}_{W'}E_1(\MVHD(X,\mc{C}))=
	\Tot(H^*(X,X-\mc{C}_\bullet)\otimes \OS(\mc{C}))\zgqis 
		\Tot(H^*(X,X-L_\bullet)\otimes\OS(L))=E_1.$$
\end{proof}

\begin{rmrk}\label{rmrk:weaker_assumptions}
One can check that for a locally geometric  $[X,L]$, 
the formula $$\MVHD(X,L):=\MVCC(X,L,\mc{K}^{\mc{Q}^{\op}})=\Tot(p_*(\check{C}(\mc{K}^{\mc{Q}^{\op}})\otimes \OS(L))),$$
defines $(\MVHD(X,L),F,W')\in \MHD(X)$ such that ${}_{W'}E_1(\MHD(X,L))$ is described by 
Theorem \ref{thrm:mv_spectral_sequence}.
\end{rmrk}
\begin{rmrk}
Let us note that the requirement that the base algebraic variety $X$
should be compact can be weakened. 
It is enough to require
the mixed Hodge structure $H^k_{L_x}(X)$
to be pure of weight $k$ for all $x\in L$. 
\end{rmrk}

\newpage
\section{Applications}\label{sect:applications}
In Section \ref{sect:app:subspaces} we provide a model for 
affine subspace complements and then establish multiplicative version of the 
Goresky-Macpherson formula.
In particular we generalize a formality result due to 
Feichtner-Yuzvinsky \cite{FY} stating 
that the complement of a complex vector subspace 
arrangement with geometric intersection lattice is formal. 

Then we consider the chromatic configuration
space $F(X,G)$ (Example \ref{exmpl:arrangementss}) 
given by the complement of an arrangement in $X^{V(G)}$ where $X$ is a topological space, $G$ is a finite graph and $V(G)$ is the set of the vertices of $G$.
In Section \ref{sect:mv_totaro} we will show that the spectral
sequence $E^{**}_1\Rightarrow H^*(F(X,G);\mb{Z})$ provided 
by Theorem \ref{thrm:ss_for_lattices} generalizes the Totaro spectral sequence 
\cite[Theorem 1]{Totaro}.
Finally, in Section \ref{sect:app:conf_homotopy_type}
we combine this result with our main theorem
and describe an equivariant model for the $\mb{Q}$-homotopy type of
$F(X,G)$ in the case when $X$ is a smooth compact algebraic variety over $\mb{C}$.

\subsection{Subspace arrangements}
\label{sect:app:subspaces}
Here we describe a generalization of Yuzvinsky's results on vector subspace arrangements 
to the affine case.
Consider a collection $\{Z_i\subset V:=\mb{C}^n\}$ of affine subspaces.
Let $(L,\leq)$ be the corresponding intersection poset, i.e.\ 
the elements of $L$ are given by
all \emph{different} non-empty intersections $Z_I\subset \mb{C}^n,I\subset \Atoms(L)$.
The corresponding arrangement poset $[V,L]$ was considered in \ref{GorMac},
it is multiplicative and equipped with the weak grading $c(x)=\codim_{\mb{C}}(L_x/V)$.

First we will describe a rational model of the complement $U:=V-\cup_i Z_i$. 
In particular we will establish formality of complements to arrangements with locally geometric intersection
lattice.
Then we will show that the Goresky-MacPherson formula can be upgraded to an isomorphism of 
algebras over $\mb{Q}$:
$$H^*(V-\cup_i Z_i)\simeq \bigoplus^{x\in L}_{} H^{*-2\codim(L_x)}(\mc{D}_x),$$
where $H^*(\mc{D}_x)\simeq \tilde{H}_{-*-2}(L(0,x))$ according to Theorem \ref{order_complex}.
This isomorphism identifies the natural weight filtration on $H^*(V-\cup_i Z_i)$ with
the lattice filtration \ref{GorMac}.
This implies an isomorphism of algebras in pure Hodge structures:
$$\Gr^{\Dec W} H^*(V-\cup_i Z_i)\simeq \bigoplus^{x\in L}_{} H^{*-2\codim(L_x)}(\mc{D}_x)(-\codim L_x),$$
Since the mixed Hodge structure $H^k_{L_x}(V)$ is pure of weight $k$, 
according to Remark \ref{rmrk:weaker_assumptions} one can omit properness assumption
and use the cubical model immediately.
For completeness we also derive it in by using a certain compactification
in the end of the section (Proposition \ref{prop:nonproper_model}).  

For now we assume that the cubical model $E^{**}_1=\bigoplus_{I} H^*_{Z_I}(X)\otimes dx_I$
with grassman monomials $\deg dx_I=-|I|$ 
models the complement $V-\cup_i Z_i=V-\cup_{x>0}L_x$. It
is isomorphic to cdga \ref{GorMac}:
$\bigoplus_{x\in L}\tau_x\mc{D}^*_x$,
where $\tau_x$ is a formal variable (the Thom class) of degree $2\codim L_x$ corresponding to the generator of 
MHS $H^{2\codim L_x}_{L_x}(V)(-\codim L_x)$, so that $\tau_x \tau_y=\tau_{x\vee y}$ if $c(x\vee y)=c(x)+c(y)$
(i.e. $L_x,L_y$ are transversal) and is zero otherwise.
Unwinding the definition of the atomic complex $\mc{D}$ we obtain:
\begin{thrm}\label{thrm:model_subspaces}
	A cdga $$\bigoplus\limits^{x\in L,I}_{p(I)=x}\tau_x \otimes dx_I,$$
	with differential $\partial(dx_I)=\sum_{i|p(I-\{i\})=p(I)}\iota_i$
	is quasi-isomorphic to $A^*_{\PL}(V-\cup_i Z_i)$.
	If $(L,\le)$ is locally geometric, then $V-\cup_i Z_i$ is formal and modeled by:
	$$\bigoplus_{x\in L}\tau_x \OS(L).$$
\end{thrm}
\begin{proof}
	The first assertion is clear. The formality in locally geometric case follows from
	formality of atomic complex $\mc{D}\zgqis \OS(L)$ by Theorem \ref{thrm:atomic_complex}.
\end{proof}

Passing to the cohomology and taking the Hodge structure into account we obtain:
$$E^{pq}_2=\bigoplus\limits^{x\in L}_{2\codim L_x=q}H^p(\mc{D}_x)(-\codim L_x).$$
By our main theorem $\oplus_{p+q=n}E^{pq}_2\simeq H^n(U)$, hence:
\begin{thrm}
	There is a functorial isomorphism of algebras:
	$$H^*(V-\cup_i Z_i)\simeq \bigoplus\limits^{x\in L} \tau_x H^*(\mc{D}_x)\simeq
	\bigoplus\limits^{x\in L}\tau_x \tilde{H}_{-*-2}(L(0,x)).$$
\end{thrm}
\begin{rmrk}
	The multiplication on the order complex can be described by shuffle-type formula \cite{YuzAtomic}.
\end{rmrk}
Recall that the lattice filtration ${}_{L}W_p$ on $E^{pq}_1$
\begin{prop}\label{weight_lattice_filtrations}
	The weight and the lattice filtration ${}_{L}W$ on $H^n(U)$ coincide
	after application of the Goresky-MacPherson formula.
\end{prop}
\begin{proof}
	Let $W'$ denotes the column filtration from the cubical model. 
	By the main theorem 
the weight filtration on $H^n(U)$ coincide with $\Dec W'$, where
$(\Dec W')_q H^n(U)=W'_{q-n}H^n(U)$.
By definition we have:
$$W'_{p}H^n(U)=\bigoplus\limits^{x\in L,i\le p}H^{n-i}_{L_x}(X)\otimes H^{i}(\mc{D}_x)=
\bigoplus^{x\in L}_{2c(x)\le n+p}H^{n-2c(x)}(\mc{D}_x),$$
for $p=q-n$ we obtain:
$$(\Dec W')_q H^n(U)=\bigoplus\limits^{x\in L}_{2c(x)\le q}H^{n-2c(x)}(\mc{D}_x).$$
On the other hand
$${}_{L}W_q H^n(U)=\bigoplus\limits^{x\in L,i}_{c(x)\le q}H^{n-i}_{L_x}(X)\otimes H^{i}(\mc{D}_x)=
\bigoplus\limits^{x\in L}_{c(x)\le q}H^{n-2c(x)}(\mc{D}_x).$$

\end{proof}
\begin{rmrk}\label{rmrk:purity}
	In case of $C$-equal codimension arrangement the intersection poset $(L,\le)$
	is geometric lattice with the rank function $r(x)=c(x)/C$.
	Then $H^*(\mc{D}_x)$ is concentrated in degree $-r(x)$ and 
	$\Gr^{\Dec W}_q H^n(U)\ne 0$ implies that $n-q=-\frac{q}{2C}$ or $q=\frac{2Cn}{2C-1}$.
\end{rmrk}

Now let us show that all assertions above can be derived from a certain compactification,
which will provide a model quasi-isomorphic in category of complexes of 
MHS's to the given above \ref{thrm:model_subspaces}.

Let 
$X:=\mb{P}^n\supset \mb{C}^n$ and
$\tilde{L}$ be the intersection poset obtained by taking 
the closure of $\{Z_i\}$ in $X$
and adding the hyperplane at infinity
$\tilde{L}_\infty:=\mb{P}^{n-1}\subset \mb{P}^n$.
So $\Atoms(\tilde{L})=\Atoms(L)\sqcup \{\infty\}$.
Clearly, $$U=\mb{P}^n-\bigcup\limits_{x\in \tilde{L}}\tilde{L}_x=
\mb{C}^n-\bigcup\limits_{x\in L}L_x.$$

\begin{rmrk}
	Note that $\tilde{L}$ is not locally geometric in general. 
	For example consider the intersection poset $(L,\le)$ given by
	a plane and a line in $\mb{A}^3$ with empty intersection. 
	In this case $\tilde{L}$ is not graded and its description is given by 
	Part 5 of Example \ref{exmpl:arrangementss}.
\end{rmrk}
Let us first describe a simple model for the $\mb{Q}$-homotopy type of $U$ 
without any assumptions on $(L,\le)$ using the cubical model.
By Theorem \ref{thrm:main_cubical}, 
the following cdga models the rational 
homotopy type of $U$:
$$K:=\left(\bigoplus\limits_{I\subset \Atoms(\tilde{L})}
H^*_{\tilde{L}_I}(\mb{P}^n)\cdot d\nu_I,\partial\right).$$

Recall that $d\nu_I$ has degree $-|I|$ and the 
differential $\partial=\sum\limits_i \iota_i$ is
the whole Koszul differential.

Since $H^*(\mb{P}^n)=\mb{Q}[t]/t^{n+1},\deg t =2$ 
and $\tilde{L}_I$
is a projective subspace, one can identify 
the cohomology with support $H^*_{\tilde{L}_I}(\mb{P}^n)\subset H^*(\mb{P}^n)$ 
with 
$t^{\codim \tilde{L}_I/\mb{P}^n}H^*(\mb{P}^n)\subset H^*(\mb{P}^n)$ as a graded $H^*(\mb{P}^n)$-module, where by definition 
$\codim \emptyset/\mb{P}^n=n+1$.
This description is compatible with the algebra structure on $K$. 
The element 
$t^{\codim L_I}\in H^*(\mb{P}^n)$ 
corresponds to the Thom class $\tau_I\in H^*_{\tilde{L}_I}(\mb{P}^n)$.

Further, since taking the direct product with an affine space 
does not change the homotopy type of the affine complement
and the codimension of the subspaces in the arrangement,
one may assume from the beginning that 
for any $x\in L\subset \tilde{L}$, we have 
$\tilde{L}_x\cap \tilde{L}_{\infty}\neq \emptyset\subset \mb{P}^n$, i.e.\ all 
subspaces $L_x$ are at least one dimensional. Note that for $I\subset \Atoms(L)\subset \Atoms(\tilde{L})$ we have by definition $\tilde{L}_I=\cap_{i\in I}\tilde{L}_i$, which could be non-empty even if $L_I=\emptyset$.

Consider the following  complexes over $\mb{Q}[t]$,
\begin{enumerate}
	\item $K=\bigoplus\limits^
		{I\subset \Atoms(\tilde{L})}
		t^{\codim \tilde{L}_I}
		\mb{Q}[t]/t^{n+1}\otimes\mb{Z}\cdot d\nu_I$,
	\item
		$K'=\bigoplus\limits^
		{I\subset \Atoms(\tilde{L})}_{\tilde{L}_I\neq\emptyset}
		t^{\codim \tilde{L}_I}\mb{Q}[t]\otimes\mb{Z}\cdot d\nu_I$,
	\item
		$K''=\bigoplus\limits^
		{I\subset \Atoms(\tilde{L})}_{\tilde{L}_I\neq\emptyset}
		t^{n+1}\mb{Q}[t] \otimes\mb{Z}\cdot d\nu_I$
\end{enumerate}
equipped with the Koszul differential.
We define a cdga structure on $K'$ by setting $d\nu_I\cdot d\nu_J=s(I,J)d\nu_{IJ}$
where $s(I,J)=0$ if $\tilde{L}_{IJ}=\emptyset$ and 
the sign of the shuffle $(I,J)$ in $IJ$ otherwise. The multiplication 
is well defined since
$\codim L_I+\codim L_J\ge \codim L_{IJ}$.

Clearly the natural projection $K'\ar K$ is a morphism of cdga's 
with the kernel $K''$.
\begin{lemma}\label{lemma:subspaces:acyclic}
	The complex $K''$ is acyclic.
\end{lemma}
\begin{proof}
	Note that for $I\subset \Atoms(L)$, we have $\tilde{L}_I=\emptyset$ 
	iff $\tilde{L}_{I\infty}=\emptyset$.
	Indeed, if $\tilde{L}_{I}\cap \tilde{L}_\infty=\emptyset$ and $\tilde{L}_I\neq\emptyset$,
	then $\tilde{L}_I$ is a compact subset of $\mb{C}^n$. Then $\tilde{L}_I=L_I$, but by our 
	assumption $L_I\neq\emptyset$ is at least one dimensional.

	We can write 
	$$K''=\left[\bigoplus\limits_{I\subset \Atoms(L)}d\nu_{I\infty}\ar[\iota_{\infty}] 
		\bigoplus\limits_{I\subset \Atoms(L)}d\nu_{I}\right]\cdot t^{n+1}\mb{Q}[t]$$
	as a bicomplex with two columns. 
	Since the horizontal differential $\iota_{\infty}$ is acyclic, the claim follows.
\end{proof}

%
Let $s\colon K'\to K'[-1]$ be multiplication by
the element
$t d\nu_\infty\in K'$.
Note that in general $sK'\not\subset tK'$, 
(e.g.\ $d\nu_\infty\not\in K'$).

\begin{lemma}
\begin{enumerate}
	\item  The commutator $[s,\partial]\colon K'\to K'$ 
		is multiplication by $t$.
	\item  The subcomplex $sK'+tK'\subset K'$ is an acyclic ideal. 
\end{enumerate}
\end{lemma}
\begin{proof}
	1. This is straighforward. 

	2. Clearly, $sK'+tK'$ is a subcomplex and an ideal of $K'$. 
	Let $sc_1+tc_2, c_1,c_2\in K'$ be a cocycle, i.e.\ 
	$$0=\partial(sc_1+tc_2)=-s\partial c_1+tc_1+t\partial c_2.$$
	Then, since $s^2=0$, we have $t(sc_1+s\partial c_2)=0$, and hence 
	$sc_1+s\partial c_2=0$.
	So $$\partial(sc_2)=-s\partial c_2+tc_2=sc_1+tc_2$$ is
	a coboundary.
\end{proof}
By the lemmas above, the rational homotopy type of $U$
is given by the cdga $K'/(sK'+tK')$.

\begin{lemma}
	There is an isomorphism of complexes:
	\begin{equation}\label{eq:subspaces:eq1}
		K'/(sK'+tK')\simeq 
		\bigoplus\limits^{I\subset \Atoms(L)}_{\tilde{L}_I\neq \emptyset}
			\mb{Q}\cdot t^{\codim \tilde{L}_I}\otimes \mb{Z}\cdot d\nu_I
		\oplus
		\bigoplus\limits^{I\subset\Atoms(L)}_{\tilde{L}_I\neq\emptyset,L_I=\emptyset}
			\mb{Q}\cdot t^{\codim\tilde{L}_I}\otimes\mb{Z}\cdot d\nu_{I\infty}.
	\end{equation}
\end{lemma}
\begin{proof}
	Assume $\tilde{L}_I\neq\emptyset$ for some $I\subset \Atoms(L)$.
	There are two cases. If $L_I=\emptyset$, then $\tilde{L}_I=\tilde{L}_{I\infty}$, and
	$$s(t^{\codim \tilde{L}_I}d\nu_I)=\pm t^{\codim \tilde{L}_{I\infty}+1}d\nu_{I\infty}.$$
	If $L_I\neq\emptyset$, then $\codim \tilde{L}_I=1+\codim \tilde{L}_{I\infty}$, and
	$$s(t^{\codim \tilde{L}_I}\cdot d\nu_I)=\pm t^{\codim\tilde{L}_{I\infty}}d\nu_{I\infty}.$$
	This immediately implies the required description for $K'/(sK'+tK')$ and also equips
	the RHS in \eqref{eq:subspaces:eq1} with an obvious cdga structure.
\end{proof}

To proceed let $\mc{Q}$ denote the locally cubical lattice corresponding to $L$, and let
$p\colon (\mc{Q},\le)\ar (L,\le)$ be the natural 
projection (Proposition \ref{prop:p_contraction_homomorphism}).
Since $(L,\le)$ is an intersection poset, 
the vertices of $\mc{Q}$ are given by $I\in \mc{Q}$, where
$I\subset \Atoms(L)$ are such that $L_I\neq\emptyset$.
The cdga $\Tot\big(p^*H_{L_\bullet}(\mb{C}^n)\otimes\OS(\mc{Q})\big)$ is given by 
$$\bigoplus\limits^{I\subset \Atoms(L)}_{L_I\neq \emptyset}t^{\codim L_I}\otimes \mb{Z}\cdot d\nu_I,$$
where the product of $t^{\codim L_I}\cdot d\nu_I$ with $t^{\codim L_J}\cdot d\nu_J$ 
is zero unless $\codim L_I+\codim L_J=\codim L_{IJ}$.

Finally, for all local sup-lattices $L$ we obtain the following:
\begin{prop}\label{prop:nonproper_model}
	There is a quasi-isomorphism of cdga's:
	$$\phi\colon K\ar \Tot\big(p^*H_{L_\bullet}(\mb{C}^n)\otimes\OS(\mc{Q})\big).$$
\end{prop}
\begin{proof}
	Note that $H^*_{L_\bullet}(\mb{C}^n)$ is a $H^*(\mb{P}^n)$-module where
	$t\in H^2(\mb{P}^n)$ acts by zero.
	The required homomorphism of algebras is given by 
	\begin{enumerate}
		\item $\phi(t^{\codim \tilde{L}_I}\cdot d\nu_I)=0$
			if $I\ni \infty$ or $L_I=\emptyset$, and
		\item $\phi(t^{\codim \tilde{L}_I}\cdot d\nu_I)=t^{\codim \tilde{L}_I}\cdot d\nu_I$
			otherwise.
	\end{enumerate}
	So $\phi$ factorizes through $\phi'\colon K'/(sK'+tK')\ar p^* H^*_{L_\bullet}(\mb{C}^n)*\OS(\mc{Q})$.
	In the notation of \eqref{eq:subspaces:eq1} we have:
	$$\ker \phi'=\bigoplus\limits^{I\in\mc{Q}}_{L_I=\emptyset}t^{\codim \tilde{L}_I}\cdot d\nu_I\oplus
	\bigoplus\limits^{I\in \mc{Q}}_{L_I=\emptyset}t^{\codim \tilde{L}_{I}}\cdot d\nu_{I\infty}.$$
	Similarly to the proof of Lemma \ref{lemma:subspaces:acyclic}, it is easy to see that $\ker\phi'$ is acyclic,
	and the claim follows.
\end{proof}
\subsection{Chromatic configuration spaces}
\label{sect:mv_totaro}
For a topological space $X$ and a finite (non-oriented) graph $G$ with the set of 
vertices $V(G)=\{1,\ldots, n\}$ we define the 
{\it chromatic} (or {\it graph}) {\it configuration space} $F(X;G)$ by setting
$$F(X;G):=X^n-\bigcup\limits_{\{a,b\}\in E(G)}\Delta_{\{ab\}}$$ 
where $E(G)$ is the set of edges of $G$ and
$\Delta_{\{ab\}}=\{(x_1,\ldots,x_n)\in X^n\mid x_a=x_b\}$.
We have a natural action of $Aut(G)$ on $F(X;G)$.
As an immediate application of Theorem \ref{thrm:ss_for_lattices} 
let us describe a spectral sequence converging to $H^*(F(X;G);\mb{Z})$.

Following part 2 of Example \ref{exmpl:os_algebras} we consider the intersection poset
$\Delta:=\Delta^G=(\Delta_*,\le)$ with $\Atoms(\Delta)$ given by 
$\Delta_{\{ab\}}$ for each $\{a,b\}\in E(G)$.
An element $x=\Delta_{S_1,\ldots,S_k}\in \Delta$ is an unordered partition
$\sqcup_{i\le k} S_i=V(G)$ such that each $S_i\subset G$ is a non-empty and connected subset.
If $T=\{e_1,\ldots,e_t\}
\subset E(G)$ is a collection of edges, 
then $\sup(T)\in \Delta$ corresponds to a partition of $V(G)$ given by the
vertices of the connected components of the graph $G_T:=(V(G_T)=V(G),E(G_T)=T)$. 

The rank function $r$ is defined by $r(\Delta_{S_1,\ldots,S_k})=\sum\limits_{i\le k} (|S_i|-1)$.
A set $\{\Delta_{e_1},\ldots,\Delta_{e_k}\}$ of atoms for $e_i\in E(G)=\Atoms(\Delta)$ is independent iff
the union of the edges $e_1,\ldots,e_k$ contains no cycles.

Recall that $\OS(\Delta)$ is the quotient $\Lambda\langle \Delta_{\{ab\}}\rangle/\mc{J}$ where $\deg(\Delta_{\{ab\}})=-1$ and $\mc{J}$ is the ideal generated by the elements
$$\partial(\Delta_{\{i_1 i_2\}}\wedge\ldots \wedge\Delta_{\{i_{k} i_1\}}),$$
for all cycles given by a sequence of edges $\{i_1,i_2\},\ldots,\{i_{k-1},i_k\},\{i_k,i_1\}\in E(G)$.

By Theorem \ref{thrm:ss_for_lattices}, 
for a paracompact, locally contractible topological space $X$ and 
for any coefficient ring $\Bbbk$ we have an $Aut(G)$-equivariant spectral sequence
$$E^{pq}_1=\bigoplus\limits^{x\in \Delta}_{r(x)=-p} 
    H^q_{\Delta_x}(X^n;\Bbbk)\otimes_{\mb{Z}} \OS(\Delta)\Rightarrow H^*(F(X;G);\Bbbk).$$
Now assume $\Bbbk$ is a field and $X$ is a $\Bbbk$-oriented topological manifold 
of dimension $m$, i.e.\ that an identification $H^m(X;\Bbbk)=\Bbbk$ is chosen.

The rank function satisfies $m\cdot r(x)=\codim \Delta_x:=\codim_{\mb{R}}\Delta_x/X^n$, and
elements $x,y\in \Delta$ are independent iff $\codim(\Delta_x\cap \Delta_y)=\codim(\Delta_x)+\codim(\Delta_y)$.

The aim of this subsection is to express the cdga $(E^{**}_1,d_1)$ in terms of generators and relations.
From now on we replace the set $E(G)$ of edges of $G$ by its ordered version, i.e.\ for each unordered pair $\{a,b\}$ in the old $E(G)$ we have two ordered ones $(a,b),(b,a)$ in the new one.\label{redefining_edges}

Let $p_a\colon X^n\to X$ denote the projection onto the $a$-th component.
Similarly we let $p_{ab}\colon X^n\to X^2$
to be the obvious projection for an ordered pair $(a,b)\in E(G)$.
Denote by $cyc(i_1,\ldots,i_k)$ the subgroup of $\Sigma_{V(G)}$ that cyclically permutes the arguments.
We say that a sequence $[i_1,\ldots,i_k]$ of vertices $i_t\in V(G)$ is a \textit{cycle} in $G$ if
$(i_1,i_2),\ldots, (i_{k-1},i_k),(i_k,i_1)\in E(G)$. A cycle is {\it simple} if it contains no proper subcycles.
For a permutation $\sigma$ 
we will write $(-1)^\sigma:=\mathop{\mathrm{sign}}(\sigma)$ and 
$(-1)^{k\cdot \sigma}=((-1)^{\sigma})^{k}=(-1)^{\sigma^k}$.

\begin{thrm}\label{thrm:totaro_ss}
    For a manifold $X$, $m=\dim_{\mb{R}} X$, graph $G$ and 
	the coefficient ring $\Bbbk$ as above,
    there is an $Aut(G)$-equivariant spectral sequence $E^{**}_1\Rightarrow H^*(F(X;G))$ such that
    the cdga $(E^{**}_1,d_1)$ is isomorphic to the following cdga over $H^*(X^n)$:
    $$(H^*(X^n)[\tilde{\Delta}_{ab}]/\mc{J},d),$$
    where
    $\deg\tilde{\Delta}_{ab}=m-1$ for each $(a,b)\in E(G)$ and $\mc{J}$ is the ideal generated by the following relations:
    \begin{enumerate}
        \item $\tilde{\Delta}_{ab}=(-1)^m\tilde{\Delta}_{ba}$ for each $(a,b)\in E(G)$;
        \item 
            $\sum\limits_{\sigma\in cyc(i_1,\ldots,i_k)}
                (-1)^{(m+1)\cdot \sigma}\sigma.(\tilde{\Delta}_{i_1,i_2}\cdot\ldots\cdot \tilde{\Delta}_{i_{k-1},i_k})=0$
                for each simple cycle $[i_1,\ldots,i_k]$ in $G$;
        \item $\tilde{\Delta}^2_{ab}=0$;
        \item $p^*_a(\gamma)\tilde{\Delta}_{ab}=p^*_b(\gamma)\tilde{\Delta}_{ab}$ for all $\gamma\in H^*(X)$.
    \end{enumerate}
    Under the isomorphism, the natural action of $\sigma\in Aut(G)\subset \Sigma_{V(G)}$ on $(E^{**}_1,d_1)$
    is uniquely determined by $\sigma.\tilde{\Delta}_{ab}=
    \tilde{\Delta}_{\sigma(a),\sigma(b)}$.
    
    The differential $d$ is zero on $H^*(X^n)$ and is given by 
    $d\tilde{\Delta}_{ab}=p^*_{ab}[\Delta]\in H^m(X^n)$ on $\tilde\Delta_{ab}$'s where $[\Delta]\in H^m(X^2)$ is the Poincar\'e dual class of the diagonal.
\end{thrm}
\begin{rmrk}
If $m$ is even, the second relation is equivalent to
$\partial(\tilde{\Delta}_{i_1 i_2}\wedge\ldots\wedge \tilde{\Delta}_{i_{k-1}i_k})=0$.
Which is not the case when $m$ is odd and $k$ is even. 
\end{rmrk}
This generalizes Totaro's spectral sequence described in \cite{Totaro} for 
the case when $G$ is the complete graph $K_n$.
In the course of the proof it will be clear that in case $G=K_n$
it is enough to consider only cycles of length $3$ in the generating set for $\mc{J}$, 
which corresponds to Arnold's relations.

For simplicity we omit the case when $X$ is a non-orientable manifold
and $\chr\Bbbk=2$, in which the above theorem is also true.

Let $\tau_x\in H^{\codim \Delta_x}_{\Delta_x}(X^n)$ denote the Thom class.
For $x,y\in \Delta$ independent we have $\tau_x\cup \tau_y=\pm \tau_{x\vee y}$.
For $(a,b)\in E(G)$ we set $\tau_{ab}=\tau_{\Delta_{\{ab\}}}$ if $a<b$
and $\tau_{ab}=(-1)^m\tau_{\Delta_{\{ab\}}}$ otherwise. 
So the natural action of $\sigma\in Aut(G)\subset \Sigma_{V(G)}$ 
on $\tau_{ab}\in H^*(X^n)$ is given by  
$\sigma.\tau_{ab}=\tau_{\sigma(a) \sigma(b)}$.
\begin{rmrk}
Note that by our convention the generators $\Delta_{\{ab\}}$ of $\OS(\Delta)$
are parametrized by \textit{unordered} pairs,
i.e.\ $\Delta_{\{ab\}}=\Delta_{\{ba\}}$ for $(a,b)\in E(G)$.
\end{rmrk}
Let $F$ be the free super-commutative $\Bbbk$-algebra on the generators $\tilde\Delta_{ab}, 
(a,b)\in E(G)$ of degree $m-1$ quotiented by $\tilde\Delta_{a b}=(-1)^m \tilde\Delta_{b a}$; 
we denote the resulting generators of $F$ again by $\tilde\Delta_{a b}$. Note that $F$ is again free.
The cdga $(E^{**}_1,d_1)$ is generated over $H^*(X^n)$ by the
elements $\tau_{ab}\otimes \Delta_{\{ab\}}$.
Sending $\tilde{\Delta}_{ab}$ to 
$\tau_{ab}\otimes \Delta_{\{ab\}}$ defines a surjection of 
$\Delta$-graded algebras:
$$H^*(X^n)\otimes_{\Bbbk}F\ar E^{**}_1.$$
The morphism is clearly $Aut(G)$-equivariant.
We want to show that its kernel is the ideal generated
by the remaining relations of Theorem \ref{thrm:totaro_ss}.

Let $\mc{T}\subset E^{**}_1$ denote the $\Bbbk$-submodule spanned by the elements of the form 
$\tau_x\otimes \OS(\Delta)_x$. Forgetting the differential, $\mc{T}$ 
becomes a $\Bbbk$-subalgebra. 

\begin{prop}\label{prop:even_os}
The natural morphism $f\colon F\to \mc{T}$
induced by $\tilde{\Delta}_{ab}\to \tau_{ab}\otimes \Delta_{\{ab\}}$
is surjective and the kernel is equal to the ideal $\mc{I}$ generated by the elements
\begin{enumerate}
    \item $\tilde{\Delta}^2_{ab}$,
    \item $u(i_1,\ldots,i_k):=\sum\limits_{\sigma\in cyc(i_1,\ldots,i_k)}(-1)^{(m+1)\cdot \sigma}
        \sigma.(\tilde{\Delta}_{i_1 i_2}\cdot\ldots\cdot \tilde{\Delta}_{i_{k-1}i_k})$,
            for all simple cycles $[i_1,\ldots, i_k]$ in $G$.
\end{enumerate}
\end{prop}
\begin{rmrk}
In the case $G=K_n$ one can check that this ideal is generated by the usual Arnold's relations, i.e.\ by all cycles of length $3$.
\end{rmrk}
Let us postpone the proof and deduce the main result first. 
Set $\Mu':=F/\mc{I}$.
\begin{proof}[Proof of Theorem \ref{thrm:totaro_ss}]
By Proposition \ref{prop:even_os}, $\Mu'=F/\ker{f}$. Note that $F$ and $\ker f$ are naturally $\Delta$-graded.
By the same proposition we have an isomorphism of $\Delta$-graded algebras
$$\bigoplus\limits_{x\in \Delta} \Mu'_x\simeq \bigoplus\limits_{x\in \Delta} \tau_x\otimes_{\mb{Z}} \OS(\Delta)_x=\mc{T}.$$
It extends to a surjective map of $\Delta$-graded algebras
$$\bigoplus\limits_{x\in \Delta}H^*(X^n)\otimes_{\mb{Z}} \Mu'_x
\ar
 \bigoplus\limits_{x\in \Delta}H^*_{\Delta_x}(X^n)\otimes_{\mb{Z}} \OS(\Delta)_x.$$
Pick an $x=\Delta_{S_1,\ldots,S_k}\in \Delta$.
It remains to show that the kernel of the last map is equal to the $H^*(X^n)$-submodule
$K_x\subset H^*(X^n)\otimes\Mu'_x$ generated by
\begin{equation}\label{eq:totaro_rel}
p^*_a(\gamma) \tilde{\Delta}_{e_1}\cdot\ldots\cdot \tilde{\Delta}_{e_k}-p^*_b(\gamma)\tilde{\Delta}_{e_1}\cdot\ldots \cdot \tilde{\Delta}_{e_k},
\end{equation}
for all collections $\{e_1,\ldots,e_k\}\subset \Atoms(\Delta)$ with $\sup(e_1,\ldots,e_k)=x$, 
$\gamma\in H^*(X)$ and $e_1=(a,b)$.
Clearly the last condition can be replaced by requiring that $a,b\in V(G)$ 
are connected by a path in $G$ going along some (unoriented) chain of $e_t$'s. Equivalently,
$K_x$ is equal to the $H^*(X^n)$-submodule spanned by
$\im[p^*_a-p^*_b]\otimes \Mu'_x$
for all $a,b\in S_i$ for some $i$. 

On the other hand,
using the K\"unneth formula 
$H^*(X^n)=H^*(X)^{\otimes_{\Bbbk} k}$
it is easy to see that there is a natural isomorphism of $H^*(X^n)$-modules
$$H^*_{\Delta_x}(X^n)\simeq H^*(X^n)\otimes_{\Bbbk} \tau_x/K'_x$$ 
where $K'_x$ 
is the $H^*(X^n)$-submodule spanned by 
$\im[p^*_a-p^*_b]\otimes \tau_x$ for all $a,b\in S_i$ for some $i$.

Clearly this implies $K_x=K'_x$ and hence we get a cdga isomorphism between $E^{**}_1$ and the algebra stated in the theorem. 
Since $d_1(\tau_{ab}\otimes \Delta_{\{ab\}})=p^*_{ab}[\Delta]$, the isomorphism is $Aut(G)$-equivariant.
\end{proof}

Now we prove Proposition \ref{prop:even_os}.
\begin{lemma}
For each cycle $[i_1,\ldots,i_k]$ in $G$ we have 
$u(i_1,\ldots,i_k)\in \ker f$.
\end{lemma}
\begin{proof}
Let $I=[i_1,\ldots,i_k]$ and $x=\sup((i_1,i_2),\ldots, (i_{k-1},i_k))$.
Then
\begin{equation}\label{eq:lemma:inker1}
f(u(I))=
(-1)^{m\cdot\binom{k-1}{2}}\sum\limits_{\sigma\in cyc(I)}(-1)^{(m+1)\cdot \sigma}
\sigma.(\tau_{i_1 i_2}\cup\ldots \cup \tau_{i_{k-1} i_k})\otimes \sigma.(\Delta_{\{i_1 i_2\}}\wedge\ldots\wedge 
\Delta_{\{i_{k-1} i_k\}})
\end{equation}
inside $\tau_{x}\otimes \OS(\Delta)_x$.
Assume $e_1=(i_1,i_2),\ldots,e_{k-1}=(i_{k-1},i_k)$ are independent atoms as otherwise 
each term of the sum vanishes.
Let $x=\sup(e_1,\ldots,e_k)\in \Delta$.
Note that $\tau_{i_1 i_2}\cup\ldots\cup \tau_{i_{k-1} i_k}=\pm\tau_{x}$. Since $\sigma$ acts trivially on $\Delta_x\subset X^n$, it is easy to see that $\sigma.\tau_x=(-1)^{m\cdot \sigma}\tau_x$. 
On the other hand we have
$$\sum\limits_{\sigma}(-1)^{\sigma}\sigma.(
\Delta_{\{i_1 i_2\}}\wedge\ldots \wedge \Delta_{\{i_{k-1} i_k\}})=
\partial(\Delta_{\{i_k i_1\}}\wedge \Delta_{\{i_1 i_2\}}\wedge\ldots\wedge \Delta_{\{i_{k-1} i_k\}}).$$
So 
$f(u(I))=\pm \tau_x\otimes \partial(\Delta_{\{i_k i_1\}}\wedge \Delta_{\{i_1 i_2\}}\wedge\ldots \wedge 
\Delta_{\{i_{k-1} i_k\}})$ vanishes in $\tau_x\otimes \OS(\Delta)_x$.
\end{proof}

\begin{proof}[Proof of Proposition \ref{prop:even_os}]
Since $\mc{I}$ is $\Delta$-graded, we have a surjective map of $\Delta$-graded algebras over $\mb{Z}$:
$$\bar{f}\colon \Mu'=F/\mc{I}\ar \mc{T}^{\mb{Z}}.$$
Note that if $m$ is even, after forgetting the $\mb{Z}$-grading, both sides 
become Orlik-Solomon algebras which are trivially isomorphic.
By the same argument, $\bar{f}\otimes \mb{F}_2$ is an isomorphism for all $d$.

Suppose $m$ is odd. Passing to the $\mb{Z}/2$-grading,
it is straightforward to check that $F/\mc{I}$ is the Orlik-Terao algebra (see \cite{OT}) corresponding to 
the complement of the vector hyperplane arrangement in $\mb{R}^{V(G)}$, which is the configuration space $F(\mb{R};G)$. Then \cite[Theorem 4.3]{OT} implies that 
$$\dim_{\Bbbk}(F/\mc{I})_x=\mu_{\Delta}([0,x])=\dim_{\Bbbk}(\OS(\Delta)_x).$$ So the surjectivity of
$(F/\mc{I})_x\to \tau_x\otimes \OS(\Delta)_x$ implies
that $F/\mc{I}\to \mc{T}$ is an isomorphism.

Let us explain in more detail how $F/\mc{I}$ is identified with the Orlik-Terao algebra. Let $x^a$ be a coordinate function on $\mb{R}^{V(G)}$.
Set $L_{ab}=x^a-x^b$.
A simple cycle $[i_1,\ldots,i_k]$ in $G$ defines
a dependent set
$e_1=(i_1,i_2),\ldots,e_{k-1}=(i_{k-1},i_k),e_k=(i_{k},i_1)\in \Atoms(\Delta)$. Such sets are precisely circuits in $\Delta$ (see \cite{OS} and \cite{OT}). 
Note that all linear dependencies
$\alpha_1\cdot  L_{i_1 i_2}+\ldots+\alpha_{k-1}\cdot L_{i_{k-1} i_k}+\alpha_k\cdot L_{i_k i_1}=0$ are proportional to 
$L_{i_1 i_2}+\ldots+L_{i_{k-1} i_k}+L_{i_k i_1}=0$.
It is easy to check that the relations in \cite[Proposition 2.3]{OT}
coincide with the relations stated in the proposition.
\end{proof}

Let $X=\mb{R}^m$. Then
$E^{**}_1\simeq \mc{T}$
vanishes outside of the diagonal 
$q=-mp$.
Hence $E^{pq}_t=E^{pq}_1$ for all $t\ge 1$, and we have an isomorphism 
of algebras $(H^*(F(X;G)),0)\simeq (E^{**}_1,d_1)$.
We get
\begin{cor}
In the case $X=\mb{R}^m$, 
the cdga $E^{**}_1$ from
Theorem \ref{thrm:totaro_ss}
is isomorphic to $H^*(F(X;G))$.
\end{cor}

\subsection{Generalization of Kriz-Totaro model}
\label{sect:app:conf_homotopy_type}
Let $M$ be a smooth proper algebraic variety over $\mb{C}$ and consider the 
chromatic configuration space $F(M,G)$ for a graph $G$. 
(Recall that $E(G)\subset V(G)\times V(G)$ is a symmetric subset, see p.~\pageref{redefining_edges}.) 
As in Section \ref{sect:mv_totaro},
let $\Delta$ be the arrangement of all diagonals $\Delta_*\subset M^{V(G)}$. 
As we saw, Theorem \ref{thrm:totaro_ss} 
describes the Mayer-Vietoris spectral sequence 
${}_{W'}E^{**}_1\Rightarrow H^*(F(M;G);\mb{Q})$, and 
Theorem \ref{thrm:main_lattices} implies that ${}_{W'}E^{**}_1$ describes 
the $\mb{Q}$-homotopy type of $F(M;G)$. The description is $Aut(G)$-equivariant in the following sense:
there is a zig-zag of $Aut(G)$-equivariant quasi-isomorphisms
${}_{W'}E_1\zgqis A_{PL}(F(M;G))$.

As a direct corollary, remembering that $\dim_{\mb{R}}M$ is even, we obtain the chromatic version of the Kriz-Totaro model:
\begin{thrm}\label{thrm:kriz_model}
	The rational homotopy type of $F(M,G)$ 
        has a model equal to the cdga over $H^*(M^{V(G)})$ given by the quotient 
	$$H^*(M^{V(G)})\otimes 
	\Lambda^*\langle \tilde{\Delta}_{ab}\rangle/\mc{J}$$
        where $\langle \tilde{\Delta}_{ab}\rangle$ is the $\mathbb{Q}$-vector space spanned by generators
	$\tilde{\Delta}_{ab}$ of degree 
 $2\dim_{\mb{C}} M-1$ for all \emph{ordered} pairs $(a,b)\in E(G)$,
	and $\mc{J}$ is the ideal generated by the following relations
	\begin{enumerate}
    \item $\tilde{\Delta}_{ab}=\tilde{\Delta}_{ba}$;
    
		\item $\partial(
		\tilde{\Delta}_{i_1i_2}\wedge\ldots\wedge \tilde{\Delta}_{i_{k-1}i_k})=0$
		for each simple cycle $(i_1,i_2),\ldots,(i_{k-1},i_k),(i_k,i_1)\in E(G)$ in $G$;

		\item $p^*_a(\gamma)\cdot \tilde{\Delta}_{ab}=
			p^*_b(\gamma)\cdot \tilde{\Delta}_{ab}$ for all $\gamma\in H^*(M)$ and $(a,b)\in E(G)$.
	\end{enumerate}

	The differential $d$ is zero on $H^*(M^{V(G)})$, and we have
	$d\tilde{\Delta}_{ab}=[\Delta_{ab}]\in H^{2\dim M}(M^{V(G)})$ where $[\Delta_{ab}]$ 
	is the Poincar\'e dual class of the diagonal $[\Delta_{ab}]$.

    We define an action of $Aut(G)$ on our model by setting 
	$\sigma.\tilde{\Delta}_{ab}=\tilde{\Delta}_{\sigma(a)\sigma(b)}, \sigma\in Aut(G)$ and using the permutation action on $M^{V(G)}$. With this definition, the model is $Aut(G)$-equivariant in the sense described above.
\end{thrm}\QEDB

\newpage
\begin{appendices}

\section{Technical lemma}
\label{sect:tech_lemma}
Here we give a proof of Lemma \ref{lemma:mobius_inversion_multiplication}.
\begin{lemma}\label{lemma:mob_inversion_proof}
	The product $\hat{m}$
	provides
$\hat{A}^{\bar{N}}$ with a $\mc{Q}|_{[0,\bar{N}]}$-chain algebra structure.
\end{lemma}
\begin{proof}
We need to check that 
$\hat{m}$
is associative, commutative and satisfies the Leibniz rule. In order to do this it suffices show that the complex
$\bigoplus\limits_{I\subset K}^{I,K\subset N}\mb{Z}\cdot (d\mu_I)^K$, obtained from $\hat{A}^N$ by replacing all $A^I$'s by $\mb{Z}$,
with the multiplication $\hat{m}$ and differential $\delta$ defined as in \eqref{eq:dif_cech} 
is a commutative associative dg-algebra.

1. Commutativity is obvious. For associativity let
$$m^2\colon 
\mb{Z}\cdot (d\mu_I)^K\otimes \mb{Z}\cdot (d\mu_{I'})^{K'}
\otimes \mb{Z}\cdot (d\mu_{I''})^{K''}\ar 
\mb{Z}\cdot (d\mu_I d\mu_{I'} d\mu_{I''})^{KK'K''}$$
be the morphism equal to the product of monomials 
in the case $$K\sm I\scup K'\sm I'\scup K''\sm I''=
KK'K''\sm II'I''$$ and to zero otherwise.
The former is the case iff 
both equalities $$
K\sm I\scup K'\sm I'=KK'\sm II',
$$
and
$$
KK'\sm II'\scup K''\sm I''=KK'K''\sm II'I''
$$
hold.
So $m^2=m\circ(m\otimes \id)$. Similarly $m^2=m\circ(\id\otimes m)$, i.e.
$m$ is associative.

2. Let us prove the Leibniz rule. We need to check that the following diagram commutes:
\begin{equation*}
\begin{tikzcd}
	\mb{Z}\cdot (d\mu_I)^K\otimes 
	\mb{Z}\cdot (d\mu_{I'})^{K'}\arrow[d,"\delta"]
	\arrow[r,"m"] & \mb{Z}\cdot (d\mu_I d\mu_{I'})^{KK'}\arrow[d,"\delta"]\\
	\substack{\mb{Z}\cdot \sum_p\big(\overbrace{(\iota_p d\mu_I)^K\otimes (d\mu_{I'})^{K'}}^{\alpha_p}
+\overbrace{(-1)^{I}(d\mu_I)^K\otimes (\iota_{p}d\mu_{I'})^{K'}}^{\alpha'_p}
\\
-\underbrace{(\iota_p d\mu_I)^{K-p}\otimes (d\mu_{I'})^{K'}}_{\beta_p}
-\underbrace{(-1)^I(d\mu_I)^{K}\otimes (\iota_{p}d\mu_{I'})^{K'-p}}_{\beta_p}
\big)}
		\arrow[r,dashed,"m"] & 
\substack{\mb{Z}\cdot \sum_{p}\big( 
\overbrace{(\iota_p (d\mu_I d\mu_{I'}))^{KK'}}^{\bar{\alpha}_p}
\\-\underbrace{(\iota_p(d\mu_I d\mu_{I'}))^{KK'-p})}_{\bar{\beta}_p}\big)}.
\end{tikzcd}
\end{equation*}
Let $A$ and $B$ be the images of 
$(d\mu_I)^K\otimes (d\mu_{I'})^{K'}$ under the top and bottom paths 
respectively.
So 
$A=m(\sum_p \alpha_p+\alpha'_p-\beta_p-\beta'_p)$ and $B=\sum_p \bar{\alpha}_p-\bar{\beta}_p$.

If $|I\cap I'|>1$, then $A=0=B$.
Assume $I\cap I'=\{p\}$. Then $B=0$.
On the other hand,
$m(\alpha_q+\alpha'_q-\beta_q-\beta'_q)=0$ for all $q\neq p$.
From 
$$p\in K\sm(I-p)\scup K'\sm I'\supsetneq KK'\sm (I-p\scup I')\not\ni p$$
we get $m(\alpha_p)=0$. By symmetry $m(\alpha'_p)=0$ as well.
We have
$$(K-p)\sm (I-p)\scup K'\sm I'=K\sm I\scup K'\sm I' \mbox{ }\supseteq\mbox{ } 
(K-p\scup K')\sm (I-p\scup I')=KK'\sm II'.$$
In the case of strict inclusion we obtain $m(\beta_p)=0$ and similarly
$m(\beta'_p)=0$.
Otherwise $$K\sm I\scup K'\sm I'=KK'\sm II',$$ and 
since $(K-p)\scup K'=K\scup (K'-p)=KK'$,
we obtain
$m(\beta_p+\beta'_p)=(\iota_p(d\mu_I d\mu_{I'}))^{KK'}=0$.
Thus in any case we have $A=0=B$.

\bigskip

Finally, assume $I\cap I'=\emptyset$.
Consider the case $K\sm I\scup K'\sm I'=KK'\sm II'$.
For $p\in I$, we have $p\not\in K'$ and
$$K\sm I-p\scup K'\sm I'=p\sqcup (K\sm I\scup K'\sm I')=
	p\ssqcup KK'\sm II'=KK'\sm (I-p\scup I'),$$
hence $m(\alpha_p)=((\iota_p d\mu_I) d\mu_{I'})^{KK'}$.
By symmetry we have $m(\sum_p \alpha_p+\alpha'_p)=
\sum_p (\iota_p (d\mu_I d\mu_{I'}))^{KK'}$.
Similarly we have $$K-p\sm I-p\scup K'\sm I'=K\sm I\scup K'\sm I'=KK'\sm II'
\mbox{ } =\mbox{ } 
(K-p\scup K')\sm(I-p\scup I'),$$
where the last equality follows from $p\not\in K'$.
Using $K-p\scup K'=KK'-p$ and arguing as above we get
$m(\beta_p+\beta'_p)=(\iota_p(d\mu_I d\mu_{I'}))^{KK'-p}$.
Thus $A=\sum_p m(\alpha_p+\alpha'_p-\beta_p-\beta'_p)=B$.

\smallskip

It remains to cover the case $K\sm I\scup K'\sm I'\supsetneq KK'\sm II'$ and
$I\cap I'=\emptyset$. Clearly $B=0$.
Pick $p\in I$. If $p\in K'$, then the inclusion
$$K\sm I-p\scup K'\sm I'=K\sm I\scup K'\sm I'\mbox{ }\supset\mbox{ } 
	p\ssqcup (KK'\sm II')=KK'\sm (I-p\scup I')$$
is strict iff so is
$$K-p\sm I-p\scup K'\sm I'=K\sm I\scup K'\sm I'\mbox{ }\supset\mbox{ } p\ssqcup (KK'\sm II')=
	(K-p\scup K')\sm (I-p\scup I').$$
In particular $m(\alpha_p-\beta_p)=
(\iota_p d\mu_I d\mu_{I'})^{KK'}-
	(\iota_p d\mu_I d\mu_{I'})^{(K-p) \cup K'}=0$	since $K-p\scup K'=KK'$.

If $p\not\in K'$, then using
$$K\sm I-p\scup K'\sm I'=p\sqcup (K\sm I\scup K'\sm I')
\mbox{ }\supsetneq\mbox{ } 
	p\sqcup (KK'\sm II')=KK'\sm (I-p\ssqcup I')$$
we get $m(\alpha_p)=0$, while
$$K-p\sm I-p\scup K'\sm I'=K\sm I\scup K'\sm I'\mbox{ }\supsetneq\mbox{ } KK'\sm II'=
(K-p\scup K')\sm (I-p\ssqcup I')$$
implies $m(\beta_p)=0$.
Similarly we obtain $m(\alpha'_p)=m(\beta'_p)=0$.
Thus $A=0=B$.

This completes the proof that $\hat{A}^N$ is a monoid in $\ChF$.
The component $\hat{m}_{I\subset K;I'\subset K'}$
takes values in $W'_{|KK'|}\hat{A}^N$.
Hence $\Gr^{W'}\hat{m}$ is well defined and it is easy to see that
it is equal to the associated graded of 
the convolution multiplication \eqref{eq:mobius_inversion}.
\end{proof}

\section{Cup product formula for clean intersections}\label{cup_product}
In this section by a (sub)manifold we understand an oriented smooth (sub)manifold.
Let $X$ be a manifold.
Our main model
is described in terms $H^*_W(X)=H^*(X,X-W)$, where $W\subset X$ is
a submanifold. Let $c(W)=\codim(W/X)$.
It is convenient to have a description 
of the model
in terms of the natural Thom isomorphism
$\Th_W\colon H^*_W(X)\simeq H^{*-c(W)}(W)$.
If $i\colon W_1\ar W_2$ is an inclusion,
then the induced morphism $i_*\colon H^*_{W_1}(X)\ar H^*_{W_2}(X)$ corresponds to
the Gysin map $i_!\colon H^{*-c(W_1)}(W)\ar H^{*-c(W_2)}(W)$.

Now we shall describe the multiplication. For simplicity we will assume that $W_1,W_2\subset X$ 
are submanifolds with clean intersection $W_{12}=W_1\cap W_2$, i.e.\ the intersection is locally modeled by an intersection of 
vector subspaces.
Note that $\Th_W$ is naturally an isomorphism of $H^*(W)$-modules. 
The cup product 
$$H^*_{W_1}(X)\otimes H^*_{W_2}(X)\ar H^{*}_{W_1\cap W_2}(X)\ar H^{p+q}_{W_{12}}(X)$$
is a morphism of $H^*(W_1)\otimes H^*(W_2)$-modules in the obvious way, and it is compatible 
with the Thom isomorphisms.
Denote the tangent bundle of $W$ by $TW$.
The quotient $$D:=TX|_{W_{12}}/(TW_1|_{W_{12}}+TW_2|_{W_{12}})$$ 
is called the \emph{excess bundle} (or \emph{defect bundle}) 
over the intersection $W_{12}$. It has rank $c(W_1)+c(W_2)-c(W_{12})$.
Manifolds $W_1,W_2$ are transversal in $X$ if and only if 
the excess bundle is zero.
We leave a proof of the following proposition to the reader.
\begin{prop}\label{prop:euler}
	Under the Thom isomorphism the cup product is identified with the unique morphism
	$$m\colon H^*(W_1)\otimes H^*(W_2)\ar H^{*+c(W_1)+c(W_2)-c(W_{12})}(W_{12})$$
	of $H^*(W_1)\otimes H^*(W_2)$-modules
	such that 
	$$m(1,1)=e(D)\in H^{c(W_1)+c(W_2)-c(W_{12})}(W_{12})$$
	where $e(D)$ is the Euler class of $D$.
\end{prop}

\section{Computational examples}\label{first_example}
In this section, in order to illustrate the main theorem \ref{thrm:main_lattices} as a computational device, 
we consider a couple of examples.
\begin{exmpl}\label{exmpl:thrm:for_lines}
	Consider the complement $U=\mb{P}^2-\cup_i L_i$ of a collection of 
	distinct lines $\{L_i\subset \mb{P}^2|i\leq n\}$
	passing through $p\in \mb{P}^2$.
The corresponding intersection poset $(L,\leq)$ is a geometric lattice 
with $n$ atoms corresponding to $L_i$ and one
maximal element corresponding to $p=\cap_i L_i$. Then $\Mu=\OS(L)$ is the 
cdga with $\Mu(0)=\mb{Q}$, 
$\Mu(1):=\bigoplus_i \Mu_{L_i}=\mb{Q}\langle \nu_i\rangle$, the 
$\mb{Q}$-vector space spanned by $\nu_i,i=1,\ldots, n$, and
$\Mu(2):=\Mu_{p}=\ker[\Sigma\colon \mb{Q}\langle e_i\rangle\to \mb{Q}]$. Here $\nu_i$ and $e_i$ are formal variables of degree $-1$ and $-2$ respectively. The multiplication $-\cdot -$ is determined by $\nu_i\cdot \nu_j=e_i-e_j\in \Mu(2)$, while the differential is given by $\partial(e_i)=\nu_i\in \Mu(1),\partial(\nu_i)=1\in \Mu(0)$.

The second quadrant spectral sequence $E^{pq}_1$ of the main theorem \ref{thrm:main_lattices} 
takes the following form.
\begin{equation*}
\begin{tabular}{c c c|r}
$H^0(p)\otimes \Mu(2)[-2]$  &  $H^2(\mb{P}^1)\otimes \Mu(1)[-1]$  &  $H^4(\mb{P}^2)\otimes \Mu(0)$   & $4$  \\
$0$   				  &  $0$  			      &  $0$  		     		   & $3$  \\
$0$   				  &  $H^0(\mb{P}^1)\otimes \Mu(1)[-1]$  &  $H^2(\mb{P}^2)\otimes \Mu(0)$   & $2$  \\
$0$   				  &  $0$  			      &  $0$  		     		   & $1$  \\
$0$   				  &  $0$  			      &  $H^0(\mb{P}^2)\otimes \Mu(0)$   & $0$  \\
\hline
$-2$  				  & $-1$  			      &  0  		& 
    \diagbox{$p$}{$q$}\\
\end{tabular}
\end{equation*}

The differential $d_1$ has degree $(1,0)$ and can be written as $d_1=g_!\otimes \partial$, where $g_!$ is the pushforward with compact support of 
the corresponding inclusion. The multiplicative structure is induced by the product in $H^*(\mb{P}^2)\otimes\Mu$. 
It is immediate to see that that the top row is acyclic and $\dim_{\mb{Q}} E^{-1,2}_2=n-1, \dim_{\mb{Q}} E^{0,0}=1$, while
the other entries in $E^{pq}_2$ vanish, hence $E^{**}_2\simeq \mb{Q}[0]\oplus \mb{Q}^{\oplus (n-1)}[-1]=:A$. 
Thus we get a natural quasi-isomorphism of complexes $E^{pq}_2\to E^{pq}_1$. 
This quasi-isomorphism is multiplicative since 
by definition 
$$(\nu_i-\nu_{i'})\cdot (\nu_j-\nu_{j'})=(e_i-e_j)-(e_i-e_{j'})-(e_{i'}-e_j)+(e_{i'}-e_{j'})=0$$ 
So $H^*(U;\mb{Q})\simeq A$, and $U=\mb{P}^2-\cup_i L_i$ is formal.

This is what we expect geometrically: under the identification 
$\mb{P}^2-\{p\}\simeq \Tot(\mc{O}_{\mb{P}^1}(1))$, 
the lines $L_i-\{p\}$ correspond to the fibers, thus $$\mb{P}^2-\cup_i L_i\sim \mb{P}^1-\cup_i p_i\sim \bigvee^{n-1} S^1$$ 
where $\{p_i\}$ is a set of $n$ distinct points.
\end{exmpl}
\begin{exmpl}\label{exmpl:thrm_for_quadric}
Let $Q\subset \mb{P}^3$ be a smooth quadric and $L_\tau\subset \mb{P}^3$ be a line such that
$L_\tau\cap Q$ consists of exactly one point $p_\tau$. Let $L_t\subset \mb{P}^3$ be a line transversal to $Q$ such that 
$L_\tau\cap L_t\cap Q=\{p_\tau\}$.
Then $E^{pq}_1$ takes this form:
\begin{equation*}
\resizebox{\columnwidth}{!}{%
\begin{tabular}{c c c c |r}
$H^0(L_\tau\cap L_t\cap Q)$ & $H^0(L_t\cap Q)\oplus H^0(L_\tau\cap Q)\oplus H^0(L_\tau\cap L_t)$ & 
	$H^2(L_t)\oplus H^2(L_\tau)\oplus H^4(Q)$ & $H^6(\mb{P}^3)$ & $6$\\
	$0$		  &		$0$& $0$                                       & $0$                & $5$\\
	$0$	  &                     $0$& $H^0(L_\tau)\oplus H^0(L_t)\oplus H^2(Q)$ & $H^4(\mb{P}^3)$ & $4$\\
	$0$	  &			$0$&                $0$                        &    $0$              & $3$\\
	$0$	  &                     $0$ &       $H^0(Q)$                            & $H^2(\mb{P}^3)$ & $2$\\
	$0$	  &                                     $0$ &                $0$                        & $0$                 & $1$\\
	$0$	  &                                     $0$ &                $0$                        & $H^0(\mb{P}^3)$ & $0$\\
	\hline
	$-3$	  &			 	      $-2$ &		$-1$	                   &	$0$ &
    \diagbox{$p$}{$q$}\\
\end{tabular}%
}
\end{equation*}

To describe the differential and the multiplication consider the
Grassmann algebra $\Lambda^* V$ generated by a vector space
$V$ with a basis $d\nu_1,d\nu_2,d\nu_3$ corresponding to 
$Z_1=L_t,Z_2=L_\tau,Z_3=Q\subset \mb{P}^3$.
Thus $\Lambda^* V$ is spanned by the monomials 
$d\nu_I=d\nu_{i_1}\wedge\ldots \wedge d\nu_{i_{|I|}}$, 
$I=\{i_1<\ldots<i_{|I|}\}$ of degree $-|I|$.
For each $I\subset \{1,2,3\}$, put $p=|I|$ and
$Z_I=\cap_{i\in I}Z_i\subset \mb{P}^3$.
We view the term 
$H^{q-2\codim_{\mb{C}}Z_I}(Z_I)\simeq H^q_{Z_I}(\mb{P}^3)
\subset E^{pq}_1$ as 
$H^q_{Z_I}(\mb{P}^3)\cdot d\nu_I\subset E^{**}_1$.
Then $$d_1=\sum\limits^{i,I}_{i\subset I\subset \{1,2,3\}}
g_{I-\{i\},I}\otimes \iota_i,$$ where $g_{J,I}\colon 
H^*_{Z_I}(\mb{P}^3)\to H^*_{Z_J}(\mb{P}^3)$ is induced by $Z_I\subset Z_J$ and $\iota_i$ is the substitution of the dual of $d\nu_i$ in the Grassmann monomials.
The multiplication $$H^*_A(\mb{P}^3)\otimes H^*_B(\mb{P}^3)\to H^*_{A\cap B}(\mb{P}^3)$$ together with the product in the Grassmann algebra induces a cdga structure on $(E^{**}_1,d_1)$. Note that $H^*_{Z_I}(\mb{P}^3)\to H^*(\mb{P}^3)$ is injective for each $I$.

Let us show that $E^{pq}_1$ is almost acyclic
and deduce from this the formality of $\mb{P}^3-(L_\tau\cup L_t\cup Q)$.

Clearly the complex $E^{*,2}_1$ is acyclic and the cohomology of
$E^{*,4}_1$ is concentrated in one degree.
The top row $E^{*,6}_1$ is the following complex:
$$\mb{Q}\hookrightarrow \mb{Q}^4\os{f}{\to} 
	\mb{Q}^3\twoheadrightarrow \mb{Q}.$$
A simple calculation shows that $f$ has rank $2$. 
Namely, let $a\in H^0(L_\tau\cap Q), b_t,b_\tau\in H^0(L_t\cap Q)$ and $c\in H^0(L_t\cap L_\tau)$ correspond to $p_\tau, b_t,b_\tau$ and $p_\tau$ respectively. 
%
Consider $$f\colon H^0(L_\tau\cap Q)\oplus H^0(L_t\cap Q)\oplus H^0(L_t\cap L_\tau)
	\to H^2(L_t)\oplus H^2(L_\tau)\oplus H^4(Q).$$
Choosing the natural basis of the RHS allows us to write a matrix of $f$:
$\begin{pmatrix}
	0 &-1& 0&-1\\
	-1&0 &-1& 1\\
	1 &1 & 1& 0\\
\end{pmatrix}$.
Clearly, $\rk f=2$, hence $E^{-2,6}_2\simeq \mb{Q}$ is 
the only non-trivial entry for $q=6$.

Thus $E^{pq}_2$ has the form:
\begin{equation*}
\begin{tabular}{c c c| r}
	$\mb{Q}$ & $0$		& $0$         & $6$\\	
	$0$   & $0$		& $0$	    & $5$	   \\			
	$0$   & $\mb{Q}^3$& $0$         & $4$    \\
	$0$   & 	  $0$     & $0$         & $3$    \\
	$0$   & $0$		& $0$	    & $2$    \\		
	$0$   & $0$		& $0$	    & $1$	   \\
	$0$   & $0$		& $\mb{Q}$& $0$    \\
	\hline
	$-2$   & $-1$		& $0$	    &
    \diagbox{$p$}{$q$}\\
\end{tabular}
\end{equation*}

By degree counting it is clear that there is a quasi-isomorphism 
$E^{**}_2\to E^{**}_1$ which is multiplicative, so 
$U=\mb{P}^3-(L_\tau\cup L_t\cup Q)$
is formal and its cohomology is given by $E^{pq}_2$.

\end{exmpl}
\
\end{appendices}

\newpage
\printbibliography[title=Bibliography]
\addcontentsline{toc}{section}
{Bibliography}
\bigskip
\noindent

Faculty of Mathematics, Higher School of Economics, 6 Usacheva ulitsa, Moscow, Russia 119048; 

\smallskip 
%
%
Chennai Mathematical Institute, H1, SIPCOT IT Park, Siruseri, Kelambakkam, India, 
603103;
\smallskip
\noindent

email: xaxa3217 ( •\_•) gmail com, azakharov ( •\_•) cmi ac in
\end{document}